\documentclass[leqno,11pt]{amsart}
\usepackage{amssymb, amsmath}
\usepackage[latin1]{inputenc}
\usepackage{geometry}
\usepackage[T1]{fontenc}
\usepackage{amsmath,amssymb,amsfonts}
\usepackage{graphicx,xcolor}
\usepackage{hyperref}

\usepackage[official]{eurosym}
\usepackage{graphics}

\def\refer#1{~\ref{#1}}
\def\refeq#1{~(\ref{#1})}

\def\inte#1{
\displaystyle\mathop{#1\kern0pt}^\circ }

\def\sumetage#1#2{\sum_{\substack{{#1}\\{#2}}}}

\let\al=\alpha

\let\lam=\lambda

\let\f=\phi

\let\wt=\widetilde

\def\cA{{\mathcal A}}

\def\cC{{\mathcal C}}
\def\cD{{\mathcal D}}
\def\cE{{\mathcal E}}
\def\cF{{\mathcal F}}

\def\cG{{\mathcal G}}
\def\cH{{\mathcal H}}

\def\cL{{\mathcal L}}

\def\cO{{\mathcal O}}
\def\cP{{\mathcal P}}

\def\cR{{\mathcal R}}
\def\cS{{\mathcal S}}
\def\cT{{\mathcal T}}

\def\cW{{\mathcal W}}

\def\cY{{\mathcal Y}}

\def\virgp{\raise 2pt\hbox{,}}
\def\cdotpv{\raise 2pt\hbox{;}}

\def\sgn{\mathop{\rm sgn}\nolimits}

\def\C{\mathop{\mathbb C\kern 0pt}\nolimits}
\def\DD{\mathop{\mathbb D\kern 0pt}\nolimits}
\def\EE{\mathop{{\mathbb E \kern 0pt}}\nolimits}
\def\K{\mathop{\mathbb K\kern 0pt}\nolimits}
\def\N{\mathop{\mathbb N\kern 0pt}\nolimits}
\def\Q{\mathop{\mathbb Q\kern 0pt}\nolimits}
\def\R{{\mathop{\mathbb R\kern 0pt}\nolimits}}
\def\SS{\mathop{\mathbb S\kern 0pt}\nolimits}
\def\ZZ{\mathop{\mathbb Z\kern 0pt}\nolimits}
\def\TT{\mathop{\mathbb T\kern 0pt}\nolimits}
\def\P{\mathop{\mathbb P\kern 0pt}\nolimits}
\def \H{{\mathop {\mathbb H\kern 0pt}\nolimits}}
\newcommand{\ds}{\displaystyle}

\newcommand{\beq}{\begin{equation}}
\newcommand{\eeq}{\end{equation}}
\newcommand{\ben}{\begin{eqnarray}}
\newcommand{\een}{\end{eqnarray}}
\newcommand{\beno}{\begin{eqnarray*}}
\newcommand{\eeno}{\end{eqnarray*}}
\newcommand{\bqs}{\begin{equation*}}
\newcommand{\eqs}{\end{equation*}}
\newcommand{\andf}{\quad\hbox{and}\quad}
\newcommand{\with}{\quad\hbox{with}\quad}
\renewcommand{\O}{\mathcal{O}}

\def\equivH#1 {\buildrel\hbox{\tiny {$#1$}}\over \equiv}
\def\simH#1 {\buildrel\hbox{\footnotesize {$#1$}}\over \sim}

\newtheorem{definition}{Definition}[section]
\newtheorem{theorem}{Theorem}[section]
\newtheorem{lemma}{Lemma}[section]
\newtheorem{remark}{Remark}[section]
\newtheorem{cor}{Corollary}[section]
\newtheorem{proposition}{Proposition}[section]
\numberwithin{equation}{section}
\begin{document}
\title[Blow up dynamics for surfaces asymptotic to Simons cone]
{Blow up dynamics for the hyperbolic  vanishing mean curvature flow
of surfaces asymptotic to Simons cone}
\author[H. Bahouri]{Hajer Bahouri}
\address[H. Bahouri]%
{Laboratoire d'Analyse et de Math{\'e}matiques Appliqu{\'e}es UMR 8050 \\
Universit\'e Paris-Est  Cr{\'e}teil \\
61, avenue du G{\'e}n{\'e}ral de Gaulle\\
94010 Cr{\'e}teil Cedex, France}
\email{hajer.bahouri@math.cnrs.fr}
\author[A. Marachli]{Alaa Marachli}
\address[A. Marachli]%
{Laboratoire d'Analyse et de Math{\'e}matiques Appliqu{\'e}es UMR 8050 \\
Universit\'e Paris-Est  Cr{\'e}teil \\
61, avenue du G{\'e}n{\'e}ral de Gaulle\\
94010 Cr{\'e}teil Cedex, France}
\email{alaa.marachli@u-pec.fr}
\author[G. Perelman]{Galina Perelman}
\address[G. Perelman]%
{Laboratoire d'Analyse et de Math{\'e}matiques Appliqu{\'e}es UMR 8050 \\
Universit\'e Paris-Est  Cr{\'e}teil\\
61, avenue du G{\'e}n{\'e}ral de Gaulle\\
94010 Cr{\'e}teil Cedex, France}
\email{galina.perelman@u-pec.fr}

\date{\today}

\begin{abstract} 
In this article, we establish  the existence of a family of  hypersurfaces  $(\Gamma (t))_{0< t \leq T}$  which evolve by the vanishing mean curvature flow in Minkowski space and which as $t$ tends to~$0$ blow up towards a hypersurface  which behaves like the Simons cone at infinity. This issue amounts to investigate the singularity formation for a  second order   quasilinear wave equation. Our constructive approach consists in proving the existence of finite time blow up solutions of this hyperbolic equation under the form $u(t,x) \sim  t^ {\nu+1} Q\Big(\frac {x} {t^ {\nu+1}} \Big) $,  where~$Q$ is a stationary solution and $\nu$    an arbitrary large positive   irrational number. Our approach roughly follows that of Krieger, Schlag and Tataru in \cite{KST, KST1, KST2}. However contrary to these works, the equation to be handled  in this article is quasilinear.  This induces a number of difficulties to face.

\end{abstract}

\maketitle

\tableofcontents

\noindent {\sl Keywords:}  Vanishing mean curvature flow, quasilinear wave equation, Simons cones,   Blow up dynamics.

\vskip 0.2cm

\noindent {\sl AMS Subject Classification (2000):} 43A30, 43A80.

\section{Introduction}\label {intro}
\subsection{Setting of the problem}
In this article we address the question of singularity formation for the hyperbolic  vanishing mean curvature flow
of surfaces that are asymptotic to Simons cones at infinity.

\medskip

In \cite{BGG},  Bombieri, De Giorgi and Giusti proved that the Simons cone defined as follows 
\beq
\label {Sim}  
 \ds C_n=\big\{ X=(x_1,\cdots,x_{ 2 n })\in \R^ { 2 n }\,, x^2_1+ \cdots + x^2_n= x^2_{n +1 }+ \cdots + x^2_{ 2 n } \big\}, \eeq
is a globally minimizing surface if and only if  $n \geq 4$.      
 It is clear that  the Simons cone which has dimension~$d=2n-1$ is invariant under the action of the group~$O(n)\times O(n) $, where $O(n)$ is the orthogonal group of $\R^ {  n }$, and that can be parametrized in the following way:
 \begin{eqnarray} 
\nonumber \R_+\times \SS^{n-1} \times \SS^{n-1} &\longrightarrow& C_n\subset \R^{2n}\\
 \label {paramsimons} (\rho,\omega_1,\omega_2)&\longrightarrow& (\rho \omega_1, \rho \omega_2)\, \cdot
\end{eqnarray}
The Simons cones are linked to Bernstein's problem which states as follows: if the graph of a $C^2$ function $u$ on $\R^ {m-1}$ is a minimal surface in $\R^ {m}$, does this imply that this graph  is an hyperplane? Such issue amounts to ask if the solution $u$ to the following equation known as the minimal surface equation $$  \sum^{m-1}_{j=1}\partial_{x_j} \Big( \frac {u_{x_j}} {\sqrt{1+|\nabla u|^2}} \Big) =0 \, ,$$
 is linear. This  problem  due to  Sergei Natanovich Bernstein who solved the case $m= 3$ in $1914$ admits only an affirmative answer in the case of dimension $m \leq 8$.  Actually in \cite{giorgi}, De Giorgi shows that the falsity of the extension  Bernstein's theorem to the case of $\R^ {m}$  would imply the existence of a minimizing cone in $\R^ {m-1}$.  We refer for instance to \cite{ Almgren, Appleby, At, BGG,  giorgi,  DP,   Mazet,    N, S, SS} and the references therein for further details   on Bernstein's problem and related issues.

\medskip
By the works  \cite{BGG, Vel},  it is known that for $n\geq 4$ the complementary of the Simons cone (which has two connected components $|x|<|y|$ and $|y|<|x|$)  is foliated by two  families of smooth minimal surfaces $(M_a)_{a>0}$ and $({\wt M}_a)_{a>0}$ asymptotic to the Simons cone at infinity. These families are the scaling invariant:
$M_a=aM$ and ${\wt M}_a= a {\wt M}$ with $M$ and ${\wt M}$ admitting respectively the parametrization:
\begin{equation}\label {paramM}\R_+\times \SS^{n-1}\times \SS^{n-1}\ni (\rho, \omega_1, \omega_2)\mapsto (\rho\omega_1, Q(\rho)\omega_2)\in  \R^{2n} \, , \end{equation}  \begin{equation}\label {paramMtild}\R_+\times \SS^{n-1}\times \SS^{n-1}\ni (\rho, \omega_1, \omega_2)\mapsto (Q(\rho) \omega_1, \rho \omega_2)\in  \R^{2n},\end{equation} where $Q$ is  a positive  radial function which belongs to  $\cC^\infty(\R^n)$ and   satisfies
$Q(0)=1$,  $\ds Q(\rho)>  \rho$ for any positive real number $\rho$, and $$\ds Q(\rho)=  \rho + \frac {d_\alpha}   {\rho^{\alpha}}\,  \big(1+ \circ(1)\big) \, \virgp  $$
as  $\rho$ tends to infinity,   with $d_\alpha$ some positive constant and
$$\ds \alpha= -1+  \frac 1 2 \big((2n-1) - \sqrt{(2n-1)^2-16(n-1)}\big)\, .$$

 \medskip 
The minimal surface equation in Riemannian geometry has a natural hyperbolic analogue in the Lorentzian framework. In particular working in the Minkowski space $\R^ {1,  2n }$ equipped with the standard metric: $ dg=  - dt^2+ \sum^n_{j=1} dx_j^2 \, ,$ and considering the time-like surfaces that for fixed $t$ can be parametrized under  the form
\begin{equation}\label {param}\R^n\times \SS^{n-1}\ni (x,\omega)\to \Gamma (t)=(x, u(t,x)\omega)\in \R^{2n}\, ,
\end{equation} with some positive function $u$,
lead to the following quasilinear second order wave equation (see Appendix \ref {deerq}  for the corresponding computations)
$$ \partial_t \Big( \frac {u_t} {\sqrt{1-(u_t)^2+|\nabla u|^2}} \Big)- \sum^{n}_{j=1}\partial_{x_j} \Big( \frac {u_{x_j}} {\sqrt{1-(u_t)^2+|\nabla u|^2}} \Big) + \frac {n-1} {u \,\sqrt{1-(u_t)^2+|\nabla u|^2}}=0 \, \virgp$$
that can be also rewritten as
\beq  
\label {geneq}  \begin{aligned} &{\rm(NW)}\,u:=(1+|\nabla u|^2) \, u_ {tt} - (1-(u_t)^2+|\nabla u|^2)  \,\Delta u\\
&\qquad \qquad  + \sum^{n}_{j, k=1} u_{x_j} u_{x_k} u_{x_j x_k}  - 2 {u_t} (\nabla u \cdot \nabla u_t)+ \frac {(n-1)} {u}\, (1-(u_t)^2+|\nabla u|^2) =0 \, \cdot\end{aligned} \eeq
  Note   that  this equation  is invariant by the scaling 
\begin{equation} \label {scaling} u_a(t, x)= a \, u\Big(\frac {t} {a} \virgp \frac {x} {a}\Big)\, \virgp \end{equation}
in the sense that if $u$ solves \eqref {geneq}  then $u_a$  is also a solution to \eqref {geneq}. In the framework of Sobolev spaces\footnote{ All along this article, we shall denote by $H^s(\R^n)$ the non homogeneous Sobolev space and  by $\dot H^s(\R^n)$ the homogeneous Sobolev space. We refer to \cite{BCD} and the references therein for all necessary definitions and  
properties of those spaces.}, $\ds \dot H^{\frac {n+2}2}(\R^n)$ is invariant under the scaling \eqref{scaling}.

 \medskip 
In this paper, we shall   consider the case when  $n =4$ and assume that $u$ is radial
which implies that for fixed  $t$  the surfaces we are considering are  invariant under the action of the group~$O(4)\times O(4) $. We readily check  that   in that case the function $u$ satisfies the following equation:
\beq  \label {eq:NW} (1+ u^2_{\rho}) \,  u_{tt}  -(1- u^2_{t}) \,  u_{\rho\rho} - 2 u_{t} u_{\rho} u_{\rho t} +3 (1+ u^2_{\rho}-u^2_{t})\Big(\frac 1 u- \frac {u_\rho} {\rho}\Big)= 0 \,\cdot \eeq 
Note that the Simons cone and the minimal surfaces $M_a$ are stationary solutions of our model with
$u(t,\rho)=\rho$ in the case of Simons cone and $u(t,\rho)=Q_a(\rho)$, $\ds Q_a(\rho)=aQ\Big(\frac {\rho} {a }\Big)$ in the case of~$M_a$.  Let us also emphasize that in that case, we have
\beq  \label {eq:Q4}\ds Q(\rho)=  \rho + \frac {d_2}   {\rho^{2}}\,  \big(1+ \circ(1)\big) \, \virgp  \eeq 
as  $\rho$ tends to infinity,   with $d_2$ some positive constant. 
 
 \medskip
We shall be interested in time-like surfaces of the form \eqref{param} that are asymptotic to the Simons cone as $|x|\rightarrow \infty$. To take care of this behavior
we introduce the spaces $X_L$, with $L$ an integer sufficiently large, that we define as being the set of functions~$(u_0,u_1)$ such that~$\nabla (u_0-Q)$ and $u_1$ belong to $H^{L-1} (\R^4)$, and  which satisfy \begin{equation}
\label {Condcauchy} \inf u_0 > 0 \andf  \inf \, (1+|\nabla u_0|^2-(u_1)^2)> 0 \, .\end{equation} 
The Cauchy problem for the quasilinear wave equation \eqref {geneq} is locally well posed in $X_L$ provided that $L$ is sufficiently large.
More precisely one has the following theorem the proof of which is given in Appendix~\ref {welCauchy}:
\begin{theorem}
\label {Cauchypb}
{\sl Consider the Cauchy problem\begin{equation}
\label {Cauchy}
\left\{
\begin{array}{l}
 \refeq {geneq} \,u=0 \\
{u}_{|t=0}= u_0\\
(\partial _t u)_{|t=0}= u_1\,. 
\end{array}
\right.
\end{equation}
Assume that the Cauchy data $(u_0,u_1)$ belongs to the functions class $X_L$ 
with  $L> 4$, then  there exists a  unique maximal solution $u$ of \eqref {Cauchy}  on $[0,T^*[$ such that 
\beq  
\label {maxsol}  (u, \partial_t u)  \in  C([0,T^*[, X_L) \, . \eeq
 Besides if the maximal time $T^*$ of such a solution is finite  (we then say that it blows up), then
 \beq  
\label{star}  \begin{split}
&  \limsup_{t \nearrow T^*} \Big(\Big\|\frac 1 {u (t,\cdot)} \Big\|_{L^{ \infty}} + \Big\|\frac 1 {(1+|\nabla u|^2-(\partial_t u)^2)(t,\cdot)} \Big\|_{L^{ \infty}} +  \sup_{|\gamma| \leq 1 }  \big\| \partial^\gamma_{x} \nabla_{t,x}  u\big\|_{L^{ \infty} }\Big)= \infty \cdot
\end{split}\eeq  }
\end{theorem}

\bigbreak

 The question we would like to address in this paper is that of blow up i.e., the description of possible singularities that
smooth hypersurfaces may develop as they evolve by   the Minkowski zero mean curvature flow, which amounts to investigate blow up dynamics for quasilinear wave equations.
There is by now a considerable  literature dealing with the construction  of type II   blow up solutions  for semilinear heat, wave and Schr\"odinger type equations both in critical and supercritical cases   
(see for instance the articles \cite{Collot, Collot2,  jen, john, K1, K2, KST, KST1, KST2, K, MRI,   P2,   RR,   Vel,  Vel3} and the references   therein). At the most basic level, the strategy in these works is to construct solutions in a two step process, first building an approximate solution, and then completing it to an exact solution, by controlling the remaining error via well-established arguments. The viewpoint we shall adopt here is the one which has been initiated by Krieger,   Schlag and  Tataru in \cite{KST2}, where for the energy critical focusing wave equation they constructed type II blow up solutions, with a continuum of blow up rates, that become singular via a concentration of 
a stationary state profile. The goal of the present paper is to show that this blow up mechanism exists as well  for
 the quasilinear wave equation defined by \eqref {eq:NW}.

%%%%%%%%%%%%%%%%%%%%%%%%%%%%%%%%%%%%%%%%%%%
%%%%%%%%%%%%%%%%%%%%%%%%%%%%%%%%%%%%%%%%%%%

\subsection{Statement of the result} \label{s:results} 
Our  main result is given by the following theorem. 
\begin{theorem}
\label {main}
{\sl For any  irrational  number  $\ds \nu>  \frac 1 2$   and   any positive real number $\delta$ sufficiently small, there exist a positive time $T$ and a radial solution $u(t,\cdot)$   to     \eqref {eq:NW}  on the interval ~$(0,T]$ such that\footnote{  where all along this paper,~$[x]$ denotes the entire part of $x$.}  for any time $t$ in $(0,T]$ 
\beq  
\label {solcont}  (u, \partial_t u)  \in  C((0,T], X_{L_0}) \with  L_0:=2M+1  \, ,  \, \, M=  \Big[ \frac 3  2 \nu   +  \frac 5 4  \Big] \virgp  \eeq
and such that it blows up  at $t=0$ by concentrating the soliton profile: there exist   two  radial functions  ${\mathfrak g}_0 \in \dot H^{s+1 } (\R^4)$ and  ${\mathfrak g}_1  \in \dot H^{s} (\R^4)$,   for any   $0 \leq s < 3 \nu+2$, such that  one has
\begin{eqnarray*}\label{form} u(t,x) &= & t^ {\nu+1} Q\Big(\frac {x} {t^ {\nu+1}} \Big)+ {\mathfrak g}_0(x)+ \eta(t,x)  \, \virgp \\ u_t(t,x) &= &  {\mathfrak g}_1(x)+   \eta_1(t,x) \virgp \end{eqnarray*}
with
$$ \ds \| \nabla \eta( t, \cdot)\|_{ H^{2} (\R^4)} + \|  \eta_1( t, \cdot)\|_{H^{2} (\R^4)} \stackrel{t \to 0}\longrightarrow 0\, .$$
 Moreover, writing 
 \begin{eqnarray*}\label{form} u(t,x) &= & t^ {\nu+1} \Big( Q\Big(\frac {x} {t^ {\nu+1}} \Big)+ \zeta\Big(t  \virgp \frac {x} {t^ {\nu+1}} \Big)\Big) \, \virgp \\ u_t(t,x) &= &  \zeta_1\Big(t  \virgp\frac {x} {t^ {\nu+1}} \Big)   \virgp \end{eqnarray*}
we have 
$$\ds \|  \nabla  \zeta( t, \cdot)\|_{\dot H^{s} (\R^4)} + \|  \zeta_1( t, \cdot)\|_{\dot H^{s} (\R^4)} \stackrel{t \to 0}\longrightarrow 0\, ,  \forall \, 2 < s \leq L_0 -1 \, .$$  
Besides,   ${\mathfrak g}_0$, ${\mathfrak g}_1$  are compactly supported, belong to  $\cC^ {\infty} (\R^4\setminus \{0\})$ and verify   for all  $0 \leq s < 3 \nu+2$
$$  \| \nabla {\mathfrak g}_0\|_{H^{s } (\R^4)} + \| {\mathfrak g}_1\|_{H^{s } (\R^4)}  \leq C_s \, \delta^{3 \nu+2-s} \, , $$
$${\mathfrak g}_0 (x) \sim \frac {d_2} {3\nu+4 } | {\sqrt 2}\, x|^ {3\nu+1}\, , \, {\mathfrak g}_1 (x) \sim d_2 | {\sqrt 2}\, x|^ {3\nu} \, , \,  \mbox{as} \,  x \to 0\, \virgp$$ 
where $d_2$ denotes the constant involved in  \refeq{eq:Q4}.}
\end{theorem}

\bigbreak

\begin{cor}
{\sl There exists a family of hypersurfaces $(\Gamma (t))_{0< t \leq T}$ in $\R^8$ 
    which evolve by  the hyperbolic  vanishing mean curvature flow, and which as $t$ tends to $0$
  blow up   towards a hypersurface  which behaves asymptotically as the Simons cone at infinity. Moreover  $$ t^{ -(\nu +1)}  \Gamma (t) \stackrel{t \to 0} \longrightarrow M  \, , $$
uniformly on compact sets, where $M$ denotes the hypersurface defined by \eqref {paramM}.   
}
\end{cor}

\bigbreak

\begin{remark} \qquad
{\sl  \begin{itemize} 
\item A similar result was established in the case of parabolic vanishing mean curvature flow   by Vel\'azquez in \cite{Vel}. 

\smallskip

\item Combining Theorem \ref {main} with the asymptotic \refeq{eq:Q4}, we readily gather  that the blow up solution $u$ to \eqref {eq:NW} given by Theorem \ref{main}
 satisfies 
\begin{enumerate}
\item $\ds \|  \nabla (u( t, \cdot)-  Q)\|_{ L^\infty((0,T], \dot H^s (\R^4) )}  \lesssim 1\, ,  \forall \, 0 \leq s < 2  $,
\item $\ds \|  \nabla (u( t, \cdot)- |x|- {\mathfrak g}_0)\|_{ \dot H^s (\R^4) }  \stackrel{t \to0}\longrightarrow 0 \, ,  \forall \, 0 \leq s < 2  $,
\item $\ds \|  \nabla (u( t, \cdot)-  Q)\|_{  \dot H^s (\R^4) } \stackrel{t \to0}\longrightarrow \infty\, ,  \forall \, 2 \leq s \leq L_0-1 \cdot$
\end{enumerate} 

\smallskip
\item The parameter $\nu$ is restricted to the  irrationals   just to avoid   the formation of additional logarithms in the construction of an approximate solution to \eqref {eq:NW}. Its limitation to $\nu > \frac 1 2$ is technical.  
\end{itemize}
}
\end{remark}

\bigbreak

%%%%%%%%%%%%%%%%%%%%%%%%%%%%%%%%%%%%%%%%%%%%

%%%%%%%%%%%%%%%%%%%%%%%%%%%%%%%%%%%%%%%%%%%%

\subsection{Strategy of the proof}  \label{outline} 
Let us outline our strategy that is concisely implemented in Sections~\ref{step1}, \ref{step2}, \ref{step3} and \ref{end}.  Roughly speaking, the proof of Theorem \ref{main} is done in two main steps. The first step is dedicated to the construction of an  approximate solution  to  \eqref {eq:NW} as a perturbation of  the concentrating soliton profile $\ds  t^ {\nu+1} Q\Big(\frac {x} {t^ {\nu+1}} \Big)\virgp$ where  $\ds \nu>  \frac 12$ is a fixed irrational  number. 
 The second step which will be the subject of Section \ref{end} is to complement  this approximate solution  to an actual solution ~$u$ by a perturbative argument.  In that step, 
 the properties of the linearized operator of the quasilinear wave equation~\eqref {eq:NW}
around~$Q$, which are studied in  Appendix~\ref{end1}, are essential.

\medskip As we shall see,  the blow up result we establish in this article   heavily relies on the asymptotic behavior of the soliton $Q$.  Thus we shall focus in Section \ref{Analysisst}  on its analysis.  

\medskip To built a good approximate solution, we shall analyze separately the three regions that correspond to three different space scales: the inner region corresponding to $\ds \frac {\rho} {t} \leq t^ { \epsilon_1}$,  the self-similar region where $\ds \frac 1 { 10} \, t^ { \epsilon_1} \leq \frac {\rho} {t} \leq 10 \, t^  {- \epsilon_2}$, and finally  the remote region defined by $\ds  \frac {\rho} {t} \geq t^ {-\epsilon_2}$, where~$0< \epsilon_1< \nu$ and $0< \epsilon_2< 1$ are two  fixed positive real numbers. The inner region is the region where the blow up concentrates. In this region the solution will be constructed as a perturbation of  the profile $\ds t^ {\nu+1} Q\Big(\frac {x} {t^ {\nu+1}} \Big)\cdot$  In the self-similar region,  the profile of the solution is determined uniquely by the matching conditions coming out of the inner region, while in the  remote region the profile remains essentially a free parameter of the construction. 

\medskip

  In Section  \ref{step1}, we investigate  the equation in   the inner region $\ds \frac {\rho} {t} \leq t^ { \epsilon_1}\cdot$ In that region,   we shall look for an  approximate solution   as a power expansion in $t^ {2 \nu}$ of the form:  \beq
\label{first} u_{{\rm
in}}^ {(N)}( t, \rho) = t^ {\nu+1} \sum^{N}_{k=0} t^ {2 \nu k} V_k \Big(\frac {\rho} {t^ {\nu+1}} \Big)\, \virgp\eeq
where $V_0$ is nothing else than  the soliton $Q$,  and where  the functions $V_k$, for  $1\leq k \leq N$,    are obtained recursively, by solving a recurrent system of the form: $$ 
\left\{
\begin{array}{l}
 {\mathcal
L} V_k = F_k(V_0, \cdots, V_{k-1}) \\
V_k(0)= 0 \andf V_k'(0)= 0\,, 
\end{array}
\right.$$ where ${\mathcal
L} $ is the operator defined by:     \begin{equation}
\label{linQ} {\mathcal
L} =  \partial^2_{y} + \Big(\frac 3 y +{B}_1\Big) \partial_{y}+ B_0 \,  ,\end{equation}
with 
\begin{equation}\label{linQcoef}\quad  \left\{
\begin{array}{l}
\ds { B}_1(y)= 9    \frac {Q^2_y} {y} - 6   \frac {Q_y} {Q}\, \virgp \\
\ds B_0(y)= 3 \frac {1+ Q^2_y} {Q^2}\, \cdot
\end{array}
\right.\end{equation}
As it will be established in Paragraph  \ref{step1V_k}, these functions $V_k$  grow at infinity as follows:
$$V_k(y)= \sum^{k}_{\ell=0} \,(\log y)^\ell \sum_{n \geq 2-2(k-\ell)}  d_{n,k,\ell} \,  y^{-n}\, .$$
To obtain a good approximate solution, we are then constrained   to   
restrict the construction to the region $\ds \frac {\rho} {t} \leq t^ { \epsilon_1}\cdot$

\medskip  
The aim of Section  \ref{step2}  is   to   extend the approximate solution  built in Section \ref{step1} to  the  self-similar region~$\ds \frac 1 { 10} \, t^ { \epsilon_1} \leq \frac {\rho} {t} \leq 10 \, t^  {- \epsilon_2}$. Taking into account  the matching conditions coming out from the inner region, we   seek to this extension under the form:
\beq
\label{second}  u_{{\rm
ss}}^ {(N)}  (  t, \rho) = \rho +\lam(t)  \, \sum^{N}_{k = 3} t^ {\nu k} \sum^{\ell(k)}_{\ell=0}   \,\big( \log t \big)^\ell \, w_{k,\ell} \Big(\frac {\rho} {\lam(t)}  \Big) \,\virgp \eeq where $\lam(t)$ is a suitable function which behaves like $t$ near $t=0$, and where the functions~$w_{k,\ell}$, for~$3\leq k \leq N$ and $ \ds 0\leq \ell  \leq \ell(k) = [\frac  {k -3} 2] \virgp$  are determined by induction  using again  \refeq {eq:NW}.
Actually, the natural idea is rather to look for  $u_{{\rm
ss}}^ {(N)}$ under the form \eqref{second}, where $\lam(t)$ is replaced by $t$. However, it turns out that the operator  coming to play in this region  that is   the linearized of~\eqref {eq:NW} around $\rho$ written with respect to the variable $\ds \Big(t,  \frac \rho t \Big)\virgp$ is degenerate on the light cone $\ds \frac {\rho} {t}=\frac 1 {\sqrt{2}} \virgp$  which induces a loss of regularity at each step.

\medskip 
 Actually, the   light cone associated  to the quasilinear equation to be dealt with in this region is different from~$\ds \frac {\rho} {t}=\frac 1 {\sqrt{2}} \virgp$ and is rather given by $\ds \frac {\rho} {\lam(t)}=\frac 1 {\sqrt{2}} \cdot$  It will be constructed concurrently with the solution.  
   As it will be established in Paragraph \ref{system}, the   functions $\ds w_{k,\ell}$  which are determined successively   by solving a  recurrent system admits  an asymptotic behavior under the form:
$$ w_{k,\ell}(z) \sim c_{k,\ell} \, z^{k \nu+1} (\log z)^{\frac {k-3}  2 - \ell} \,, $$
as $z$ tends to infinity, which imposes  to restrict the   self-similar region to~$\ds   \frac {\rho} {t} \lesssim   t^  {- \epsilon_2}$, with $0< \epsilon_2 <1$.

\medskip  
In Section  \ref{step3}, we construct an approximate solution $u_{{\rm
out}}^ {(N)}$  which extends   $u_{{\rm
ss}}^ {(N)}$ to   the whole space, by solving the quasilinear wave equation \refeq {eq:NW} associated to an adapted Cauchy data in the remote region $\ds  \frac {\rho} {t} \geq t^ {- \epsilon_2} \cdot$  We shall look for the  approximate solution in that region under the following form:
\beq
\label{third} u_{{\rm
out}}^ {(N)}   (t, \rho)= \rho+ {\mathfrak g}_0(\rho)+  t  {\mathfrak g}_1(\rho) +  \sum^N_{k = 2 } t^k  {\mathfrak g}_k(\rho)\,, \eeq
where the  Cauchy data $(\cdot+ {\mathfrak g}_0, {\mathfrak g}_1)$ is determined by the  the matching conditions coming out of the self-similar region, and where for $k \geq 2$ the functions ${\mathfrak g}_k$ are determined successively   by a   recurrent relation under the form
$$ {\mathfrak g_k}= 
  {\cG}_k\big({\mathfrak g_j}, j\leq k-1\big) \,\cdot$$

As it will be seen in Paragraph \ref{st33},  these functions ${\mathfrak g}_k$, $k\geq0$, are compactly supported and behaves as $\rho^{1-k+3\nu}$ close to $0$, which ensures that \eqref{third} provides us with a good approximate solution in the remote region.

\medskip   Section \ref{end} is dedicated to the end of  the proof of the blow up result  by constructing an exact solution to \eqref {eq:NW}, thanks to a perturbative argument.   For that purpose, we firstly write 
$$u=   u^ {(N)}+ \varepsilon^ {(N)} \, ,$$ where $u^ {(N)}$ is the approximate solution  to \eqref {eq:NW} constructed in Sections~\ref{step1}, \ref{step2}, \ref{step3} , and then taking into account  \eqref {eq:NW},  we derive the equation satisfied by the remainder term $\varepsilon^ {(N)}$ with respect to  the  variable $\ds \Big(t, \frac  {x} {  t^{1+\nu}} \Big)\cdot$

The study of the equation for  $\varepsilon^ {(N)}$ is based on continuity arguments coupled with  suitable energy estimates. These energy  estimates are established by combining the  spectral  properties of the operator~${\mathcal
L}$  together with the   estimates of  the approximate solution $u^ {(N)}$.   

\medskip   
Finally, we deal in appendix with several complements for the sake of completeness and the convenience of the reader. It is organized as follows.    Section \ref {deerq} is devoted to the derivation of the quasilinear wave equations  \eqref{geneq}. In Section \ref {end1}, we analyze the spectral properties of the operator~${\mathcal
L}$.  In Section \ref   {welCauchy}, we give the proof of the local well-posedness result for the Cauchy problem  \eqref {Cauchy}, namely Theorem \ref {Cauchypb}, and  in Section \ref   {ap:genLres2}, we collect  some useful  ordinary differential equations results  that we use in the self-similar region.

\medskip  
To avoid heaviness, we shall omit in this text   the dependence of all the functions on the parameter $\nu$. All along this article,~$T$ and $C$ will denote  respectively    positive time and   constant depending on several parameters, and  which may vary from line to line.    We also use $A\lesssim B$  to
denote an estimate of the form $A\leq C B$   for some absolute
constant $C$.

%%%%%%%%%%%%%%%%%%%%%%%%%%%%%%%%%%%%%%%%%%%%

\bigbreak\noindent{\bf Acknowledgments:}
  The authors wish to thank very warmly  Laurent Mazet  and Thomas Richard for  enriching  and enlightening discussions   about the Simons cones and Bernstein's problem.

%%%%%%%%%%%%%%%%%%%%%%%%%%%%%%%%%%%%%%%%%%%%

\section{Analysis of the stationary solution} \label{Analysisst} 
\subsection{Asymptotic behavior  of the stationary solution} \label{asym} 
Our analysis in this paper is intimately connected to the   behavior at infinity of the stationary solution to the quasilinear wave equation \eqref {eq:NW}.  In this subsection, we collect the properties of $Q$ that we will use through out this paper.   
\begin{lemma}
\label {ST}
{\sl The Cauchy problem  \begin{equation} \label {eq:ST}  \quad  \left\{
\begin{array}{l} \ds 
-Q_{\rho\rho}+3 (1+ Q^2_{\rho})\Big(\frac 1 Q- \frac {Q_\rho} {\rho}\Big)= 0  \,  \\
\ds  Q(0)= 1 \, \andf  Q_{\rho}(0)= 0\, 
\end{array}
\right.
\end{equation} has a unique solution \footnote{ All along this paper, we identify the radial functions on     $\R^n$ with the functions on $\R_+$.  }    $Q \in \cC^\infty(\R_+)$ which satisfies the following properties:
\begin{itemize}
\item  $Q$ has an even Taylor expansion \footnote{ All the asymptotic expansions  through this paper  can be differentiated any number of times.}   at $0$:
\beq
\label {ST0} 
 \ds  Q(\rho)= \sum_{n  \geq 0}  \gamma_{2n} \, \rho^{2n} ,\eeq
 with some constants $\gamma_{2n}$   such that $\gamma_{0}=1$,
 \item $Q$ enjoys the following bounds  for any  $\rho$ in $\R_+$ \beq \label{propQ} Q(\rho)> \rho \andf Q''(\rho) >0  \, ,\eeq
\item $Q$  has the following asymptotic expansion  as  $\rho$ tends to infinity:
\beq
\label {4} \ds 
Q(\rho) = \rho +  \sum_{n \geq 2} d_{n} \rho^{-n}  \, , \eeq
\end{itemize}
with some constants  $d_{n}$ such that   $d_{2} >0$ and  $d_{4}=0$.  
}
\end{lemma}

\medbreak

\begin{proof}
It is well-known (see for instance  \cite{BGG, Vel}  and the references therein) that the Cauchy problem~\eqref {eq:ST}  admits a unique solution $Q$ in $\cC^\infty(\R_+)$ satisfying \eqref{propQ}, and which behaves  as  
$$ Q(\rho) =\rho +  \frac {d_{2 } }   {\rho^{2}}\,     (1+ \circ(1))\, \,  \mbox{when}  \, \,  \rho \to \infty \, ,$$
with $d_{2 }  > 0$.

\medskip 

In order to determine the   asymptotic formula
\eqref {4}, let us for  $\rho \geq 1$ set
$$ Q(\rho)= \rho \, v(\log \rho) \, .$$
According to \eqref {eq:ST}, this ensures that the function $v$ satisfies 
 \begin{equation}  \label{chang}-(v_{yy}+ v_{y})+ 3 (1+(v+ v_{y})^2) ( \frac 1 v - v- v_{y})= 0 \, .\end{equation} 
Observe that the function $v \equiv 1$ solves \eqref{chang} and that   the linearization of \eqref {chang}  around~$v \equiv 1$ takes the following form: 
\begin{equation}  \label{linchang}   - w_{yy} - 7 w_{y} - 12 w =0\, .\end{equation} 
The characteristic equation of the above linear differential equation \eqref{linchang} 
admits   two real distinct roots $r_1=-3$ and $r_2=-4$.
This ensures that  $Q$ has the following asymptotic development  as  $\rho$ tends to infinity:
$$
Q(\rho) = \rho +  \sum_{n \geq 2} d_{n}  \rho^{-n}  \, , $$
for some constants  $d_{n} $, with $d_{4} =0$.

\medskip 

Finally one can check  that   the formulae \eqref {ST0}  and   \eqref {4} can be differentiated at any order with respect to the variable $\rho$, which completes  the proof of the lemma.
\end{proof}

\medbreak

\subsection{Properties of the linearized operator of the quasilinear wave equation around the ground state} \label{adprop} 

\medskip
The blow up solution we construct in this paper is a small perturbation of the profile $$t^ {\nu+1} Q\Big(\frac {\rho} {t^ {\nu+1}} \Big)\virgp$$  and thus
  the linearization of the quasilinear wave equation~\eqref {eq:NW}
around   $Q$ will play an important role  in our approach. This linearized equation has the form:
\beq
\label {linearbis} (1+Q^2_\rho) w_{tt}- {\mathcal
L} w =0\, \virgp\eeq
where ${\mathcal
L}$ denotes the  operator introduced in  \eqref{linQ} that will be extensively analyzed  in Appendix~\ref{end1}. 
It will be useful later on to emphasize that the function $\Lambda  Q$, where $\Lambda  Q= Q-\rho \,  Q_\rho$ is   a particular solution of the homogeneous equation ${\mathcal L} w=0$ which is positive.  Indeed by virtue  of Lemma~\ref {ST},~$\Lambda  Q$ is positive on~$\rho=0$,   tends to $0$ at infinity and satisfies $(\Lambda  Q)_\rho = - \rho \, Q_{\rho \rho}$. Recalling that~$Q_{\rho \rho}(\rho) > 0$, we end up with the claim. 
 
\medskip
 The strategy we shall adopt in this article  is based on the fact that, up to the  change of function
~$\ds w= H  \, g \,$,  with  \beq
\label {changfunct}\ds H:= \frac  {(1+Q^2_\rho)^{\frac  1 4}} {  Q^{\frac  3 2}}\, \virgp\eeq the above equation   \eqref {linearbis}  rewrites on the following way: 
\beq
\label {redlinear} g_{tt} +{\mathfrak L} g=0
 \, \virgp\eeq
where ${\mathfrak L} $ is the  positive self-adjoint operator on $L^2(\R^4)$ defined by (see Appendix~\ref{end1} for the proof of this fact):
\beq
\label {redlineartrans}{\mathfrak L} =  - q\, \Delta  \, q +\cP    \,  \virgp\eeq
with $\ds q=\frac  1 {(1+Q^2_\rho)^ {\frac 1 2}}\virgp$ 
and where the potential $ \cP$ belongs to $\cC_{rad}^\infty(\R^4)$ and satisfies
\beq
\label {potspect}  \cP (\rho)  =  - \frac 3 {8 \rho^2}  (1 +\circ(1)),\, \,  \mbox{as}  \, \,  \rho \to \infty  \cdot \eeq

\medskip
The spectral properties of the operator ${\mathfrak L}$ which are investigated in  Appendix~\ref{end1} rely on the asymptotic behavior of the potential $ \cP$ at infinity given by \eqref {potspect}.   It comes out of this spectral  analysis    that the operator ${\mathfrak L}$ is positive. Furthermore,     there is a positive constant~$c$ such that we have
\beq
\label {positivity*}  \bigl({\mathfrak L}f |f\bigr)_{L^2(\R^4)}\geq c   \, \|\nabla f\|^2_{L^2(\R^4)} \,,
\quad \forall f\in \dot H_{rad}^1(\R^4)\,.\eeq

%%%%%%%%%%%%%%%%%%%%%%%%%%%%%%%%%%%%%%%%%%%%

\section{Approximate solution in the  inner region }\label{step1}
\subsection{General scheme of the construction of the approximate solution  in the  inner region}\label{step1sub}
In this section, we   shall built in the  region $\ds \frac {\rho} {t} \leq t^ {\epsilon_1}$ (where $0<\epsilon_1< \nu$  is a fixed positive   real number) a family of  approximate solutions $u_{{\rm
in}}^ {(N)}$ to the quasilinear wave equation \eqref {eq:NW} as a perturbation of the  profile $\ds t^ {\nu+1} Q\Big(\frac {x} {t^ {\nu+1}} \Big)\cdot$

\medskip \noindent Writing
\beq
\label{exp} u (t, \rho) = t^ {\nu+1} V\Big(t, \frac {\rho} {t^ {\nu+1}} \Big) \, \virgp \eeq
we get by straightforward computations 
\begin{eqnarray*}u_{\rho}(t,\rho) &=& V_y \Big(t, \frac {\rho} {t^ {\nu+1}}\Big)\, ,\\ u_{\rho\rho} (t, \rho)&=&   \frac {1} {t^ {\nu+1}}\,V_{yy} \Big(t,\frac {\rho} {t^ {\nu+1}}\Big)\, , \\u_t (t, \rho) &= & t^ {\nu+1} \, V_t \Big(t, \frac {\rho} {t^ {\nu+1}}\Big) +(\nu+1) \, t^ {\nu} \, \Lambda V \Big(t, \frac {\rho} {t^ {\nu+1}} \Big):=t^ {\nu}  (\Gamma V)  \Big(t, \frac {\rho} {t^ {\nu+1}} \Big)\, , \\ u_{t\rho}(t, \rho) &=&  {t^ {-1}}\, (\Gamma V)_{y} \Big(t, \frac {\rho} {t^ {\nu+1}} \Big)\, \andf \\ t^ {\nu+1} \, u_{tt} (t,\rho ) &=&  t^ {2\nu} [\Gamma^2 V- \Gamma V] \Big(t, \frac {\rho} {t^ {\nu+1}} \Big)\,  \virgp
\end{eqnarray*}
where  we denote 
\beq
\label{defop} \Gamma V:= t \partial_t V+ (\nu+1) \Lambda V \with    \Lambda  V= V-y V_y  \andf y= \frac {\rho} {t^ {\nu+1}}\, \cdot \eeq

\medskip \noindent
Thus replacing  $u$ by means of  \eqref{exp} into  \eqref {eq:NW} and multiplying by $t^ {\nu+1}$,
we get the following equation 
\beq
\label{eqpart1}
\begin{split}
& \qquad\qquad  (1+V^2_{y}) t^ {2 \nu} \, [\Gamma^2 V- \Gamma V]- \big(1- t^ {2\nu} (\Gamma V)^2\big)V_{yy} \\
&- 2 \, t^ {2\nu} \, V_y \,  (\Gamma V) \,(\Gamma V)_y+ 3 \,  (1+V^2_{y} - t^ {2\nu} (\Gamma V)^2) \Big(\frac 1 V- \frac {V_y} {y}\Big)
=0
\,  \cdot
\end{split}
\eeq
It will be useful later on to point out that the above equation \eqref{eqpart1} multiplied by $\ds \frac V Q$ is  polynomial of order four with respect to $(V, V_y, V_{yy}, \Gamma V, (\Gamma V)_y, \Gamma^2 V)$.

\medskip  In what follows, we shall   look for solutions $V$ to Equation \eqref{eqpart1} under the form
\begin{equation} \label {eq:genformsol} V(t,y) = \sum_{k\geq 0} t^ {2 \nu k} V_k \big(y\big)\,  ,\end{equation}
with $V_0= Q$, where $Q$ is  the stationary solution introduced in Lemma \ref  {ST}.

\medskip  Substituting  this ansatz into  \eqref{eqpart1} multiplied by $\ds \frac V Q \virgp$ we deduce  the following recurrent equation  for $k \geq 1$ 
\begin{equation} \label {eq:V_k}
{\mathcal
L} V_k = F_k(V_0, \cdots, V_{k-1})
\end{equation}
subject to the initial conditions
\begin{equation} \label {eq:initV_k}
V_k(0)= 0\, \andf  V_k'(0)= 0\,, 
\end{equation}
where $ F_k$ depends on $V_j$, $j=0, \cdots, k-1$ only. 

\medskip \noindent
Here ${\mathcal
L}$  is  defined by  \eqref{linQ}. Taking advantage of the asymptotic formula \eqref {4}, this  easily leads  for  $y$ large to the following  asymptotic expansions
\beq
\label {L}
\quad  \left\{
\begin{array}{l}
\ds {   B}_1(y)=  \frac 3 y +   \sum_{n \geq 4 } \beta_n y^{-n}    \\
\ds  B_0(y)= \frac 6 {y^2}  +   \sum_{n \geq 5 } \alpha_n y^{-n} \, \virgp
\end{array}
\right.
\eeq
with some  constants  $\beta_n$ and $\alpha_n$ that can be computed by means of the coefficients~$d_{n}$ involved in the asymptotic formula \eqref {4}. 

\medskip \noindent  Along the same lines, in view of \eqref {ST0} we find the following asymptotic formulae when~$y$ is close to  $0$
\beq
\label {L0}
\quad  \left\{
\begin{array}{l}
\ds {  B}_1(y)=    \sum_{n \geq 0 } a_{2n+1} \, y^{2n+1}    \\
\ds  B_0(y)= 3 +   \sum_{n \geq 1} b_{2n} \, y^{2n } \,\virgp
\end{array}
\right.
\eeq
with some  constants $(a_{2n+1})$ and $(b_{2n})$ that can be expressed in terms of the coefficients~$(\gamma_{2n})$ that arise  in  \eqref {ST0}.  

\medskip \noindent Besides  the source term $ F_k$ can be splitted on two parts as follows:
$$ F_k= F_k^ {(1)} + F_k^ {(2)} \,,$$
with 
$F_1^ {(1)}= 0$ and $F_k^ {(1)}$  for $k \geq 2$  determined by the following equation
\beq
\label {F1}
- \frac V Q V_{yy} + 3 \,  (1+V^2_{y}) \Big(\frac 1 Q- \frac {V\, V_y} {y\, Q}\Big)
= \sum_{k \geq 1 }\Big(- {\mathcal
L}V_k +  F_k^ {(1)}\Big)t^ {2 \nu k} 
\,  \cdot
\eeq
According to \eqref {eq:genformsol}, this gives explicitly \footnote { Here and below, we use  the convention that the sum is null if it is over an empty set.}
\beq
\label {expFk1}
\begin{split} F_k^ {(1)} & = - \frac 1 Q \sumetage{j_1+j_2=k}{j_{i}\geq 1 } V_{j_{1}} \Big((V_{j_{2}})_{yy}+ 3\frac {(V_{j_{2}})_{y}} {y} \Big) -  \frac {3} {y\, Q} \sumetage{j_1+ j_2+j_3 +j_4=k}{j_{i}\leq k-1 } (V_{j_{1}})_{y}(V_{j_{2}})_{y}(V_{j_{3}})_{y} V_{j_{4}} \, \\
& \qquad \qquad \qquad \qquad \qquad \qquad \qquad \qquad \qquad \qquad+\frac {3} {Q} \sumetage{j_1+ j_2=k}{j_{i}\geq 1} (V_{j_{1}})_{y}(V_{j_{2}})_{y} \, \cdot
\end{split}
\eeq

\medskip \noindent Finally combining  \eqref{eqpart1} together with  \eqref {F1} and using the fact  that \beq
\label{Gamma}\ds \Gamma (t^{2 \nu k} V_k)=t^{2 \nu k} \Gamma_k V_k\, ,\eeq where  $\ds \Gamma_k= 2 \nu k+ (1+\nu) \Lambda$,  we readily gather that 
\beq
\label{expFk2}
\begin{split}
 F_k^ {(2)}&=  \sumetage{j_1+j_2+j_3+j_4=k-1}{j_{i}\geq 0 } \frac {V_{j_{1}}} Q (\Gamma_{j_{2}}V_{j_{2}})(\Gamma_{j_{3}}V_{j_{3}}) \Big((V_{j_{4}})_{yy} +  \frac {3 \, (V_{j_{4}})_y} y \Big) \\
& \qquad + \sumetage{j_1+j_2=k-1}{j_{i}\geq 0 } \frac {V_{j_{1}}} Q \Big(\Gamma^2_{j_{2}} - \Gamma_{j_{2}}\Big) V_{j_{2}} + \sumetage{j_1+j_2+j_3+j_4=k-1}{j_{i}\geq 0 } \frac {V_{j_{1}}(V_{j_{2}})_{y}(V_{j_{3}})_{y}} Q \Big(\Gamma^2_{j_{4}} - \Gamma_{j_{4}}\Big) V_{j_{2}}\\
& \qquad -2 \sumetage{j_1+j_2+j_3+j_4=k-1}{j_{i}\geq 0 } \frac {V_{j_{1}}(V_{j_{2}})_{y}} Q\, (\Gamma_{j_{3}}V_{j_{3}}) \,(\Gamma_{j_{4}}V_{j_{4}})_y  - \sumetage{j_1+j_2=k-1}{j_{i}\geq 0 } \frac {3} Q (\Gamma_{j_{1}}V_{j_{1}})(\Gamma_{j_{2}}V_{j_{2}}) \, \cdot
\end{split}
\eeq

%%%%%%%%%%%%%%%%%%%%%%%%%%%%%%%%%%%%%%%%%%%%

%%%%%%%%%%%%%%%%%%%%%%%%%%%%%%%%%%%%%%%%%%%%

\subsection{Analysis of the functions $V_k$}\label{step1V_k}
The goal of the present paragraph is to prove the following result:
\begin{lemma}
\label {STK}
{\sl For any integer $k  \geq 1$, the Cauchy problem  \eqref {eq:V_k}-\eqref {eq:initV_k} has a unique solution $V_k$ in~$\cC^\infty(\R_+)$ with the following asymptotic behaviors:  
\beq
\label {0V_k}
V_k(y) =   \sum_{n \geq 1} c_{2n,k} \,y^{2n}\, , \, \, \mbox{as} \, \, y \sim 0
   \, ,\eeq 
   and
\beq \label{formulaV_k}
V_k(y)= \sum^{k}_{\ell=0} \,(\log y)^\ell \sum_{n \geq 2-2(k-\ell)}  d_{n,k,\ell} \,  y^{-n}\, , \, \, \mbox{as} \, \, y \sim \infty \, ,
\eeq
with \beq \label{concoefV_k}d_{-2(k-2),k,1}=0\, .
\eeq 
}
\end{lemma}

\medbreak
\begin{proof}
\medskip  Let us firstly emphasize that  by classical  techniques of ordinary differential equations,  for any regular function $g$, the solution to the Cauchy problem  \beq
\label{gencp}  \quad  \left\{
\begin{array}{l}
{\mathcal
L} f= g\\
f(0)= 0\, \andf f'(0)= 0\,,  
\end{array}
\right.
\eeq writes under the following form (see Appendix \ref {ap:genLres2} for the proof of this formula) 
\beq
\label{gensol}f (y)=- (\Lambda  Q)(y) \int^y_0 \frac {(1+ (Q_r (r))^2)^ {\frac 3 2} } {Q^3(r) \, r ^3 \,(\Lambda  Q)^2(r)}  \int^r_0 \frac {Q^3(s) \, s ^3 \,(\Lambda  Q)(s)} {(1+ (Q_s (s))^2)^ {\frac 3 2}} \,  g(s) \,  ds \, dr \, \cdot \eeq
 Let us start by  considering  the case when $k=1$. 
 Under notations  \eqref{defop} and in light  of \eqref{expFk1} and~\eqref{expFk2}, we have 
 \beq
\label {expF(Q)1}
\begin{split} F_1(Q) & = F_1^ {(2)}(Q)= (1+ Q^2_y) \,  \Big((1+\nu)^2\Lambda^2 - (1+\nu)\Lambda\Big)  Q 
  \\
& \qquad \qquad \qquad \qquad   - 2  (1+\nu)^2 Q_y   (\Lambda Q)   (\Lambda Q)_{y} +  (1+\nu)^2  (\Lambda Q)^2\, \frac {Q_{yy} \, (Q_y)^2} { (1+ Q^2_y)}    \, \cdot
\end{split}
\eeq
According to   \eqref {ST0}, this implies that  for  $y$ close to $0$  the following asymptotic formula holds 
\beq
\label{dev0F_1} F_1(Q)=   \sum_{n \geq 0}  g_{2n, 1}\, y^{2n} .\eeq
Besides  in view of  \eqref {4}, we get    for    $y$ sufficiently large  the following expansion
\beq
\label{devF_1} F_1(Q)=   \sum_{n \geq 2 }  c_{n,1,0} \, y^{-n}   \, .\eeq  
By virtue of Lemma \ref {ST} which asserts that  \footnote{ We shall designate in what follows the coefficients $d_p$ involved in Formula \eqref {4} by $d_{p,0,0}$.} $d_{4,0,0}:=d_4=0$, we find  that $c_{4,1,0}=0$. Indeed invoking ~\eqref {4} together with \eqref{expF(Q)1}, we easily check that  
$$ c_{4,1,0}= 10 \,(1+\nu)\,(4+5\nu) \,d_{4,0,0}\,, $$ which implies  that the coefficient $c_{4,1,0}$ is null. 

\medskip  This ensures in view of  Duhamel formula \eqref{gensol} that  in that case, 
the Cauchy ~ problem~\eqref {eq:V_k}-\eqref {eq:initV_k}   admits a unique solution~$V_1$ in $\cC^\infty(\R_+)$ satisfying the asymptotic expansions  \eqref {0V_k}  and~\eqref{formulaV_k}  respectively close to~$0$ and at infinity. 

\medskip

Regarding to the expansion coefficients $d_{n,1,\ell} $  of  $V_1$ at infinity,  we can find them by  substituting 
$$ V_1(y)= \sum^{1}_{\ell=0} \,(\log y)^\ell \sum_{n \geq 2\ell}  d_{n,1,\ell} \,  y^{-n}$$ into   \eqref {eq:V_k}  and taking into account   \eqref {L} and \eqref{devF_1}.  This   gives rise to 
 \beq
\begin{split} \label{eqV1} &    \qquad   \qquad \frac   { 2\, d_{1,1,0}} {y^3}  - \sum_{n \geq 2 } ((2n+1) d_{n,1,1}  - n (n+1)\, d_{n,1,0}) \,y^{-n -2} 
  \\
&   + \Big(\frac 6 y +   \sum_{n \geq 4 } \beta_n y^{-n}\Big) \Big(-\frac { d_{1,1,0}} {y^2} +   \sum_{n \geq 2 } ( d_{n,1,1}  - n \, d_{n,1,0}) y^{-n-1} -  \sum_{n \geq 2} n \, d_{n,1,1} \,(\log y) \,y^{-n-1}\Big)  \\
& \qquad \qquad + \Big(\frac 6 {y^2}  +   \sum_{n \geq 5 } \alpha_n y^{-n}\Big) \Big( d_{0,1,0}  + \frac { d_{1,1,0}} {y} +   \sum_{n \geq 2 }  d_{n,1,0}  y^{-n} + \sum_{n \geq 2}  d_{n,1,1} \,(\log y) \,y^{-n} \Big) \\
&  \qquad   \qquad   \qquad   \qquad     \qquad   \qquad   \qquad    + \sum_{n \geq 2 }  n (n+1) \, d_{n,1,1}  \,(\log y)  \,y^{-n -2}  =  \sum_{n \geq 2 } c_{n,1,0}  y^{-n} \,. 
\end{split}
\eeq
 In particular, the identification of  the coefficient of  $y^{-4}$  in \eqref{eqV1}      gives
\beq
\label {coefV_1}
\ds d_{2,1,1}= c_{4,1,0}=0   \, ,
\eeq
which proves that Condition \eqref{concoefV_k} is fulfiled for $k=1$.

\medskip \noindent 
Now using the fact that the coefficient of $(\log y)  \,y^{-n -2}$ in \eqref{eqV1} is null, we find that 
for any integer~$n \geq 2$
$$d_{n,1,1} (n^2-5n+6)+ \sumetage{k_1+k_2=n+2}{k_1 \geq 5, \,  k_2\geq2 } d_{k_2,1,1} \, \alpha_{k_1}-  \sumetage{k_1+k_2=n+1}{k_1 \geq 4, \,  k_2\geq2 } k_2 \,  d_{k_2,1,1}\,  \beta_{k_1} = 0 \,. $$
Along the same lines, by  computing   the coefficients of $y^{-n-2}$  we get 
$$
 d_{n,1,0} (n^2-5n+6) + (5-2n) d_{n,1,1}   +\sumetage{k_1+k_2=n+1}{k_1 \geq 4, \,  k_2\geq2 }   \beta_{k_1}(d_{k_2,1,1}-k_2d_{k_2,1,0}) +\sumetage{k_1+k_2=n+2}{k_1 \geq 5, \,  k_2\geq2 }   \alpha_{k_1}d_{k_2,1,0} =c_{n+2,1,0} .$$
This implies that all the coefficients $d_{n,1,\ell}$ can be  determined successively  in terms  of the coefficients of $F_1(Q)$ involved in \eqref{devF_1} and the coefficients $d_{2,1,0}$ and~$d_{3,1,0}$ that are fixed by the initial  data.

\medskip
We next turn our attention to the general case of any index $k \geq 2$. To this end, we shall proceed by induction  assuming  that, for any integer $1 \leq j \leq k-1$,  the Cauchy problem  \eqref {eq:V_k}-\eqref {eq:initV_k}  admits  a unique solution~$V_j$  in $\cC^\infty(\R_+)$ satisfying  formulae  \eqref {0V_k}  and \eqref{formulaV_k}   as well as Condition \eqref{concoefV_k}.

\medskip
Invoking  \eqref {expFk1} together with  \eqref {0V_k} and \eqref{formulaV_k}, one can easily check that 
\beq
\label{dev0F_k1} F^ {(1)}_{k}(V_0, \cdots, V_{k-1})= \sum_{n \geq 0} g^ {(1)}_{2n, k} \, y^{2n} \, , \, \, \mbox{as} \, \, y \sim 0\, ,
     \eeq
\beq
\label{devinfF_k1} F^ {(1)}_ {k} (V_0, \cdots, V_{k-1}) (y)=   \sum^{k}_{\ell=0} \,(\log y)^\ell \sum_{n \geq 7-2(k-\ell)}  c^ {(1)}_{n,k,\ell} \,  y^{-n}\, , \, \, \mbox{as} \, \, y \sim \infty \, .  \eeq
Similarly  from \eqref {expFk2}, \eqref {0V_k} and \eqref{formulaV_k}, we deduce that 
\beq
\label{dev0F_k2} F^ {(2)}_{k}(V_0, \cdots, V_{k-1})(y)= \sum_{n \geq 0} g^ {(2)}_{2n, k} \, y^{2n} \, , \, \, \mbox{as} \, \, y \sim 0 \, ,
      \eeq
\beq
\label{devinfF_k2} F^ {(2)}_ {k} (V_0, \cdots, V_{k-1}) = (1+ Q^2_y) \,  \big(\Gamma^2_{k-1}- \Gamma_{k-1} \big) \, V_{k-1}+ {\wt F}^ {(2)}_ {k} (V_0, \cdots, V_{k-1}) \, ,  \eeq
where $ {\wt F}^ {(2)}_ {k}$ admits the following expansion at infinity 
\beq
\label{devinfF_k2part2} {\wt F}^ {(2)}_ {k} (V_0, \cdots, V_{k-1}) (y) =   \sum^{k-1}_{\ell=0} \,(\log y)^\ell \sum_{n \geq 7-2(k-\ell)}  {\wt c}^ {(2)}_{n,k,\ell} \,  y^{-n}\, .\eeq
Recall  that by definition
$$\ds \Gamma_{k-1}= 2 \nu (k-1)+ (1+\nu) \Lambda:= \alpha ( \nu,  k-1)+ (1+\nu) \Lambda \, ,$$
which by straightforward computations gives rise to \footnote{ In order to make notations as light as possible, we shall    omit  all along this proof  the dependence of the function $\alpha$  on the parameters $\nu$ and $k$.} 
$$ \big(\Gamma^2_{k-1}- \Gamma_{k-1} \big) \, V_{k-1}= \alpha (\alpha -1) \, V_{k-1} + (1+\nu) (2 \alpha-1) \, \Lambda V_{k-1}+ (1+\nu)^2  \Lambda^2 V_{k-1} \, .$$ 
Setting $$\beta:= \alpha (\alpha -1) + (1+\nu) (2 \alpha-1)+ (1+\nu)^2 \, ,$$
we easily gather that 
$$ \big(\Gamma^2_{k-1}- \Gamma_{k-1} \big) \, V_{k-1}= \beta \, V_{k-1} -  (1+\nu) (2 \alpha-1) \, y\partial_y V_{k-1}+ (1+\nu)^2  y^2\partial^2_{y} V_{k-1} \, .$$ 
It follows therefore from \eqref {4} and  \eqref{formulaV_k} that the following expansion holds at infinity
\beq
\label{devinfF_k2part1} (1+ Q^2_y) \,  \big(\Gamma^2_{k-1}- \Gamma_{k-1} \big) \, V_{k-1} (y) =   \sum^{k-1}_{\ell=0} \,(\log y)^\ell \sum_{n \geq 2-2(k-1-\ell)}  c^ {(2)}_{n,k,\ell} \,  y^{-n}\, ,\eeq
where, under the above notations,  for any integer $0 \leq \ell \leq k-1$ 
$$ 2 \, c^ {(2)}_{2-2(k-1-\ell),k,\ell}= (\beta + n(1+\nu) (2 \alpha-1)+  (1+\nu)^2 n(n+1)) \, d_{2-2(k-1-\ell),k-1,\ell} \, .$$
In view  of the induction assumption  \eqref {concoefV_k} for the index $k-1$,  we get
\beq  \label{formulacoefl1} c^ {(2)}_{2-2(k-2),k,1}= 0  \, .\eeq
Combining \eqref{dev0F_k1} together with \eqref{devinfF_k1}, \eqref{dev0F_k2}, \eqref{devinfF_k2part2}, \eqref{devinfF_k2part1}  and \eqref{formulacoefl1}, we deduce that 
$$F_{k}(V_0, \cdots, V_{k-1})= F^ {(1)}_{k}(V_0, \cdots, V_{k-1})+F^ {(2)}_{k}(V_0, \cdots, V_{k-1})$$
admits the following asymptotic expansions:  
\beq
\label {finFK0}
F_{k}(V_0, \cdots, V_{k-1})(y) =   \sum_{n \geq 0} g_{2n, k} \, y^{2n} \, , \, \, \mbox{as} \, \, y \sim 0
   \, ,\eeq 
\beq \label{finFKinfty}
F_{k}(V_0, \cdots, V_{k-1})(y) = \sum^{k}_{\ell=0} \,(\log y)^\ell \sum_{n \geq 4-2(k-\ell)}  c_{n,k,\ell} \,  y^{-n}\, , \, \, \mbox{as} \, \, y \sim \infty \, ,
\eeq
which can be differentiated   any number of times  with respect to $y$, and with \beq \label{concoefF_k}c^ {}_{2-2(k-2),k,1}=0\, .
\eeq

Therefore Duhamel formula \eqref{gensol} implies that 
the Cauchy problem~\eqref {eq:V_k}-\eqref {eq:initV_k}   admits a unique solution~$V_k$ in $\cC^\infty(\R_+)$ satisfying the asymptotic formulae \eqref {0V_k} and \eqref{formulaV_k} respectively close to~$0$ and at infinity. As for $V_1$ we can determine all the coefficients $d_{n,k,\ell}$  in terms of  $F_k$ and $d_{2,k,0}$ and~$d_{3,k,0}$ that are fixed by the initial data, by substituting the expansion 
$$V_k(y)= \sum^{k}_{\ell=0} \,(\log y)^\ell \sum_{n \geq 2-2(k-\ell)}  d_{n,k,\ell} \,  y^{-n}$$
into \eqref {eq:V_k}. In particular, we get for $0 \leq \ell \leq k-1$ and $n=2-2(k-\ell)$
$$ (n^2-5n+6) d_{n,k,\ell} = c_{n+2,k,\ell} \, ,$$
which by virtue of \eqref{concoefF_k} ensures that  $d_{-2(k-2),k,1}=0$ and proves 
 \eqref{concoefV_k}. 

\medskip
Clearly, the asymptotic expansions \eqref {0V_k}
 and \eqref {formulaV_k}
  can be differentiated   any number of times  with respect to the variable $y$. This concludes the proof of the lemma.
\end{proof}

\medbreak

%%%%%%%%%%%%%%%%%%%%%%%%%%%%%%%%%%%%%%%%%%%% 

\subsection{Estimate of the approximate solution in the inner region}\label{Est1}
Under the above notations, set for any integer~$N \geq 2$\beq
\label{aprx1}u_{{\rm
in}}^ {(N)}(t, \rho) = t^ {\nu+1} V_{{\rm
in}}^ {(N)}\Big(t, \frac {\rho} {t^ {\nu+1}} \Big) \, \with  V_{{\rm
in}}^ {(N)}(t,y) = \sum^{N}_{k= 0} t^ {2 \nu k} V_k \big(y\big)\, \cdot \eeq

Our aim  in this paragraph is to investigate the properties of $V_{{\rm
in}}^ {(N)}$  in the  inner region, namely in the region of $\R^4$ defined as follows:  
\beq \label{inreg} \Omega_{{\rm
in}}:=  \big\{ Y \in  \R^4, \, y= |Y|  \leq t^ {\epsilon_1- \nu} \big\} \, .  \eeq

\medskip Thanks to Lemma \ref {STK}, we easily gather that    $V_{{\rm
in}}^ {(N)}$  satisfies the following $L^{\infty}$
 estimates  on $\Omega_{{\rm
in}}$ : 
\begin{lemma}
\label {solVin}
{\sl  For any multi-index $\alpha$ in $\N^4$ and any integer $\beta \leq  |\alpha|$,  there exist a positive constant~$C_{\alpha, \beta}$ and a small positive time $T=T(\alpha, \beta, N)$ such that for all  $0 < t \leq T$, the following estimates hold: 
\begin{eqnarray}  \label {property1} \qquad \| \langle \cdot \rangle^{\beta} \nabla^{\alpha}   (V_{{\rm
in}}^ {(N)} (t,  \cdot)-Q)\|_{L^{\infty}(\Omega_{{\rm
in}})} &\leq &C_{\alpha, \beta}  \, t^{2\nu} \,, \\ \label {property2}\|  \nabla^{\alpha}   \partial_t V_{{\rm
in}}^ {(N)}(t,  \cdot)\|_{L^{\infty}(\Omega_{{\rm
in}})} &\leq &C_{\alpha}\, t^{2\nu-1} \,,\\ \label {property4} \| \langle \cdot \rangle^{\beta}  \nabla^{\alpha}   (\Gamma V_{{\rm
in}}^ {(N)}) (t,  \cdot)\|_{L^{\infty}(\Omega_{{\rm
in}})} &\leq  & C_{\alpha, \beta} \, , \\  \label {property5} \| \partial_t  (\Gamma V_{{\rm in}}^ {(N)}) (t,  \cdot)\|_{L^{\infty}(\Omega_{{\rm
in}})} &\leq & C\, t^{2\nu-1} \,,\\  \label {property6} \| \langle \cdot \rangle^{\beta}  \nabla^{\alpha}   ((\Gamma^2 -\Gamma ) V_{{\rm in}}^ {(N)})(t,  \cdot)\|_{L^{\infty}(\Omega_{{\rm
in}})}   &\leq  & C_{\alpha, \beta}   \,,
\end{eqnarray}
where as above $\Gamma= t \partial_t + (\nu+1) \Lambda$.}
\end{lemma}

\medbreak

Along the same lines taking advantage of Lemma \ref {STK},  we get the following $L^{2}$  estimates:\begin{lemma}
\label {solVinl2}
{\sl Under the above notations, we have for   all  $0 < t \leq T$:
\begin{eqnarray}  \label {property1l2} \qquad \|   \nabla    (V_{{\rm
in}}^ {(N)} (t,  \cdot)-Q)\|_{L^{2}(\Omega_{{\rm
in}})} &\leq &C \, t^{ \nu} \,,  
\\ \label {property11l2} \qquad \|  \nabla^{\alpha}   (V_{{\rm
in}}^ {(N)} (t,  \cdot)-Q)\|_{L^{2}(\Omega_{{\rm
in}})} &\leq &C_{\alpha}  \, t^{2\nu} \, ,\, \forall \,  |\alpha| \geq 2 \,,
 \\ \label {property4l2} \|     (\Gamma^\ell V_{{\rm
in}}^ {(N)}) (t,  \cdot)\|_{L^{2}(\Omega_{{\rm
in}})} &\leq  & C \,  \log t \,, \, \forall \, \ell=1,2 \,,\\  \label {property5l2} \|   \nabla^{\alpha}   (\Gamma^\ell V_{{\rm
in}}^ {(N)}) (t,  \cdot)\|_{L^{2}(\Omega_{{\rm
in}})} &\leq  & C_{\alpha} \,,   \, \, \forall \,   |\alpha| \geq 1 \, , \forall \,  \ell=1,2\,.
\end{eqnarray}
}
\end{lemma}

\medbreak

\begin{remark}
{\sl  Denoting by $$ \Omega^x_{{\rm
in}}:=   \big\{ x \in  \R^4, \, |x|    \leq t^ {1+\epsilon_1} \big\} \,  ,$$ and combining   \eqref {aprx1} together with the above lemma, we infer  that the following estimates hold for the radial function $u_{{\rm
in}}^ {(N)}$ on $\Omega^x_{{\rm
in}}$: 
\beq \label{sobin} \|  \nabla^{\alpha}   (u_{{\rm
in}}^ {(N)} (t,  \cdot)-t^ {\nu+1} Q\bigl(\frac {\cdot} {t^ {\nu+1}} \bigr))\|_{L^{2}(\Omega^x_{{\rm
in}})} \leq C_{\alpha}  \, t^{\nu+(|\alpha|-3) (\nu+1)}\, ,\, \forall \,  |\alpha| \geq 1 \, ,\eeq \beq \label{sobindt}\|  \nabla^{\alpha}   \partial_t u_{{\rm
in}}^ {(N)} (t,  \cdot)\|_{L^{2}(\Omega^x_{{\rm
in}})} \leq C_{\alpha}  \, t^{\nu+(|\alpha|-3) (\nu+1)}\, ,\, \forall \,  |\alpha| \geq 1 \,, \eeq
\beq \label{sobindt}\|     \partial_t u_{{\rm
in}}^ {(N)} (t,  \cdot)\|_{L^{2}(\Omega^x_{{\rm
in}})} \leq C  \, t^{\nu -3 (\nu+1)} \, \log t \, ,  \eeq
 for   all  $0 < t \leq T$.}
\end{remark}

\medbreak

 Let us end this section by estimating the remainder term: $$\ds {\cR}_{{\rm
in}}^ {(N)}:= \eqref {eqpart1}\, V_{{\rm
in}}^ {(N)}  \, .$$ One has:\begin{lemma}
\label {estap1}
{\sl For any multi-index $\alpha$,  there exist a positive constant~$C_{\alpha, N}$ and a small positive  time $T=T(\alpha, N)$ such that for all  $0 < t \leq T$,  the  remainder term~$\cR_{{\rm
in}}^ {(N)}$ satisfies \begin{equation}\label {Condremin} \|\langle \cdot \rangle^{\frac 3 2}  \,  \nabla^{\alpha} \cR_{{\rm
in}}^ {(N)}(t, \cdot)\|_{L^2(\Omega_{{\rm
in}})} \leq C_{\alpha, N} \, t^{2\nu+ 2N \epsilon_1 - \frac 3 2 (\nu- \epsilon_1)}   \, .\end{equation}
}
\end{lemma}

\medbreak

\begin{proof}
In view of computations carried out in Section \ref{step1V_k} and particularly on page \pageref {F1}, we have 
$$ \frac V Q \, \Big[\eqref {eqpart1} \Big(\sum_{k\geq 0} t^ {2 \nu k} V_k\Big)\Big]= \sum_{k \geq 1 }\Big(- {\mathcal
L}V_k +  F_k\Big)t^ {2 \nu k} 
\,  .$$
Thus recalling that $\ds \frac V Q \, \big[\eqref {eqpart1} V\big]$ is  a polynomial of order four and  taking into account   Lemma
~\ref {STK},   we deduce that 
\begin{equation}\label {defrem1} {\wt \cR}_{{\rm
in}}^ {(N)}= \frac {V_{{\rm
in}}^ {(N)}} Q \, {\cR}_{{\rm
in}}^ {(N)}= \sum_{N+1 \leq k \leq 4N }t^ {2 \nu k} G_k
\,  ,\end{equation}
where $G_k$ depending  on $V_j$, $j=0, \cdots, N$,   is defined as the function $F_k$  by formulae similar to~\eqref {expFk1} and \eqref{expFk2}, where we  assume  in addition  that the  involved indices $j_i$     range from $0$ to $N$. 

\medskip This of course implies that the function $G_k$,  $N+1 \leq k \leq 4N$ admits the following expansions that can be differentiated   any number of times  with respect to the variable $y$, respectively close to $0$ and at infinity:
\begin{equation}\label {dev0rem1}
G_{k}(y) =   \sum_{n \geq 0} {\wt g}_{2n, k} \, y^{2n}  
   \,  ,
   \, \end{equation}
\begin{equation}\label {devinftyrem1}
G_{k}(y) = \sum^{k}_{\ell=0} \,(\log y)^\ell \sum_{n \geq 4-2(k-\ell)}  {\wt c}_{n,k,\ell} \,  y^{-n}\, ,
\end{equation}
with some  constants ${\wt g}_{2n, k}$ and  ${\wt c}_{n,k,\ell}$ that can be determined recursively in terms of the functions~$V_j$, for $j=0, \cdots, N$.

\medskip Recalling that   by definition 
$${ \cR}_{{\rm
in}}^ {(N)}= \frac Q {V_{{\rm
in}}^ {(N)}}  \, {\wt \cR}_{{\rm
in}}^ {(N)} \,, $$ we deduce taking into account  Lemma \ref {solVin}  and  Formula \eqref{inreg}   that for any multi-index $\alpha$,  there exist a positive constant~$C_{\alpha, N}$ and a  positive  time $T=T(\alpha, N)$ such that for any time~$0 < t \leq T$, we have 
$$\|\langle \cdot \rangle^{\frac 3 2}  \,  \nabla^{\alpha} \cR_{{\rm
in}}^ {(N)}(t, \cdot)\|_{L^2(\Omega_{{\rm
in}})} \leq C_{\alpha, N} \, t^{2\nu+ 2N \epsilon_1 - \frac 3 2 (\nu- \epsilon_1)}\, .$$
This ends the proof of   the lemma.

\end{proof}

\bigbreak

%%%%%%%%%%%%%%%%%%%%%%%%%%%%%%%%%%%%%%%%%%%%

%%%%%%%%%%%%%%%%%%%%%%%%%%%%%%%%%%%%%%%%%%%%

\section{Approximate solution in the  self-similar region }\label{step2} 
\subsection{General scheme of the construction of the approximate solution in the  self-similar region}\label{step2sub}  
Our aim in this section  is to  built  in the   region 
  $\ds \frac 1 { 1 0} \,  t^ { \epsilon_1} \leq \frac {\rho} {t} \leq 1 0 \, t^ {- \epsilon_2} $   an approximate solution~$u_{{\rm
ss}}^ {(N)}$  to \eqref {eq:NW}  which extends  the approximate solution~$u_{{\rm
in}}^ {(N)}$ constructed  in the  inner region~$\ds  \frac {\rho} {t} \leq t^{ \epsilon_1} \cdot$ Here  $0< \epsilon_2< 1$   is  fixed.

\medskip We shall look for this  solution under the following form: 
\beq
\label{exp2}u (t, \rho) = \lam(t) \,    (z+ W (t,z) )\with z= \frac {\rho} {\lam(t)}\virgp\eeq
  and where $\lam(t)$  is a function which behaves like $t$, for $t$ close to $0$, and that will be constructed at the same time as the profile $W$. In fact, $\lam(t)$ will be given by an expression of the form:   \beq
\label{pert} \lam(t)=  t \Big(1+ \sum_{k \geq 3} \sum^{\ell(k)}_{\ell  = 0}\lam_{k, \ell} \, t^ {\nu k} \,(\log t)^\ell \Big)  \with  \ell(k) = \big[\frac  {k -3} 2\big] \cdot \eeq  
 By straightforward computations, we find that \begin{eqnarray} \nonumber u_{\rho} (t, \rho) &=& 1+  W_z \Big(t, \frac {\rho} {\lam(t)} \Big)\, ,\\  \nonumber u_{\rho\rho} (t, \rho)&=&  \lam(t)^ {-1}W_{zz} \Big(t, \frac {\rho} {\lam(t)} \Big)\, , \\ \label{formulaW1}u_t (t, \rho) &= & \lam(t) \, W_t \Big(t, \frac {\rho} {\lam(t)} \Big)+\lam'(t) \, \Lambda W \Big(t, \frac {\rho} {\lam(t)} \Big):= W_1\Big(t, \frac {\rho} {\lam(t)} \Big)\, , \\  \nonumber u_{t\rho} (t, \rho) &=&  \lam(t)^ {-1} (\partial_zW_1)\Big(t, \frac {\rho} {\lam(t)} \Big) \, ,\\ \lam(t) \, u_{tt} (  t, \rho) &=&  W_2\Big(t, \frac {\rho} {\lam(t)} \Big) \label{formulaW2}\, \virgp
\end{eqnarray}
  with 
\beq\label{defW2}
\begin{split}
& \qquad  {}  W_{2}(t,z):= \lam(t) \lam''(t) \, \Lambda W 
+ 2 \lam(t) \lam'(t)\, \Lambda W_t + \lam^2(t) \, W_{tt} +(\lam'(t))^2 z^2 W_{zz} \, \\
&  \qquad  {}  \qquad  {}  \qquad    \qquad  {}  \qquad     = z^2\,W_{zz}  (t,z)+  t^ {2} W_{tt}  (t,z) +  2   t \Lambda W_t  (t,z)+  {\wt W}_2  (t,z)  
  \,   \virgp
\end{split}
\eeq
where \beq\label{defW2tilde}  {\wt W}_2  (t,z)= ((\lam'(t))^2-1) z^2 W_{zz}  + (\lam^2(t) -t^2)  W_{tt}   + 2 (\lam(t) \lam'(t)-t) \Lambda W_t  + \lam(t) \lam''(t) \, \Lambda W\,   ,\eeq
and where as above  $\Lambda W= W-z W_{z}$.

\medskip Thus substituting  $u$ by means of  \eqref{exp2} into  \eqref {eq:NW} multiplied by $\lam(t)$,
we find   that the function
~$W$ solves the following equation:
\beq
\label{eqregion2}
\begin{split}
& \qquad  (1+(1+W_{z})^2) W_2   - (1- (W_1)^2)W_{zz}  - 2 (1+W_{z}) W_1(W_1)_{z}  \\
& \qquad  {} \qquad  {} \qquad \qquad\qquad  \qquad {}-3\big(1+(1+W_{z})^2- (W_1)^2\big) \Big(   \frac   {\check   {W}  }   {z^2}  +\frac   {W_z}{z}\Big) =0\,  \virgp
\end{split}
\eeq
where $\ds \check   {W}=\frac   {W}{(1+ \frac W  z)} \, \cdot$

\medskip \noindent Introducing  the notations \beq\label{defnot}
\begin{split}
& \qquad  {}  W'_2 :=  W_2  - (\lam')^2 z^2 W_{zz} =  \lam^2 W_{tt}   
+ 2 \lam  \lam'  \Lambda W_t + \lam \lam''  \Lambda W  \, ,  \, \, W_3:= \lam \, \,W_{tz} \, , \end{split}
\eeq 
we readily gather that the above equation \eqref{eqregion2} rewrites in the following way:
\beq \label{new}  (2 z^2 -1+A_0) \, W_{zz} + A_1= 0  \,  , \eeq 
with
\begin{eqnarray} A_0 &=&   (\lam' W + \lam W_t) ^2+ 2 \lam' z  (\lam' W + \lam W_t) +2 ( (\lam')^2-1) z^2    \label{defA0} \, , \\ \quad \quad \quad A_1 &=& (1+(1+W_{z})^2) W'_2 - 2(1+W_{z}) W_1 W_3
 -3 \,\big(1+(1+W_{z})^2- (W_1)^2\big) \Big(  \frac   {\check   {W}  }   {z^2}+ \frac   {W_z}{z}\Big)\, \nonumber \cdot \end{eqnarray}

\medskip \noindent
Denoting by  $L$ the linear operator defined by: 
\begin{equation} \label{lineq} LW=  (2z^2-1) \,W_{zz} + 2 t^ {2} W_{tt}  +  4   t \Lambda W_t   - 6  \frac {W_{z}} {z} 
  - 6 \frac {W} {z^2}\,  \virgp \end{equation}
  we infer that the above equation \eqref{eqregion2} undertakes the following form: 
\beq
\label{eqregion22}
\begin{split}
& \qquad  {}  LW =  - { A}_0 W_{zz} -   \, \big[2 W_z+ (W_{z})^2\big]\, W'_2   -  2  {\wt W}'_2   + 2  \, (1+W_{z}) W_1 W_3 \\ & \qquad \qquad \qquad \qquad \qquad  \qquad \qquad {} -  \frac   {6}{z^3}   {\check {W}} \, W  +3\big(2 W_z+ (W_{z})^2- (W_1)^2\big) \Big(  \frac   {\check   {W}  }   {z^2}+ \frac   {W_z}{z}\Big) \,  \virgp
\end{split}
\eeq
where \beq\label{defnott}
\begin{split}
& \qquad  {}   {\wt W}'_2:=W'_2 - t^ {2} W_{tt} -2   t \Lambda W_t= {\wt W}_2 -((\lam')^2-1) z^2\,W_{zz}  \,.
\end{split}
\eeq  
It will be useful later on to notice    that under the above notations,  \eqref{eqregion22} also rewrites in the following way:
\beq
\label{eqregion22bis}
\begin{split}
& \qquad  {}  LW = - 2 \, {\wt W}_2 -   \, \big[2 W_z+ (W_{z})^2\big]\, W_2   - (W_1)^2\, W_{zz}   + 2  \, (1+W_{z}) W_1 (W_1)_z \\ & \qquad \qquad \qquad  {}  \qquad  {} \qquad  {}  \qquad  {}  \qquad  {} -  \frac   {6}{z^3}   {\check {W}} \, W   +3\big(2 W_z+ (W_{z})^2- (W_1)^2\big) \Big(  \frac   {\check   {W}  }   {z^2}+ \frac   {W_z}{z}\Big) \,  \cdot
\end{split}
\eeq

\medskip 

The asymptotic of the solution \eqref{exp2} at the origin has to be coherent with that of \eqref{exp} at infinity.  To determine this asymptotic, we combine the expansion  \eqref{pert} together with Formula~\eqref{formulaV_k}, which gives: 
\beq
\label{formmatch22} u_{{\rm
in}} (t, \rho) = \lam(t)  \Big( z+ \sum_{k \geq 3} t^ {\nu k} \sum^{\ell(k)}_{\ell  = 0}  (\log t)^\ell   \sum_{0   \leq \alpha \leq \frac  {k -3} 2- \ell}   (\log z)^{\alpha}   \sum_{ \beta \geq 1-k +2(\alpha+\ell)} c^{k,\ell} _{\alpha,\beta}    z^{\beta}  \Big), \, \,  \mbox{as} \, \,  z \to 0  \, ,\eeq
where the coefficients $c^{k,\ell} _{\alpha,\beta} $ admit the representation:
$$ c^{k,\ell} _{\alpha,\beta} = c^{k,\ell,0} _{\alpha,\beta} + c^{k,\ell,1} _{\alpha,\beta} \, ,$$
with $c^{k,\ell,0} _{\alpha,\beta}$ independent of $\lam$ and given by
$$\left\{
\begin{array}{l}
\ds 
c^{k,\ell,0} _{\alpha, \beta} = 0 \, , \, \,\mbox{if}\,\,\beta+k-1 \, \,\mbox{is odd}\andf   \\
\ds 
c^{k,\ell,0} _{\alpha, \beta} = (-
\nu)^\ell  \begin{pmatrix} \alpha   + \ell \\ \alpha\end{pmatrix} d_{-\beta, \frac {\beta+k-1} 2,\alpha   + \ell}\, , \, \,\mbox{if}\,\,\beta+k-1 \, \,\mbox{is even}\, \virgp
\end{array}
\right.$$ 
where $d_{n, k,  \ell}$ denotes the coefficient arising in \eqref{exp}.

\medskip  \noindent
The coefficients $c^{k,\ell,1} _{\alpha, \beta} $ depend only on $\lam_{p,q}$ involved in \eqref{pert} with $3 \leq p \leq k-3$ and are equal to zero if $\beta+k-1-2(\alpha   + \ell) \leq 2$ or if $k<6$ or if $\ds \ell >\frac {k-6}2 \cdot$

\medskip  \noindent
Let us point out that taking into account  Lemma \ref {ST}  together with Property \eqref{concoefV_k}, which respectively assert that  $d_{4, 0,0 } = 0$ and~$d_{4-2m,m,1}=0$ for any integer $m \geq 1$, we infer  that  \beq \label{coefcase50} c^{5,0} _{0, -4}=0   \andf c^{5,1} _{0, \beta}=0,  \, \forall   \beta\,.
\eeq

Formula \eqref{formmatch22}    leads us to look for  the approximate solution in the  self-similar region under the  form: \beq
\label{formreg2g}\ds u(t, \rho) = \rho+ \lam(t) \, W \Big(t, \frac {\rho} {\lam(t)} \Big)\, ,\quad   \eeq where
 \beq
\label{formreg2} W(t,z)= \sum_{k \geq 3} t^ {\nu k} \sum^{\ell(k)}_{\ell  = 0}   \,(\log t)^\ell \, w_{k,\ell} (z ) \,.\eeq
 To fix $\lam(t)$, we require that the function $A_0$ defined by \eqref{defA0} satisfies
\begin{equation} \label{condlam} \big(A_0\big)_{z=\frac 1 {\sqrt{2}}} =0 \, \cdot\end{equation}
Actually a  difficulty that we face in solving  \eqref{eqregion22}  is in handling the singularity of the operator~$L$ defined by \eqref{lineq}  on the  light cone  $\ds z=\frac 1 {\sqrt{2}} \cdot$  The above condition \eqref{condlam}   ensures that the coefficient of~$W_{zz}$ involved in  the equation we deal with vanishes at 
$\ds z=\frac 1 {\sqrt{2}} \cdot$  This  will enable us   to determine successively the functions~$w_{k,\ell}$ involved in \eqref{formreg2} without loss of regularity at each step.

\smallskip   Invoking \eqref{pert} together with \eqref{formreg2}, we infer that   the functions $W_1$,  $W_2$, ${\wt
W_2}$,  $W'_2$, ${\wt
W_2}'$, $W_3$,  $\check{W}$ and $A_0$  defined above  admit   expansions of the same form as $W$. More precisely, one has:
\begin{eqnarray*} W_i(t,z) &= & \sum_{k \geq 3} t^ {\nu k} \sum_{0 \leq  \ell  \leq  \frac  {k -3} 2}   \,(\log t)^\ell \, w^i_{k,\ell} (z ) \,, i=1,2, 3 \,,\\  
{\wt W}_2(t, z) &= &\sum_{k \geq 6} t^ {\nu k} \sum_{0 \leq  \ell  \leq  \frac  {k -6} 2}   \,(\log t)^\ell \, {\wt w}^2_{k,\ell} (z ) \,,\\ 
W'_2 (t, z) &= & \sum_{k \geq 3} t^ {\nu k} \sum_{0 \leq  \ell  \leq  \frac  {k -3} 2}   \,(\log t)^\ell \, w^{(2, ')}_{k,\ell} (z ) \,, \\  {\wt W}'_2(t, z) &= &\sum_{k \geq 6} t^ {\nu k} \sum_{0 \leq  \ell  \leq  \frac  {k -6} 2}   \,(\log t)^\ell \, {\wt w}^{(2, ')}_{k,\ell} (z ) \,,\\  \check   {W}(t, z) &= & \sum_{k \geq 3} t^ {\nu k} \sum_{0 \leq  \ell  \leq  \frac  {k -3} 2}   \,(\log t)^\ell \, \check   {w}_{k,\ell} (z ) \,, \\   A^0(t,z) &=&  \sum_ {k \geq 3}  t^ {\nu k} \sum_{0 \leq  \ell  \leq \frac {k - 3 } 2  }  \,(\log t)^\ell \,  A^{0}_{k,\ell} (z)  \,,\end{eqnarray*} 
where $w^i_{k,\ell}$, $i=1,2, 3$, and $w^{(2, ')}_{k,\ell}$  depend only on $w_{k',\ell'}$, $3 \leq k' \leq k$ and  $\lambda_{k^{''}, \ell^{''}}$, $3 \leq k^{''} \leq k-3$, where ${\wt w}^2_{k,\ell}$  and ${\wt w}^{(2, ')}_{k,\ell}$ depend on   $w_{k',\ell'}$ and  $\lambda_{k^{''}, \ell^{''}}$, $3 \leq k', k^{''} \leq k-3$ and   $ A^{0}_{k,\ell}$ on $w_{k',\ell'}$ and  $\lambda_{k^{''}, \ell^{''}}$, with $3 \leq k', k^{''} \leq k$.

\medskip \noindent Observe also that  \footnote{ with the convention all along this section that $\lam_{k, \ell'}=0$  and $w_{k,\ell'}\equiv 0$ if $k<3$ or $\ds \ell' > \big[\frac  {k -3} 2\big] \cdot$} \beq
\label{formreg1kl}
\begin{split}
&  w^1_{k,\ell} = \big(\nu k + \Lambda\big)  w_{k,\ell}+ (\ell+1) w_{k,\ell+1} +  \wt w^1_{k,\ell}  \, ,\\
&    \wt w^1_{k,\ell} = \sum_{k_1+k_2=k, \ell_1+\ell_2= \ell   }  \lam_{k_2, \ell_2} \, \big(\nu k_1 \, w_{k_1,\ell_1} + (\ell_1+1) w_{k_1,\ell_1+1}\big) \\
& \qquad  {}  \qquad {}  \qquad  {}  \qquad  {}  \qquad + \sum_{k_1+k_2=k, \ell_1+\ell_2= \ell} \big[ (1+\nu k_2)\lam_{k_2, \ell_2}+ + (\ell_2+1) \lam_{k_2,\ell_2+1} \big]\,   \Lambda w_{k_1,\ell_1}\,  , 
\end{split}
\eeq
and that \beq
\label{formreg3kl}
\begin{split}
&  w^3_{k,\ell} =  \nu k \partial_z w_{k,\ell}+ (\ell+1) \partial_z w_{k,\ell+1} +  \wt w^3_{k,\ell}  \, ,\\
&    \wt w^3_{k,\ell} = \sum_{k_1+k_2=k, \ell_1+\ell_2= \ell} \lam_{k_2, \ell_2} \, \big(\nu k_1 \, \partial_z w_{k_1,\ell_1} + (\ell_1+1) \partial_z w_{k_1,\ell_1+1}\big) \,  .
\end{split}
\eeq
In addition, one has \footnote{ One can for $\wt w^2_{k,\ell} $ and $ \wt w^{(2, ')}_{k,\ell}$ give explicit expressions of the same type  as for $\wt w^1_{k,\ell}$ and $\wt w^3_{k,\ell}$, but to avoid needlessly burdening the text, we will not explicit them.}    

\beq \begin{split} \label{formreg2kl}
 & w^2_{k,\ell} (z )= z^2 \partial^2_z w_{k,\ell}  + \nu k \big( \nu k +1 - 2 z \partial_z \big)  w_{k,\ell} + (\ell+1)\big(2\nu k+1 - 2 z \partial_z \big)  w_{k,\ell+1} \\
&  \qquad \qquad  \qquad  {}   \qquad  {}   \qquad  {}   \qquad  {}   \qquad  {}       \qquad  {}\qquad  {} \qquad  {}+(\ell+1) (\ell+2) w_{k,\ell+2} +  \wt w^2_{k,\ell} \, ,\end{split}
\eeq
and 
\beq
\begin{split}\label{formreg2tildekl}
& w^{(2, ')}_{k,\ell}  =   \nu k \big( \nu k +1 - 2 z \partial_z \big)  w_{k,\ell}   + (\ell+1)\big(2\nu k+1 - 2 z \partial_z \big)  w_{k,\ell+1}  \\
&  \qquad \qquad  \qquad  {}   \qquad  {}   \qquad  {}    \qquad  {}   \qquad  {}\qquad  {}  +(\ell+1) (\ell+2) w_{k,\ell+2} +  \wt w^{(2, ')}_{k,\ell} \, . 
\end{split}
\eeq

 Now substituting  expansions \eqref{pert} and \eqref{formreg2}    into \eqref{eqregion22} and \eqref{condlam},  we deduce the following recurrent  system for  $k \geq 3$: 
\begin{equation} \label {eq:W_kl}\quad  \left\{
\begin{array}{l}
\ds \ds\wt {\mathcal
L}_k w_{k,\ell} = F_{k,\ell} \,  ,  \,  0 \leq  \ell  \leq  \ell(k)\\
\ds   (1+\nu k) \lam_{k, \ell} + (\ell+1)\lam_{k, \ell+1}  =- \big((1+\nu k)w_{k, \ell} + (\ell+1) w_{k,\ell+1}\big) _{z=\frac 1 {\sqrt{2}}} + g_{k, \ell}  \,.
\end{array}
\right.\end{equation}
Here $\wt {\mathcal
L}_k$   refers to    the  operator  \begin{equation} \label {eq:L_k} \wt {\mathcal
L}_{k} w=  (2z^2-1) w_{zz} - \Big(4 \,z \,  \nu \,k+ \frac {6} {z}\Big)w_z + \Big( 2 \,  \nu k (1 +\nu  k)- \frac {6} {z^2} \Big)w \,,\end{equation}
and   the source term $F_{k,\ell}$ can be divided into a linear and a nonlinear  parts as follows:
\begin{equation} \label {decFkl} F_{k,\ell}=  F^{{\rm
lin}}_{k,\ell}+ F^{{\rm
nl}}_{k,\ell} \, ,\end{equation}
where
\begin{equation} \label {deflinFkl} F^{{\rm
lin}}_{k,\ell} = - 2(2 \nu k +1)\, (\ell+1)   w_{k,\ell+1} + 4\,z\, (\ell+1)(w_{k,\ell+1})_z - 2 (\ell+1)  (\ell+2) w_{k,\ell+2} \, ,\end{equation}
and where  \footnote{ $F^{{\rm
nl}}_{k,\ell}$ and $g_{k, \ell}$ are  identically null if  $k < 6 $ or $\ell > \frac {k-6} {2}\cdot$} the nonlinear part  $F^{{\rm
nl}}_{k,\ell}$, for $k \geq 6$ and $\ds 0 \leq \ell \leq \frac {k-6} {2}\virgp$   only depends  on~$w_{k', \ell'}$ and~$\lambda_{k^{''}, \ell^{''}}$,  for~$ 3 \leq k', k^{''}\leq k- 3$.   
Similarly, the coefficients  $g_{k, \ell} $  only depends on the values of the functions $w_{k', \ell'}$ on $z=\frac 1 {\sqrt{2}}$ and the coefficients $\lambda_{k^{''}, \ell^{''}}$,  for $ 3 \leq k', k^{''}\leq k- 3$.

\medskip 
In other words, for any integer $k \geq 3$ the functions $\ds (w_{k,\ell})_{0 \leq \ell \leq \ell(k)} $ satisfy:
\begin{equation} \label {reckl}   
 \cS_{k} \,  \cW_k = \cF^{{\rm nl}}_{k}\, \end{equation}
where   $\cS_{k}$ denotes the following matrix operator: 

$$
\left(
\begin{array}{ccccccccc}
\wt {\mathcal
L}_k &{\mathcal
A}_k(0)+ {\mathcal
B}(0,z)\partial_z& {\mathcal
C}(0)& 0& \dots    & \dots & \dots \\
0 &\wt {\mathcal
L}_k & {\mathcal
A}_k(1)+ {\mathcal
B}(1,z)\partial_z&{\mathcal
C}(1)&0  & \dots  & \dots\\
 \dots &0 &\wt {\mathcal
L}_k &\dots& \dots& 0& \dots  \\
\dots  &\dots & \dots & \wt {\mathcal
L}_k & \dots &  \dots&  \dots 
\\
\dots  &\dots & \dots & \dots & \dots  & \dots &  \dots \\
\dots  &\dots & \dots & \dots & \dots  & \dots &  \dots \\
\dots  &\dots & \dots & \dots & \dots  & \dots &  \dots \\
 \dots &\dots & \dots & \dots & \dots   & \wt {\mathcal
L}_k &{\mathcal
A}_k(\ell(k)-1)+ {\mathcal
B}(\ell(k)-1,z) \partial_z\\
\dots  &\dots  &\dots  & \dots & \dots    & \dots & \wt {\mathcal
L}_k
\end{array}
\right)
$$
with 
\begin{eqnarray*}{\mathcal
A}_k(\ell) &= & - 2(2 \nu k +1)\,(\ell+1)\, \virgp \\
  {\mathcal
B}(\ell,z) &= & 4 z  (\ell+1)   \, , \\
 {\mathcal
C}(\ell) &= &-  2 (\ell+1)  (\ell+2)  \, \virgp \end{eqnarray*}
and
$$
\cW_k=
 \left(
\begin{array}{ccccccccc}
w_{k,0} \\
\vdots \\
w_{k,\ell} \\
 \vdots \\
w_{k,\ell(k)}  
\end{array}
\right) \, , \qquad   \cF^{{\rm
nl}}_{k}=
 \left(
\begin{array}{ccccccccc}
F^{{\rm
nl}}_{k,0} \\
\vdots \\
F^{{\rm
nl}}_{k,\ell} \\
 \vdots \\
F^{{\rm
nl}}_{k,\ell(k)}  
\end{array}
\right)  \, .
$$
 Let us   emphasize that  we do not subject  the above system  to any Cauchy  data as for the system~\eqref {eq:V_k} corresponding to the inner region. In order to solve uniquely  \eqref {eq:W_kl}, we shall    take into account the matching conditions coming out from the inner region, namely we require that  
 \begin{equation}  \label{exporigkl} w_{k,\ell}(z)= \sumetage {0   \leq \alpha \leq \frac  {k -3} 2- \ell}  { \beta \geq 1-k +2(\alpha+\ell)}\,c^{k,\ell} _{\alpha,\beta}  \, (\log z)^{\alpha}   \, z^{\beta}  \,, \, \,  \mbox{as} \, \,  z \to 0 \, , \end{equation}
 where $c^{k,\ell} _{\alpha,\beta}= c^{k,\ell} _{\alpha,\beta}(\lam)$ are given by 
\eqref{formmatch22}.

\medskip In view of \eqref{eqregion22}, one can write $F^{{\rm
nl}}_{k,\ell}$  explicitely as follows:  \beq \label{eqnl}
F^{{\rm
nl}}_{k,\ell}= F^{{\rm
nl}, 1}_{k,\ell} + F^{{\rm
nl}, 2}_{k,\ell} + F^{{\rm
nl}, 3}_{k,\ell} + F^{{\rm
nl}, 4}_{k,\ell}\, ,\eeq  
where 
\beq \label{eqnl1} F^{{\rm
nl}, 1}_{k,\ell}= -2 {\wt w}^{(2, ')}_{k,\ell}   ,\eeq
\beq
\begin{split} \label{eqnl2} & F^{{\rm
nl}, 2}_{k,\ell}=  \sumetage {j_1+j_2=k} {\ell_1+ \ell_2= \ell }   6 \, (w_{j_1,\ell_1} )_{z} \Big( \frac 1  {z}(w_{j_2,\ell_2} )_{z} + \frac 1  {z^2} \check   {w}_{j_2,\ell_2} \Big) - 2 (w_{j_1,\ell_1} )_{z} w^{(2, ')}_{j_2,\ell_2} \\
& \qquad  {}  \qquad  {}   \qquad  {}  \qquad  {}  \qquad  {}   \qquad  {}    +   \sumetage {j_1+j_2=k} {\ell_1+ \ell_2= \ell } 2 \, w^1_{j_1,\ell_1}   w^3_{j_2,\ell_2} -  \frac 6  {z^3} w_{j_1,\ell_1}   \check   {w}_{j_2,\ell_2} \, ,
\end{split}
\eeq
\beq
\begin{split} \label{eqnl3} & F^{{\rm
nl}, 3}_{k,\ell}=  \sumetage {j_1+j_2+j_3=k} {\ell_1+ \ell_2+\ell_3= \ell }
    2 \, w^1_{j_1,\ell_1} w^3_{j_2,\ell_2} (w_{j_3,\ell_3} )_{z} -(w_{j_1,\ell_1} )_{z}(w_{j_2,\ell_2} )_{z} w^{(2, ')}_{j_3,\ell_3}
     \\
& \qquad  {}   +3 \sumetage {j_1+j_2+j_3=k} {\ell_1+ \ell_2+\ell_3= \ell }\big((w_{j_1,\ell_1} )_{z} (w_{j_2,\ell_2} )_{z}-  w^1_{j_1,\ell_1}  w^1_{j_2,\ell_2}\big) 
\Big(\frac { \check   {w}_{j_3,\ell_3} }   {z^2} +\frac {(w_{j_3,\ell_3})_{z}   }   {z}  \Big) \, ,
\end{split}
\eeq
and
\beq
\begin{split} \label{eqnl4} & F^{{\rm
nl}, 4}_{k,\ell}= -  \sumetage {j_1+j_2=k} {\ell_1+ \ell_2= \ell }  A^{0}_{j_1,\ell_1}(w_{j_2,\ell_2} )_{zz} \, .
\end{split}
\eeq
 For our purpose, it will be useful to point out  that   according to  \eqref{eqregion22bis}, one also has:   \beq \label{eqnlbis}
F^{{\rm
nl}}_{k,\ell}= {\wt F}^{{\rm
nl}, 1}_{k,\ell} + {\wt F}^{{\rm
nl}, 2}_{k,\ell} + {\wt F}^{{\rm
nl}, 3}_{k,\ell}\, ,\eeq  
where 
\beq \label{eqnl1bis} {\wt F}^{{\rm
nl}, 1}_{k,\ell}= -2 {\wt w}^2_{k,\ell}   ,\eeq
\beq
\begin{split} \label{eqnl2bis} & {\wt F}^{{\rm
nl}, 2}_{k,\ell}=   \sumetage {j_1+j_2=k} {\ell_1+ \ell_2= \ell }    6 \,  (w_{j_1,\ell_1} )_{z} \Big( \frac 1  {z}(w_{j_2,\ell_2} )_{z} + \frac 1  {z^2} \check   {w}_{j_2,\ell_2} \Big) - 2 \,(w_{j_1,\ell_1} )_{z} w^2_{j_2,\ell_2}\\
& \qquad  {}  \qquad  {} \qquad  {}  \qquad  {} \qquad  {}+ \sumetage {j_1+j_2=k} {\ell_1+ \ell_2= \ell }      2  \,  w^1_{j_1,\ell_1}   (w^1_{j_2,\ell_2})_{z}  -  \frac 6  {z^3}  w_{j_1,\ell_1}   \check   {w}_{j_2,\ell_2}  \, ,
\end{split}
\eeq
\beq
\begin{split} \label{eqnl3bis} & {\wt F}^{{\rm
nl}, 3}_{k,\ell}=    -\sumetage {j_1+j_2+j_3=k} {\ell_1+ \ell_2+\ell_3= \ell } (w_{j_1,\ell_1} )_{z}(w_{j_2,\ell_2} )_{z} w^2_{j_3,\ell_3} +   w^1_{j_1,\ell_1} w^1_{j_2,\ell_2} (w_{j_3,\ell_3} )_{zz}\\
& \qquad  {}\qquad  {} \qquad  {}\qquad  {} \qquad  {}\qquad  {}  \qquad  {} + 2 \sumetage {j_1+j_2+j_3=k} {\ell_1+ \ell_2+\ell_3= \ell } w^1_{j_1,\ell_1} (w^1_{j_2,\ell_2})_{z} (w_{j_3,\ell_3} )_{z}  \\
& \qquad  {}   +3 \sumetage {j_1+j_2+j_3=k} {\ell_1+ \ell_2+\ell_3= \ell }\big((w_{j_1,\ell_1} )_{z} (w_{j_2,\ell_2} )_{z}-  w^1_{j_1,\ell_1}  w^1_{j_2,\ell_2}\big) 
\Big(\frac { \check   {w}_{j_3,\ell_3} }   {z^2} +\frac {(w_{j_3,\ell_3})_{z}   }   {z}  \Big) \, \cdot
\end{split}
\eeq

\subsection{Analysis of the vector functions $\cW_k$} \label{secvect}

\subsubsection{Study of the linear  system $\cS_{k}$}\label{system} 
In order to determine successively the solutions $ w_{k,\ell} $ of the recurrent system \eqref {reckl},  let us under the above notations,  start by investigating the homogeneous equation: 
\begin{equation}\label{homsyst} \cS_{k} \,  X  = 0 \, .\end{equation}
We infer that  the   following lemma holds:   \begin{lemma}
\label {basis}
{\sl For $j$ in $\ds \big\{0, \cdots, \ell(k)  \big\}$, define $(f^{j, \pm}_{k,\ell})_{0 \leq \ell \leq \ell(k)}$ by 
\begin{eqnarray}\label{systbasis}  f^{j, \pm}_{k,\ell}(z) &= &   \begin{pmatrix} j  \\ \ell \end{pmatrix} \Big( \log  \Big|\frac 1 {\sqrt{2}}  \pm z \Big|\Big) ^{j-\ell}  \,\frac {\big|  \frac 1 {\sqrt{2}}  \pm z \big|^{\alpha(\nu, k)}} {z^3}    \, \virgp  \\f^{j, \pm}_{k,\ell} &= &0  \, , \, \, \mbox{for} \, \, j+1 \leq \ell \leq \ell(k)\, , \nonumber \end{eqnarray}
where $\alpha(\nu, k) = \nu k+ 4$.

\medskip \noindent Then denoting by $$
f^{j,\pm}_{k}=
 \left(
\begin{array}{ccccccccc}
f^{j,\pm}_{k,0} \\
\vdots \\
f^{j,\pm}_{k,j} \\
0 \\
 \vdots \\
0
\end{array}
\right) \, ,  
$$
the   vector functions $(f^{j, \pm}_k)_{0 \leq j \leq \ell(k)}$ constitute  a basis of  the  homogeneous   equation \eqref{homsyst}, on  the intervals~$\ds \big]0 \virgp  \frac 1 {\sqrt{2}}\big[$ and $\ds \big]\frac 1 {\sqrt{2}} \virgp \, \infty\big[\cdot$
}
\end{lemma}

\medbreak

\begin{proof} Consider   the linearization of \eqref {eq:NW}  around  $\rho$: 
\begin{equation} \label {linrho}   2 \, v_{tt}  - l_{\rho}  v = 0 \, ,\end{equation}
where 
\begin{equation} \label {deflinrho} l_{\rho} = \partial^2_{\rho}+ 6 \Big( \frac {\partial_{\rho}} {\rho}   + \frac 1 {\rho^2}  \Big)\cdot\end{equation}
Writing $$v(t,\rho)=  t w(t,z) \with z= \frac \rho t \, \virgp$$
we clearly get  under notation \eqref{lineq}
$$ Lw=0 .$$ Observe also that \eqref {linrho} is equivalent to \begin{equation} \label {eq:wave4} 2  (\rho^3v)_{tt} - (\rho^3v)_{\rho\rho}= 0 \,.\end{equation} Set  
$$ G(t, z)= t^{\nu k+1 }    \,\big(\log t + \log \big|  \frac 1 {\sqrt{2}}  \pm z \big|\big)^j \, \frac {\big|  \frac 1 {\sqrt{2}}  \pm z \big|^{\alpha(\nu, k)}} {z^3}\, \cdot $$
Since
$$ G(t, z) =   \big( \log \big| \frac t {\sqrt{2}}   \pm    \rho\big|\big)^j \, \frac {\big| \frac t {\sqrt{2}}   \pm    \rho\big|^{\alpha(\nu, k)}} {\rho^3} = \frac {F\big(\big| \frac t {\sqrt{2}}   \pm    \rho\big|\big)} {\rho^3} \, \virgp$$
for some function $F$, we infer   that  $G$ satisfies$$ L (t^{-1}	\,G) =0 \, .$$This   implies that 
$$ L \Bigl(t^{\nu k} \, \big(\log t + \log \big|  \frac 1 {\sqrt{2}}  \pm z \big|\big)^j\, \frac {\big|  \frac 1 {\sqrt{2}}  \pm z \big|^{\alpha(\nu, k)}} {z^3}\Bigr)=0\,.$$
Since  
\begin{equation} \label{formulause} t^{\nu k}    \,\big(\log t + \log \big|  \frac 1 {\sqrt{2}}  \pm z \big|\big)^j\, \frac {\big|  \frac 1 {\sqrt{2}}  \pm z \big|^{\alpha(\nu, k)}} {z^3}= t^{\nu k} \, \sum^{ j}_{\ell=0}   \,\big(\log t\big)^\ell f^{j, \pm}_{k,\ell}(z)\, \virgp \end{equation}we obtain  the result, recalling that 
$$ L \Bigl(t^{\nu k} \, \sum^{ j}_{\ell=0}   \,\big(\log t\big)^\ell f^{j, \pm}_{k,\ell}(z)\Bigr) = 0 \Leftrightarrow  \cS_{k} f^{j, \pm}_{k} = 0 \, .$$

\end{proof}  

\bigbreak

\begin{remark} \label{remself} 
{\sl   Note that in view of the above lemma,  the homogeneous equation $${\wt {\mathcal
L}}_{k} f= 0 $$    admits the following   basis:
\beq\label{basisk} \quad  \left\{
\begin{array}{l}
\ds f^{ 0, +}_{k,0} (z) = \frac {\big(\frac 1 {\sqrt{2}}+z \big)^{\alpha(\nu, k)}} {z^3}  \virgp \\
\ds  f^{ 0, -}_{k,0} (z) =  \frac {\big| \frac 1 {\sqrt{2}} -z \big|^{\alpha(\nu, k)}} {z^3}\, \cdot
\end{array}
\right.\eeq
 } 
\end{remark}

\bigbreak

Before concluding this section, let us collect some useful properties about the elements of the basis $(f^{j,\pm}_{k})_{0 \leq j \leq  \ell(k)}$ given above.
\begin{lemma}
\label {propbasis}
{\sl  Under the above notations,  the following asymptotic expansions hold \begin{equation} \label{rel} [\nu k+ \Lambda]
  f^{j, \pm}_{k,\ell}(z)+ (\ell+ 1) f^{j, \pm}_{k,\ell+ 1}(z)= z^{ \nu k}    \, 
\sum_{ 0 \leq \alpha \leq j-\ell} \, \sum_{ p \in  \N} \,  \gamma^k_{ p,  \alpha} \big( \log z \big) ^{\alpha}  \,  z ^{-p}  \,, \, \,  \mbox{as} \, \,  z \to \infty  \, ,\end{equation}
\begin{equation} \begin{split}\label{relbis}  &\quad   {} \quad {}          [z^2\,  \partial^2_z + \nu k  ( \nu k +1 - 2 z \partial_z ) ]\,
  f^{j, \pm}_{k,\ell}(z) + (\ell+ 1)\,  [2\nu k+1 - 2 z \partial_z] \, f^{j, \pm}_{k,\ell+ 1}(z) \\ & 
   \quad   {}\quad   {}  \quad   {} +(\ell+1) (\ell+2) \, f^{j, \pm}_{k,\ell+ 2}(z)=  z^{ \nu k-1}    \, 
\sum_{ 0 \leq \alpha \leq j-\ell} \, \sum_{ p \in  \N} \,  {\hat \gamma} ^k_{ p,  \alpha} \big( \log z \big) ^{\alpha}  \,  z ^{-p}  \,, \, \,  \mbox{as} \, \,  z \to \infty  \, ,  \end{split}\end{equation}
 for any integer  $k \geq 3$ and all $j, \ell$ in  $\ds \big\{0, \cdots, \ell(k)  \big\}$, respectively for some constants $\gamma ^k_{ p,  \alpha}$ and ${\hat \gamma}^k_{ p,  \alpha}$.  
 }
\end{lemma}

\medbreak

\begin{proof} 
In view of   Formula \eqref{formulause}, we have   for large $\rho$ \begin{equation} \label{transf1}  \big(  \log \big( \rho \pm \frac t {\sqrt{2}}  \big)\big)^j\, \frac {\big( \rho \pm \frac t {\sqrt{2}}  \big)^{\alpha(\nu, k)}} { \rho^3}= t^{\nu k+1} \, \sum^{ j}_{\ell=0}   \,\big(\log t\big)^\ell f^{j, \pm}_{k,\ell}\big(\frac  \rho t\big)\, \cdot\end{equation}
Therefore taking the derivative of the above  identity with respect to $t$, we deduce that
\beq \begin{split} \label{formproof} &   \frac 1{\sqrt{2}} \Big(j  \big(\log \big( \rho \pm \frac t {\sqrt{2}}  \big)\big)^{j-1} \frac {\big( \rho \pm \frac t {\sqrt{2}}  \big)^{\alpha(\nu, k)-1}} { \rho^3}
  + \alpha(\nu, k)(\log \big( \rho \pm \frac t {\sqrt{2}}  \big)\big)^{j} \frac {\big( \rho \pm \frac t {\sqrt{2}}  \big)^{\alpha(\nu, k)-1}} {\rho^3}\Big)\\
&   =  t^{\nu k} \, \sum^{ j}_{\ell=0}   \,\big(\log t\big)^\ell\Big( (\nu k+ \Lambda)f^{j, \pm}_{k,\ell} + (\ell+1)f^{j, \pm}_{k,\ell+1}\Big)\big(\frac  \rho t\big)\,. 
\end{split} \eeq
Performing the change of variables  $\ds z= \frac  \rho t \virgp$ we infer that 
\beq \begin{split} \label{formproof*} &   \frac  { t^{\nu k} } {\sqrt{2}} \frac {\big( z \pm   \frac 1{\sqrt{2}}    \big)^{\alpha(\nu, k)-1}} {z^3} \Big(j \big( \log t + \log \big( z \pm   \frac 1{\sqrt{2}}    \big)\big)^{j-1} 
  + \alpha(\nu, k)\big( \log t + \log \big( z \pm   \frac 1{\sqrt{2}}    \big)\big)^{j} \Big)\\
&   =  t^{\nu k} \, \sum^{ j}_{\ell=0}   \,\big(\log t\big)^\ell\big( (\nu k+ \Lambda)f^{j, \pm}_{k,\ell} + (\ell+1)f^{j, \pm}_{k,\ell+1}\big)(z)\,, 
\end{split} \eeq
which concludes the proof of \eqref{rel}. 

\medskip  \noindent
Along the same lines  taking   the derivative with respect to $t$ of \eqref{formproof} ensures Identity \eqref{relbis}, which ends the proof of the lemma.

\end{proof}

\bigbreak

\subsubsection{Study of the functions $w_{k,\ell}$}\label{systemstudy} 
 The goal of this  paragraph is   to   prove    by induction that the system \eqref  {reckl} admits a  solution satisfying the matching conditions \eqref{exporigkl} coming out from the inner region.

\medskip For that purpose, let us start by  the following usefull lemma which stems from standard techniques of ordinary differential equations.  For the sake of completeness and the convenience of the reader, we outline   its proof in  Appendix \ref {ap:genLres2}. 
\begin{lemma}
\label {genLres2}
{\sl Under the above notations \footnote{ and again with the convention that the sum is null if it is over an empty set.}, the following properties hold: 
\begin{itemize} \item For any  function $g$ in  $ \cC^ {\infty}(\R^*_+)$,  the equation   
$$
{\wt {\mathcal
L}}_{k} f= g
$$ 
admits a unique solution $f $ in $\cC^ {\infty}(\R^*_+)$ satisfying $\ds f \Big(\frac 1 {\sqrt{2}}\Big)= 0\, \cdot$ 

\item For any  function $h$ in  $ \ds \cC^ {\infty}\big(]0,\frac 1 {\sqrt{2}}]\big)$,  any  $\gamma > 0$, and 
any integer $q$,  the equation    
\begin{equation}
\label {seceq} 
{\wt {\mathcal
L}}_{k} f (z)= \big(\frac 1 {\sqrt{2}} -z\big)^{\gamma}   \big(\log \big(\frac 1 {\sqrt{2}} -z\big) \big)^q  h(z)
\end{equation}
admits a unique solution $f $ of the form: 
$$ f(z)= \big(\frac 1 {\sqrt{2}} -z\big)^{\gamma+ 1}  \sum_{0  \leq  \ell  \leq q } \big(\log \big( \frac 1 {\sqrt{2}} -z\big)\big)^ \ell   \,  { h }_ \ell (z)\,  ,$$
where  for all $0  \leq  \ell  \leq q$, ${ h }_ \ell $ is a function in  $ \ds \cC^ {\infty}\big(]0,\frac 1 {\sqrt{2}}]\big)\virgp$ 
provided that the exponent $\gamma$   satisfies \begin{equation}
\label {condgamma}    \nu k +4 - \gamma  \notin \N^*\, . \end{equation}
\item  Let $g$ be  a  function in $ \ds \cC^ {\infty}\big(]0,\frac 1 {\sqrt{2}}[\big)$ with an asymptotic expansion at
~$0$ of the form: $$g(z)=   (\log z)^{\alpha_0} \,  \sum_{\beta \geq   \beta _0}  \, g _{\beta}   \, z^{\beta -2}   \, ,$$  for some integers  $\alpha_0, \beta _0$,  then any solution $f$ of the  equation 
\begin{equation}
\label {4eq} 
{\wt {\mathcal
L}}_{k} f  = g
\end{equation}  belongs to $ \ds \cC^ {\infty}\big(]0,\frac 1 {\sqrt{2}}[\big)$ and  admits for $z$ close to $0$  an  asymptotic expansion of the type: 
$$ f(z)= \sum_{\beta \geq -3}   f _{0,\beta}    \,  z^\beta+  \sum_{1 \leq \alpha \leq \alpha_0} \sum_{\beta \geq \beta _0 }   f _{\alpha,\beta}   (\log z)^{\alpha}  z^\beta  
   \, , $$ 
   in the case when $\beta _0 \geq -1 $, and of the type 
   $$ \quad \quad f(z)= \sum_{\beta \geq {\rm
\min}(\beta _0, -3)}   f _{0,\beta}     z^\beta+  \sum_{1 \leq \alpha \leq \alpha_0} \sum_{\beta \geq \beta _0 }   f _{\alpha,\beta}   (\log z)^{\alpha}  z^\beta  
  +    \sum_{\beta \geq \max(\beta _0, -3) }   f _{\alpha_0+1,\beta}  (\log z)^{\alpha_0+1}   z^\beta \,, $$ 
   in the case when $\beta _0 \leq -2 $.
\item  If $g$ denotes   a  function belonging to  $ \ds \cC^ {\infty}\big(]\frac 1 {\sqrt{2}}\virgp \infty[\big)$ and admitting  at infinity  an    asymptotic expansion of the form: $$g(z)= \sum_{0 \leq \alpha \leq \alpha_0 } \sum_{p \in  \N} \,  {\hat g}  _{\alpha,p}  \,  (\log z)^{\alpha} \, z^{A-p}  \, , $$   for some real  $A< \nu k$ and some integer $\alpha_0$,     then   the  equation   
\begin{equation}
\label {5eq} 
{\wt {\mathcal
L}}_{k} f  = g
\end{equation}  
admits a unique solution $f$  in~$ \ds \cC^ {\infty}\big(]\frac 1 {\sqrt{2}} \virgp \infty[\big)$ such that    
$$ f(z)=  \sum_{0 \leq \alpha \leq \alpha_0 } \sum_{p \in  \N}  {\hat f}^k _{\alpha,p}  (\log z)^{\alpha} \, \, z^{A-p} \,   ,  z  \to \infty \,  .$$

 \end{itemize} 
  }
\end{lemma}

\bigbreak

The key result of this paragraph is the following  proposition:
\begin{proposition}
\label {STKL}
{\sl   Under the above notations, the following properties hold:

\smallskip \noindent

\begin{enumerate}  \item \underline{Existence} 

\smallskip \noindent The system \eqref {eq:W_kl} admits a   solution 
$(w_{k,\ell}, \lam_{k,\ell})_{k \geq 3,  0 \leq \ell  \leq \ell(k)} 
 $  such that for  any integer~$k \geq 3$ and any~$\ds \ell \in \big\{0, \cdots, \ell(k)  \big\}$, the function $w_{k,\ell}$ belongs to   $ \cC^ {[\alpha(\nu, k)]}\big(\R^*_+\big)\cap 
\cC^ {\infty}\big(\R^*_+\setminus \{\frac 1 {\sqrt{2}}\}\big)  \, \virgp$ 
and has the form \footnote{ Here and below, the notation  ${}^{  {\rm
reg}}$ means that the corresponding function belongs to $\cC^ {\infty}(\R^*_+) $.}:
\beq
\label{formwkl}
\begin{split}
& \qquad   w_{k,\ell}  (z )= a^{\rm
reg}_{k,\ell} (z)+  \big(\frac 1 {\sqrt{2}} - z\big)^{k \nu + 4}\,  
\sumetage {0   \leq \alpha \leq \frac  {k -3} 2- \ell}  \,  b^{\rm
reg}_{k,\ell, \alpha }(z) 
\,\big(\log \big(\frac 1 {\sqrt{2}} -z\big) \big)^\alpha   
\,\chi_{]0,   \frac 1 {\sqrt{2}}]}(z )\\
& \qquad  {}  \qquad {}  + \sumetage { 3 \leq \beta \leq k-3}  {0   \leq \alpha \leq \frac  {k -6} 2- \ell}  \, \,  b^{\rm
reg}_{k,\ell, \alpha, \beta }(z)  \big(\frac 1 {\sqrt{2}} - z\big)^{\beta \nu + 4}\,  \big(\log \big(\frac 1 {\sqrt{2}} -z\big) \big)^\alpha    \, \chi_{]0,   \frac 1 {\sqrt{2}}]}(z ) \,  \virgp
\end{split}
\eeq
 where the function $\chi $ denotes  the characteristic function, namely
$$\left\{
\begin{array}{l}
\ds \chi_{]0,   \frac 1 {\sqrt{2}}]}(z )=1\, \,\, \mbox{for} \,\, z \leq \frac 1 {\sqrt{2}} \andf\\
\ds \chi_{]0,   \frac 1 {\sqrt{2}}]}(z )= 0  \,\, \mbox{for} \,\, z > \frac 1 {\sqrt{2}}\, \cdot 
\end{array}
\right.$$
 In addition,  the following asymptotics hold:
\beq
\label {0w_kl}
w_{k,\ell} (z)=   \sumetage {0   \leq \alpha \leq \frac  {k -3} 2- \ell}  { \beta \geq 1-k + 2 (\alpha   +\ell)} d^{k,\ell} _{\alpha,\beta}  \,(\log z)^{\alpha}  \, z^\beta\,, \, \,  \mbox{as} \, \,  z \to 0  \, ,   \eeq 
with 
\beq
\label {condmatch}d^{k,\ell} _{0, -2} =  c^{k,\ell} _{0, -2}(\lam) \, , \, \,   d^{k,\ell} _{0, -3} =  c^{k,\ell} _{0, -3}(\lam) \, ,\eeq 
where $c^{k,\ell} _{0, \beta}(\lam)$ are the coefficients related to the matching conditions coming out from the inner region involved in Formula \eqref{formmatch22}.

\medskip \noindent Moreover for $\ds z> \frac 1 {\sqrt{2}}\virgp$ $w_{k,\ell}$ can be splitted into two parts as follows:
\beq \label{formulaw_kl}
w_{k,\ell} (z)= w^{{\rm
nl}}_{k,\ell} + w^{{\rm
lin}}_{k,\ell}   \, ,\eeq
where  the nonlinear part $w^{{\rm
nl}}_{k,\ell}$ is  null if  $k < 6 $ or $\ds \ell > \frac {k-6} {2}\virgp$ and has in all other cases, as~$z$ tends to infinity, an asymptotic expansion of the form 
\beq  
 \label{nlw_kl} w^{{\rm
nl}}_{k,\ell}(z)=   \sumetage {3  \leq \beta \leq  k -3  }  {0   \leq \alpha \leq \frac  {k -6} 2-\ell, \, p  \in \N}    {\hat d}^{k,\ell} _{\alpha,\beta,p} \,(\log z)^{\alpha} \,  z^{\beta\, \nu  +1-p} + z^{ \nu k +1} \sum_{0   \leq \alpha \leq \frac  {k -3} 2-\ell, \, p  \geq 2}    {\hat d}^{k,\ell} _{\alpha,k,p} \,(\log z)^{\alpha} \,  z^{-p}\, ,  \eeq for some constants ${\hat d}^{k,\ell} _{\alpha,\beta,p}$,   and where the linear part $w^{{\rm
lin}}_{k,\ell}$ is given by \beq  
 \label{linw_kl} w^{{\rm
lin}}_{k,\ell}(z)=   \sum_{0   \leq j \leq \ell(k) } \,   \alpha^{j,+ }_{k}  \,  f^{j,+ }_{k, \ell}+  \alpha^{j,- }_{k}  \,  f^{j,- }_{k, \ell} \,  , \eeq
for some constants $ \alpha^{j,\pm}_{k} $, where  $f^{j,\pm}_{k} =(f^{j,\pm}_{k, \ell})_{0   \leq j \leq \ell(k)}$ are  the solutions of the homogeneous equation \eqref{homsyst} introduced in Lemma \ref {basis}.

\bigskip

 \item   \underline{Uniqueness} 

\smallskip \noindent
 Let $(\lam_{k,\ell})_{k \geq 3,  0 \leq \ell  \leq \ell(k)}$ be fixed,  and let   $({ w}^0_{k,\ell})_{3 \leq k \leq M, 0 \leq \ell  \leq \ell(k)} $ and  $({ w}^1_{k,\ell})_{3 \leq k \leq M, 0 \leq \ell  \leq \ell(k)} $ be two solutions of 
 \beq  
 \label{eqsyst} {\wt {\mathcal
L}}_{k} { w}_{k,\ell}= F_{k,\ell} (\lam; w)\, , 3 \leq k \leq M \, , \eeq  
defined and $\cC^ {\infty}$ in a neighborhood of $0$, with ${ w}^i_{5,1} \equiv 0$  for $i\in \{0,1\}$, and  which have an asymptotic expansion of the form \eqref {0w_kl}, as $z$ tends to $0$: 
 $$
w^i_{k,\ell} (z)=   \sumetage {0   \leq \alpha \leq \frac  {k -3} 2- \ell}  { \beta \geq 1-k + 2 (\alpha   +\ell)} d^{k,\ell, i} _{\alpha,\beta}  \,(\log z)^{\alpha}  \, z^\beta\,   .   $$
If 
 \beq  
 \label{cond0}d^{k,\ell, 0} _{0, -2} =  d^{k,\ell, 1} _{0, -2}\, ,  \, \,   d^{k,\ell, 0} _{0, -3}=  d^{k,\ell, 1} _{0, -3} \, ,\,\eeq 
then $w^0_{k,\ell}= w^1_{k,\ell}$, for all $3 \leq k \leq M$ and all  $0 \leq \ell  \leq \ell(k)$.
  
\medskip \noindent Similarly, if $({ w}^0_{k,\ell})_{3 \leq k \leq M, 0 \leq \ell  \leq \ell(k)} $ and  $({ w}^1_{k,\ell})_{3 \leq k \leq M, 0 \leq \ell  \leq \ell(k)} $ are two solutions of the equation 
 \eqref{eqsyst} defined and $\cC^ {\infty}$ around  $+\infty$, with ${ w}^i_{5,1} \equiv 0$  for $i\in \{0,1\}$,  and  which   satisfy as~$z$ tends to infinity:
\beq
\begin{split} &{ w}^i_{k,\ell}=   \sum_{0   \leq j \leq \ell(k) } \,   \alpha^{j,+, i }_{k}  \,  f^{j,+ }_{k, \ell}+  \alpha^{j,-, i }_{k}  \,  f^{j,- }_{k, \ell} \\
& \qquad  {}    + \sumetage {3  \leq \beta \leq  k -3  }  {0   \leq \alpha \leq \frac  {k -6} 2-\ell, \, p  \in \N}    {\hat d}^{k,\ell, i} _{\alpha,\beta,p} \,(\log z)^{\alpha} \,  z^{\beta\, \nu  +1-p} + z^{ \nu k +1} \sum_{0   \leq \alpha \leq \frac  {k -6} 2-\ell, \, p  \geq 2}    {\hat d}^{k,\ell, i} _{\alpha,k,p} \,(\log z)^{\alpha} \,  z^{-p}  \, ,\end{split}\eeq
then \beq  
 \label{condinfty} \alpha^{j,\pm, 0 }_{k} =  \alpha^{j,\pm, 1 }_{k}\, , \, \forall \, 3 \leq k \leq M \andf  0   \leq j \leq \ell(k)\, ,\eeq 
  implies that  $w^0_{k,\ell}= w^1_{k,\ell}$, for all $3 \leq k \leq M$ and all  $0 \leq \ell  \leq \ell(k)$.

 \end{enumerate}  
  }
\end{proposition}

\bigbreak

\begin{remark}
{\sl  
 By virtue of  Lemma \ref {basis}
 and formulae \eqref{nlw_kl} and \eqref{linw_kl}, the functions $w_{k,\ell}$ admit an asymptotic expansion of the form: \beq
\label{w_klasyin}
\begin{split}
& \qquad   w_{k,\ell}  (z )=  \sumetage {3   \leq \beta \leq    k -3 }  {0   \leq \alpha \leq \frac  {k -6} 2-\ell,  \,  p  \in \N}    w_{k,\ell, \alpha,\beta,p} \,(\log z)^{\alpha} \,  z^{\beta\, \nu  +1-p}\\
& \qquad  {}  \qquad {}    \qquad  {}  \qquad {}  \qquad  {}   \qquad  {}   +  z^{k\, \nu  +1}  \,  \sumetage {0   \leq \alpha \leq \frac  {k -3} 2-\ell}   {p  \in \N}  \, w_{k,\ell, \alpha ,p} \,(\log z)^{\alpha}\,  z^{ -p} \,, \, \,  \mbox{as} \, \,  z \to   \infty \,  \virgp  
\end{split}
\eeq
for some constants $w_{k,\ell, \alpha,\beta,p}$ and $w_{k,\ell, \alpha ,p}$, as $z$ tends to infinity. 
 }
\end{remark}

\bigbreak

 {\it Proof of Proposition {\rm\refer {STKL}}.} 
Let us start with the existence part of the proposition, and first consider the indexes~$k=3, 4$ and $5$. In view of the computations carried out in Section \ref{step2sub} (see Property \eqref{coefcase50}), we have in that case~$ w_{5,1}=0$  and 
\beq \label{case3}    {\wt {\mathcal L}}_{k} w_{k,0}=0 \, , \, \, k=3,4,5 \, .\eeq 
In view of Remark \ref{remself}, this   implies that \begin{equation} \label{indss} \quad  \left\{
\begin{array}{l}
\ds w_{k,0}  =a^k_{0,+}\,f^{ 0, +}_{k,0}(z) + a^k_{0,-}\, f^{ 0, -}_{k,0}(z)\, \,\, \mbox{for} \,\, z \leq \frac 1 {\sqrt{2}} \, \virgp  \\
\ds  w_{k,0}  = a^k_{0,+}\,f^{ 0, +}_{k,0}(z) \, \,\, \mbox{for} \,\, z > \frac 1 {\sqrt{2}}  \, \virgp \,\, k=3,4,5\,, 
\end{array}
\right. \end{equation} 
   where $a^3_{0,+}= - a^3_{0,-}$ and where $\ds \big\{ f^{ 0, +}_{k,0}, f^{ 0, -}_{k,0} \big\}$ denotes the basis of solutions associated to the operator ${\wt {\mathcal L}}_{k}$ given by  \eqref{basisk}. 
The coefficients $a^k_{0,\pm}$ are determined by \eqref {condmatch}:
    \begin{equation} \label{endsolcoeffirst}  \quad  \left\{
\begin{array}{l}
\ds     2 (3\nu  +4)  \Big(\frac 1 {\sqrt{2}}\Big)^{3 \nu+3}  a^3_{0,+}  =  c^{3,0} _{0, -2} \, , \\
\ds     \Big(\frac 1 {\sqrt{2}}\Big)^{\nu k+4} \big(a^k_{0,+} + a^k_{0,-} \big)=  c^{k,0} _{0, -3} \, , \\
\ds  (\nu k+4) \Big(\frac 1 {\sqrt{2}}\Big)^{\nu k+3}  \big(a^k_{0,+} -a^k_{0,-} \big) = c^{k,0} _{0, -2} \, \virgp \,\, k= 4,5\,. 
\end{array}
\right. \end{equation} 
Clearly the functions $w_{k,0}$, $k=3,4,5$, satisfy properties \eqref{formwkl}-\eqref{linw_kl}.

\medskip Let us now  consider the general case of any index $k \geq 6$. To this end, we shall proceed by induction assuming that, for any integer $3 \leq j \leq k-1$ and all $0 \leq \ell  \leq \ell(j)$, 
 $(w_{j,\ell}, \lam_{j,\ell})$    satisfies the conclusion of    part $(1)$ of Proposition \ref {STKL}.  
 
\medskip  

The first  step consists to establish the following lemma:  
 \begin{lemma}
\label {propsource}
{\sl  Assume that   $(w_{j,\ell}, \lam_{j,\ell})_{0 \leq \ell \leq \ell(j)}$ is a solution of the system \eqref {eq:W_kl} with  $3 \leq j \leq k-1$, which satisfies \eqref{formwkl}, \eqref {0w_kl}, \eqref{formulaw_kl}, \eqref{nlw_kl} and \eqref{linw_kl}. Then~$F^{{\rm
nl}}_{k,\ell}$   has the following form:
\beq\begin{split} \label{nonlindevf} & F^{{\rm
nl}}_{k,\ell}(z)= f^{  {\rm
reg}}_{k, \ell}(z) + \big(\frac 1 {\sqrt{2}} - z\big)^{k \nu + 6}\, \sum_{0   \leq \alpha \leq \frac  {k -6} 2- \ell} f^{  {\rm
reg}}_{k, \ell, \alpha} (z)\big(\log \big(\frac 1 {\sqrt{2}} -z\big) \big)^\alpha   \,  \chi_{]0,   \frac 1 {\sqrt{2}}]}(z )  \\
& \qquad     \qquad  \qquad  + \sumetage {3 \leq \beta \leq k-3} {0   \leq \alpha \leq \frac  {k -6} 2- \ell} f^{  {\rm
reg}}_{k, \ell, \beta, \alpha} (z) \,  \big(\frac 1 {\sqrt{2}} - z\big)^{\beta \nu + 3}\,   \big(\log \big(\frac 1 {\sqrt{2}} -z\big) \big)^\alpha  \, \chi_{]0,   \frac 1 {\sqrt{2}}]}(z )  \, , \end{split} \eeq
and  has   the following asymptotic expansions  respectively close to $0$ and at infinity: 
\beq \label{0F_kl} F^{{\rm
nl}}_{k,\ell}(z)=  \sumetage {0   \leq \alpha \leq \frac  {k -6} 2- \ell} { \beta \geq 1-k + 2 (\alpha   +\ell)} \wt f_{k,\ell, \alpha,\beta}  \,(\log z)^{\alpha} \, z^{\beta-2} 
    \, \virgp\eeq
    \beq\begin{split}  \label{infF_kl} & F^{{\rm
nl}}_{k,\ell}(z)=   z^{k \nu-1}  \sumetage { 0 \leq \alpha \leq \frac {k -6} 2-\ell}  { p \in  \N}   { \hat f} _{k,  \ell,   \alpha, p} \big( \log z \big) ^{\alpha}   z ^{-p}\\
& \qquad     \qquad  \qquad   \qquad     \qquad  \qquad+   \sumetage { 3 \leq \beta \leq  k -3}   { 0 \leq \alpha \leq \frac {k -6} 2-\ell, \, p \in  \N}     { \hat f} _{k,  \ell,   \alpha, \beta, p} \big( \log z \big) ^{\alpha}     z ^ {\nu \beta + 1-p} 
    \, \virgp \end{split}  \eeq
   where the coefficients $\wt f_{k,\ell, \alpha,\beta}$ $($resp. $ { \hat f} _{k,  \ell,   \alpha, p} $  and  $ { \hat f} _{k,  \ell,   \alpha, \beta, p}$$)$  are uniquely  determined in terms of the coefficients~$d^{j,\ell} _{\alpha,\beta}$  $($resp. $w^{k,\ell} _{\alpha,\beta,p}$$)$  involved in
~\eqref {0w_kl} $($resp. \eqref {w_klasyin}$)$.   
 }
\end{lemma}

\medbreak

 \begin{proof}
Let us first address  the behavior of  $F^{{\rm
nl}}_{k,\ell}$ near $z=0$ and at infinity.  To establish \eqref{0F_kl} and~\eqref {infF_kl}, we  will use  formulae \eqref{eqnlbis}-\eqref {eqnl3bis}  combining them with the corresponding  asymptotic of~$w^1_{j,\ell}$, $\wt w^1_{j,\ell}$, $w^2_{k,\ell}$, $\wt w^2_{j,\ell}$ and  $\check   {w}_{j,\ell}$  that we start to describe now. 

\medskip    Consider  $\wt w^1_{j,\ell}$. It follows from \eqref{formreg1kl} that if $w_{j,\ell}, 3 \leq j \leq k-1$, verify \eqref {0w_kl} and \eqref{w_klasyin}, then for any $6 \leq j \leq k+2$,  $\wt w^1_{j,\ell}$ admits the following asymptotic expansions as $z \to 0$ and $z\to \infty$:
\beq
\label{exp1tilde0}  \wt w^1_{j,\ell}(z)=  \sumetage {0 \leq \alpha \leq \frac {j -6} 2-\ell} { \beta \geq 4 -j +2 (\alpha+ \ell)}    { \wt w}^{1, 0} _{j,\ell,  \alpha, \beta} \big( \log z \big) ^{\alpha}     z ^ {  \beta }\, , \, \,  \mbox{as}  \, \,  z  \to 0 \, ,\eeq
\beq
\label{exp1tildeinfty}  \wt w^1_{j,\ell}(z)=  \sumetage {0 \leq \alpha \leq \frac {j -6} 2-\ell} {   3  \leq \beta \leq j-3, \, p \geq 0}      { \wt w}^{1, \infty} _{j,\ell,  \alpha, \beta, p} \big( \log z \big) ^{\alpha}     z ^ {  \beta \nu +1- p}\, , \, \,  \mbox{as}  \, \,  z  \to \infty \, .\eeq
Combining \eqref{formreg1kl} together with \eqref{0w_kl} and \eqref{exp1tilde0},  one sees that   the function $w^1_{j,\ell}$ has   asymptotic of the same form as $w_{j,\ell}$,  as $z$ tends to   $0$: 
\beq \label{exp10}
   w^1_{j,\ell}  (z) =     \sumetage {0 \leq \alpha \leq \frac {j -3} 2-\ell} { \beta \geq 1-j +2 (\alpha+ \ell)}   {  w}^{1, 0} _{j,\ell,  \alpha, \beta} \big( \log z \big) ^{\alpha}     z ^ { \beta }\, .  \eeq
Furthermore,  invoking  \eqref{formreg1kl}, \eqref{formulaw_kl},  \eqref{nlw_kl},  \eqref{linw_kl},   \eqref{exp1tildeinfty}
 and taking into account  \eqref{rel}, one obtains as $z  \to \infty $   \beq
\label{exp1infty}
   w^1_{j,\ell}  (z) = z^{j \nu} \sumetage {0 \leq \alpha \leq \frac {j -3} 2-\ell} {   p \in  \N}   {  w}^{1, \infty} _{j,\ell,  \alpha, j, p} \big( \log z \big) ^{\alpha}   z ^{-p} +   \sumetage {0 \leq \alpha \leq \frac {j -6} 2-\ell} {  3  \leq \beta \leq j-3, \,  p \in  \N}   {  w}^{1, \infty} _{j,\ell,  \alpha, \beta, p} \big( \log z \big) ^{\alpha}     z ^ {\nu \beta + 1-p}\, ,  \eeq
for any integer $3 \leq j \leq k-1$. 

\medskip   The function $\wt w^2_{j,\ell}$ can be analyzed along the same lines as $\wt w^1_{j,\ell}$. In particular, using Definition~\eqref{formreg2kl}, one can show that under the assumptions of Lemma \ref{propsource}, for any $6 \leq j \leq k+2$, $\wt w^2_{j,\ell}$ behaves in the same way as $\wt w^1_{j,\ell}$ when $z \to 0$ and $z\to \infty$, namely 
\beq
\label{exp2tilde0}  \wt w^2_{j,\ell}(z)=  \sumetage {0 \leq \alpha \leq \frac {j -6} 2-\ell} { \beta \geq 4 -j +2 (\alpha+ \ell)}    { \wt w}^{2, 0} _{j,\ell,  \alpha, \beta} \big( \log z \big) ^{\alpha}     z ^ {  \beta }\, , \, \,  \mbox{as}  \, \,  z  \to 0 \, ,\eeq
\beq
\label{exp2tildeinfty}  \wt w^2_{j,\ell}(z)=  \sumetage {0 \leq \alpha \leq \frac {j -6} 2-\ell} {  3  \leq \beta \leq j-3, \, p \geq 0}      { \wt w}^{2, \infty} _{j,\ell,  \alpha, \beta, p} \big( \log z \big) ^{\alpha}     z ^ {  \beta \nu +1- p}\, , \, \,  \mbox{as}  \, \,  z  \to \infty \, .\eeq
Combining \eqref{formreg2kl},  with \eqref{0w_kl},  \eqref{formulaw_kl}-\eqref{linw_kl}, \eqref{exp2tilde0},  \eqref{exp2tildeinfty} 
 and taking into account  \eqref{relbis}, we deduce,     as we have done for $w^1_{j,\ell}$, that  $w^2_{j,\ell}$ has the same form as  $ w_{j,\ell}$,  $ w^1_{j,\ell}$, as $z\to 0$ 
\beq \label{exp20}
   w^2_{j,\ell}  (z) =     \sumetage {0 \leq \alpha \leq \frac {j -3} 2-\ell} { \beta \geq 1-j +2 (\alpha+ \ell)}   {  w}^{2, 0} _{j,\ell,  \alpha, \beta} \big( \log z \big) ^{\alpha}     z ^ {  \beta  }\, , \eeq
 and as $z \to \infty$, one has:  \beq
\label{exp2infty}
   w^2_{j,\ell}  (z) = z^{j \nu-1} \sumetage {0 \leq \alpha \leq \frac {j -3} 2-\ell} {   p \in  \N}   {  w}^{2, \infty} _{j,\ell,  \alpha, j, p} \big( \log z \big) ^{\alpha}   z ^{-p} +   \sumetage {0 \leq \alpha \leq \frac {j -6} 2-\ell} {  3  \leq \beta \leq j-3, \,  p \in  \N}   {  w}^{2, \infty} _{j,\ell,  \alpha, \beta, p} \big( \log z \big) ^{\alpha}     z ^ {\nu \beta + 1-p}\, , \eeq
   for all $3 \leq j \leq k-1$.  
 
\medskip  
Next we address  $\check   {w}_{j,\ell}$. Writing
 \beq
\label{exptchech} \check   {w}_{j,\ell}= \sum_{   p \geq 1}  \sumetage {j_1+\cdots+j_p=j} { \ell_1+\cdots+\ell_p= \ell}  (-1)^{p-1} z^ {1-p} {w}_{j_1,\ell_1}\cdots {w}_{j_p,\ell_p}\, , \eeq
it is easy to check that if $ {w}_{j,\ell}$, $3 \leq j \leq k-1$, verify \eqref {0w_kl}, \eqref{w_klasyin} then the same is true for~$\check   {w}_{j,\ell}$,~$3 \leq j \leq k-1$, namely the functions $\check   {w}_{j,\ell}$ admit asymptotic expansions of the form: 
\beq \label{tch0}
  \check   {w}_{j,\ell} (z) =     \sumetage {0 \leq \alpha \leq \frac {j -3} 2-\ell} { \beta \geq 1-j +2 (\alpha+ \ell)}   { \check   {w}}^{ 0} _{j,\ell,  \alpha, \beta} \big( \log z \big) ^{\alpha}     z ^ {  \beta }\, , \, \,  \mbox{as}  \, \,  z  \to 0 \, ,\ \eeq
    \beq
\label{tchinfty}
    \check   {w}_{j,\ell} (z) = z^{j \nu+1} \sumetage {0 \leq \alpha \leq \frac {j -3} 2-\ell} {   p \in  \N}    \check   {w}^{\infty} _{j,\ell,  \alpha, j, p} \big( \log z \big) ^{\alpha}   z ^{-p} +   \sumetage {0 \leq \alpha \leq \frac {j -6} 2-\ell} {  3  \leq \beta \leq j-3,   p \in  \N}    \check   {w}^{\infty} _{j,\ell,  \alpha, \beta, p} \big( \log z \big) ^{\alpha}     z ^ {\nu \beta + 1-p}\, ,    \eeq
as $ z  \to \infty $.

\medskip Combining \eqref{eqnl}-\eqref{eqnl4} with \eqref{exp10}, \eqref{exp1infty},  \eqref{exp2tilde0},  \eqref{exp2tildeinfty}, \eqref{exp20}, \eqref{exp2infty},  \eqref{tch0} and \eqref{tchinfty}, we obtain \eqref{0F_kl}  and \eqref{infF_kl}.

\medskip To end the proof of the lemma, it remains to establish \eqref{nonlindevf}. To this end, we will use  the representations \eqref{eqnl}-\eqref{eqnl4}.

\medskip Start with  $F^{{\rm
nl}, 1}_{k,\ell}$ defined by \eqref{eqnl1}. It stems from the definition of  $\wt  w^{(2, ')}_{j,\ell}$ given by \eqref{formreg2tildekl}
 that  for any $6 \leq j \leq k+2$,   $\wt  w^{(2, ')}_{j,\ell}$ assumes the form: 
\beq\begin{split} \label{firstpart} &  f^{  {\rm
reg}}_{j, \ell}(z)  + \sumetage {3 \leq \beta \leq j-3} {0   \leq \alpha \leq \frac  {j -6} 2- \ell} \big(\frac 1 {\sqrt{2}} - z\big)^{\beta \nu + 3}\,  \big(\log \big(\frac 1 {\sqrt{2}} -z\big) \big)^\alpha \,   h^{{\rm
reg}  }_{\alpha, \beta} (z) \, \chi_{]0,   \frac 1 {\sqrt{2}}]}(z )  \, , \end{split} \eeq
 which means that \eqref{nonlindevf} holds for $F^{{\rm
nl}, 1}_{k,\ell}$. 

\medskip Next consider   $F^{{\rm
nl}, i}_{k,\ell}$, $i=2,3$, respectively defined by  \eqref{eqnl2} and \eqref{eqnl3}.  In view of \eqref{formreg1kl} and~\eqref{formreg3kl}, we deduce that  $\wt w^1_{j,\ell}$ and $\wt w^3_{j,\ell}$  have the form \eqref{firstpart} for any $6 \leq j \leq k+2$,  and therefore the functions $ w^1_{j,\ell}$ and $ w^3_{j,\ell}$   can be written in the following way\beq\begin{split} \label{formulasimp} &   f^{ {\rm
reg}}_{j, \ell}(z) + \big(\frac 1 {\sqrt{2}} - z\big)^{j \nu + 3}\, \sum_{0   \leq \alpha \leq \frac  {j -3} 2- \ell} \big(\log \big(\frac 1 {\sqrt{2}} -z\big) \big)^\alpha \,  h^{ {\rm
reg}}_{j, \ell,  \alpha} (z) \,  \chi_{]0,   \frac 1 {\sqrt{2}}]}(z )  \\
& \qquad     \qquad  \qquad  + \sumetage {3 \leq \beta \leq j-3} {0   \leq \alpha \leq \frac  {j -6} 2- \ell} \big(\frac 1 {\sqrt{2}} - z\big)^{\beta \nu + 3}\,  \big(\log \big(\frac 1 {\sqrt{2}} -z\big) \big)^\alpha \,    h^{ {\rm
reg}}_{j, \ell, \alpha, \beta} (z) \, \chi_{]0,   \frac 1 {\sqrt{2}}]}(z )  \, . \end{split} \eeq
Similarly, by \eqref{formreg2tildekl} the same is true for 
$w^{(2, ')}_{j,\ell}$.  Finally using \eqref{exptchech},  one can easily check that the functions~$\check   {w}_{j,\ell}$, ~$3 \leq j \leq k-1$, have the form \eqref{formwkl}, which can be viewed as a particular case of~\eqref{formulasimp}.

\medskip Since all the functions involved in  \eqref{eqnl2} and \eqref{eqnl3} have the form \eqref{formulasimp}, one   easily deduces that  $F^{{\rm
nl}, i}_{k,\ell}$, $i=2,3$, verify \eqref{nonlindevf}.

\medskip Now consider   $F^{{\rm
nl}, 4}_{k,\ell}$ given by \eqref{eqnl4}.  It follows from the definition   of  $A_0$ (see \eqref{defA0})   that,
 for all  $3\leq  j \leq k-1$,  the function $A^0_{j, \ell}$ admits  the same  form  as $w_{j, \ell}$:
\beq\begin{split} \label{Adev} & A^0_{j, \ell}(z)=A^{ {\rm
reg}}_{j, \ell}(z)+ \big(\frac 1 {\sqrt{2}} - z\big)^{j \nu + 4}\, \sum_{0   \leq \alpha \leq \frac  {j -3} 2- \ell} \big(\log \big(\frac 1 {\sqrt{2}} -z\big) \big)^\alpha \,  A^{{\rm
reg}} _{j,\ell, \alpha }(z)  \,   \chi_{]0,   \frac 1 {\sqrt{2}}]}(z ) \\
& \qquad     \qquad  \qquad  + \sumetage {3 \leq \beta \leq j-3} {0   \leq \alpha \leq \frac  {j -6} 2- \ell} \big(\frac 1 {\sqrt{2}} - z\big)^{\beta \nu + 4}\,  \big(\log \big(\frac 1 {\sqrt{2}} -z\big) \big)^\alpha \,  A^{{\rm
reg}} _{j, \ell, \alpha, \beta}(z)  \, \chi_{]0,   \frac 1 {\sqrt{2}}]}(z )  \, . \end{split} \eeq
 Furthermore  by virtue of the required condition \eqref{condlam}, the functions  $A^0_{j, \ell}$,  $3 \leq j \leq k-1$, vanish on $\ds z=\frac 1 {\sqrt{2}}\virgp $ namely:
\beq \label{condregdev} (A^0_{j, \ell})_{z=\frac 1 {\sqrt{2}}} =0  \, , \eeq 
which  together with  \eqref{eqnl4},  \eqref{formwkl} and \eqref{Adev} give \eqref{nonlindevf}  for $F^{{\rm
nl}, 4}_{k,\ell}$.

\end{proof}

 \bigbreak
 
 The second step in the proof of Proposition \ref{STKL} relies on  the following lemma:  
 \begin{lemma}
\label {propsourcestep2}
{\sl  Consider for $k \geq 6$ the non homogeneous equation  \beq \label{eqclaim}\cS_{k} \,  X = \cF^{{\rm nl}}_{k}\,, \eeq
where $\cS_{k}$ is defined by 
 \eqref {reckl}   and  $\cF^{{\rm nl}}_{k}=\big(F^{{\rm
nl}}_{k,\ell} \big)_{0   \leq \ell \leq \ell(k)}$. Then the following properties hold:

\smallskip
  \begin{enumerate}
 \item The system \eqref{eqclaim} has a unique solution $X^0=\big(X_{0,\ell} \big)_{0   \leq \ell \leq \ell (k)}$ such that $X_{0,\ell} \equiv 0$ for any integer~$ \ell_1(k)< \ell \leq \ell(k)$, where  $ \ds \ell_1(k)= [\frac {k-6 } 2 ] \, \virgp$ and such that if~$ \ell \leq \ell_1(k)$, then $X_{0,\ell}$ belongs to  the functional space $ \cC^ {[k\nu+4]}\big(\R^*_+\big)\cap 
\cC^ {\infty}\big(\R^*_+\setminus \{\frac 1 {\sqrt{2}}\}\big) $ and has the following  form:
\beq\begin{split} \label{nonlindevx} & X_{0,\ell}(z)= X_{0,\ell}^{{\rm
reg}}(z) + \big(\frac 1 {\sqrt{2}} - z\big)^{k \nu + 7}\, \sum_{0   \leq \alpha \leq \frac  {k -6} 2- \ell} \big(\log \big(\frac 1 {\sqrt{2}} -z\big) \big)^\alpha \,   { X}^{{\rm
reg}}_{0, \ell, \alpha} (z) \,  \chi_{]0,   \frac 1 {\sqrt{2}}]}(z )  \\
& \qquad     \qquad  \qquad  + \sumetage {3 \leq \beta \leq k-3} {0   \leq \alpha \leq \frac  {k -6} 2- \ell} \big(\frac 1 {\sqrt{2}} - z\big)^{\beta \nu + 4}\,  \big(\log \big(\frac 1 {\sqrt{2}} -z\big) \big)^\alpha \,  {X}^{{\rm
reg}}_ {0, \ell, \beta, \alpha} (z) \, \chi_{]0,   \frac 1 {\sqrt{2}}]}(z ) \, , \\
&    X_{0,\ell}\Big(\frac 1 {\sqrt{2}}\Big)=0\, \cdot   \end{split} \eeq
 
 \medskip   \noindent   Moreover,  it admits an expansion of the form \eqref {0w_kl} as $z  \to 0$:  \beq
\label {0x_kl}
X_{0,\ell} (z)=  \sumetage{0   \leq \alpha \leq \frac  {k -4} 2- \ell} { \beta \geq 1-k + 2 (\alpha   +\ell)}     X_{0, \ell, \alpha,\beta}  \, (\log z)^{\alpha} \,  z^\beta 
   \, , \eeq 
for some constants $X_{0, \ell, \alpha,\beta}$. 

 \smallskip  
\item  The system \eqref{eqclaim} has a unique solution $X^1=\big(X_{1,\ell} \big)_{0   \leq \ell \leq \ell (k)}$ such that $X_{1,\ell} \equiv 0$ for any integer~$ \ell_1(k)< \ell \leq \ell(k)$,  and such that if~$ \ell \leq \ell_1(k)$, then $X_{1,\ell} \in  \cC^ {\infty}\big(]\frac 1 {\sqrt{2}}\virgp \infty[ \big) $  with the following asymptotic as $z  \to \infty$:\beq\begin{split} \label{infx_kl}&  X_{1,\ell}(z)=z^{k \nu-1}  \sumetage { 0 \leq \alpha \leq \frac {k -6} 2-\ell}   { p \in  \N}   {  X}_{ 1,  \ell,  \alpha, p} \big( \log z \big) ^{\alpha}   z ^{-p}\\
& \qquad     \qquad  \qquad    \qquad  \qquad   \qquad  \qquad+  \sumetage { 3 \leq \beta \leq   k -3 }  { 0 \leq \alpha \leq \frac {k -6} 2-\ell,  \, p \in  \N}  {X }_{  1,  \ell, \alpha, \beta, p} \big( \log z \big) ^{\alpha}     z ^ {\nu \beta + 1-p}  \,    \virgp 
\end{split} \eeq
where  ${  X}_{ 1,  \ell,  \alpha, p}$ and ${X }_{  1,  \ell, \alpha, \beta, p} $ denote some constants. 

  \end{enumerate}
}

  \end{lemma}

\medbreak

 \begin{proof}
 In order to establish this lemma, we shall proceed   by induction on the index $\ell$.   Since  for any integer $k \geq 6$, we have $$F^{{\rm
nl}}_{k,\ell } \equiv 0\, , \, \,  \ell_1(k) <  \ell \leq \ell(k)   \, ,$$ we get    $$X_{0,\ell} \equiv 0\, , \, \, \forall \, \ell_1(k) <  \ell \leq \ell(k)  \, .  $$
Consider now 
\begin{equation} \label{lasteq} {\wt {\mathcal L}}_{k} X_{0, \ell_1(k)}= F^{{\rm
nl}}_{k,\ell_1(k)} \,.\end{equation} Invoking formulae   \eqref{nonlindevf}, \eqref{0F_kl}  together with    Lemma \ref {genLres2}, we easily check    that  the above equation has a unique   solution $X_{0, \ell_1(k)}$ in $ \cC^ {[k\nu+4]}\big(\R^*_+\big)\cap 
\cC^ {\infty}\big(\R^*_+\setminus \{\frac 1 {\sqrt{2}}\}\big) $ which assumes the form \eqref{nonlindevx}  and admits   asymptotic expansion   of  type~\eqref  {0x_kl}, for  $z$  close to~$0$. 

\medskip Let us assume now that for any integer $ \ell < q  \leq  \ell_1(k)$,  the equation 
$$ {\wt {\mathcal L}}_{k} X_{0, q}= F_{k,q} \,$$
 admits a unique solution $X_{0, q}$ which  satisfies \eqref{nonlindevx}  and ~\eqref  {0x_kl}.  
 Then by virtue of Formula \eqref   {deflinFkl},   we find that  $ F^{{\rm
lin}}_{k,\ell}$  undertakes  the following form:\beq\begin{split} \label{nonlindevflin} &F^{{\rm
lin}}_{k,\ell}(z)= F_{k,\ell}^{{\rm
reg}}(z) + \big(\frac 1 {\sqrt{2}} - z\big)^{k \nu + 6}\, \sum_{0   \leq \alpha \leq \frac  {k -6} 2- \ell -1} \big(\log \big(\frac 1 {\sqrt{2}} -z\big) \big)^\alpha \,  F^{{\rm
reg}}_{k, \ell, \alpha} (z) \,  \chi_{]0,   \frac 1 {\sqrt{2}}]}(z )  \\
& \qquad     \qquad  \qquad  + \sumetage {3 \leq \beta \leq k-3} {0   \leq \alpha \leq \frac  {k -6} 2- \ell-1} \big(\frac 1 {\sqrt{2}} - z\big)^{\beta \nu + 3}\,  \big(\log \big(\frac 1 {\sqrt{2}} -z\big) \big)^\alpha \, F^{{\rm
reg}}_{k, \ell, \alpha, \beta} (z)  \, \chi_{]0,   \frac 1 {\sqrt{2}}]}(z )  \, , \end{split} \eeq
 and behaves as follows  close to  $0$:
\beq\label{asymptl} F^{{\rm
lin}}_{k,\ell}(z)=  \sumetage {0   \leq \alpha \leq \frac  {k -6} 2- \ell} { \beta \geq 5-k + 2 (\alpha   +\ell) }    \wt F_{k,\ell, \alpha, \beta}  \,(\log z)^{\alpha}   \, z^{\beta -2} \,, \eeq
for some   constants $\wt F_{k,\ell, \alpha, \beta}$, which implies that $F_{k,\ell  }= F^{{\rm
lin}}_{k,\ell }+ F^{{\rm
nl}}_{k,\ell  }$  verifies \eqref{nonlindevf} and \eqref{0F_kl}.

\medskip

Therefore taking into account      Lemma \ref {genLres2}, we infer    that  the   equation $ {\wt {\mathcal L}}_{k} X_{0, \ell}= F_{k,\ell}$ has a unique   solution $X_{0,\ell}$   in $ \cC^ {[k\nu+4]}\big(\R^*_+\big)\cap 
\cC^ {\infty}\big(\R^*_+\setminus \{\frac 1 {\sqrt{2}}\}\big) $ which   satisfies \eqref{nonlindevx}  and ~\eqref  {0x_kl}.   
 
\bigskip  

The proof of the second part of the lemma is also  by induction on   $\ell$.  First taking into account   Lemma \ref {genLres2} together with Formula \eqref{infF_kl},  we infer that the equation $$ {\wt {\mathcal L}}_{k} X_{1, \ell_1(k)}= F^{{\rm
nl}}_{k,\ell_1(k)} $$  admits a unique solution $ X_{1,\ell_1(k)} $ which belongs to $\cC^ {\infty}\big(]\frac 1 {\sqrt{2}}\virgp \infty[ \big) $  and verifies  \eqref{infx_kl}.
 Then assuming  that for any integer $ \ell < q  \leq  \ell_1(k)$,  the equation 
$$ {\wt {\mathcal L}}_{k} X_{1, q}= F_{k,q} \,$$ admits a unique solution $X_{1, q}$ which  satisfies \eqref{infx_kl},  
we deduce that   $ F^{{\rm
lin}}_{k,\ell}$ defined by \eqref   {deflinFkl} has  an expansion of the following form at infinity:
\beq
\label{astlinf}
\begin{split}
& \qquad  F^{{\rm
lin}}_{k,\ell}(z)=  \sumetage {3   \leq \beta \leq    k -3  } {0   \leq \alpha \leq \frac  {k -6} 2-\ell-1, \, p  \in \N}     {\hat F}_{k,\ell, \alpha,\beta,p} \,(\log z)^{\alpha} \,  z^{\beta\, \nu  +1-p}\\
& \qquad  {}  \qquad {}    \qquad  {} \qquad {}    \qquad  {}  \qquad {}  \qquad  {}   \qquad  {}   +  z^{k\, \nu  -1}  \,  \sumetage {0   \leq \alpha \leq \frac  {k -6} 2-\ell-1}  {p  \in \N}    {\hat F}^{k,\ell} _{k,\ell, \alpha, p} \,(\log z)^{\alpha}\,  z^{ -p} \,  \virgp
\end{split}
\eeq with some   constants   $ {\hat F}_{k,\ell, \alpha,\beta,p}$  and  $  {\hat F}^{k,\ell} _{k,\ell, \alpha, p}$.

\medskip  Since 
$$F_{k,\ell}=  F^{{\rm
lin}}_{k,\ell}+ F^{{\rm
nl}}_{k,\ell}\, , $$  it follows from  \eqref{infF_kl}, \eqref {astlinf} and Lemma \ref {genLres2} that  the equation $ {\wt {\mathcal L}}_{k} X_{1, \ell}= F_{k,\ell}$    has a unique   solution~$ X_{1,\ell} $ in  $\cC^ {\infty}\big(]\frac 1 {\sqrt{2}}\virgp \infty[ \big) $  admitting an  asymptotic of the form \eqref{infx_kl} as $z  \to \infty$.
This achieves the proof of the lemma. 
 
 \end{proof}
 
\bigbreak

We now return to  the proof of Proposition \ref{STKL}.  Taking advantage of Lemma \ref {propsourcestep2} $(1)$, 
 we get~$W_k:=\big(w_{k,\ell} \big)_{0   \leq \ell \leq \frac  {k -3} 2}$  by setting
\begin{equation} \label{endsol}  \quad  \left\{
\begin{array}{l}
\ds W_k = X^{0} + \sum_{0   \leq j \leq  \ell(k)} \big(a^k_{j,+}\, f^{ j, +}_{k}+ a^k_{j,-}\, f^{ j, -}_{k} \big)\, \,\, \mbox{for} \,\, z \leq \frac 1 {\sqrt{2}}  \virgp\\
\ds  W_k = X^{0} + \sum_{0   \leq j \leq  \ell(k)} a^k_{j,+}\, f^{ j, +}_{k}  \, \,\, \mbox{for} \,\, z > \frac 1 {\sqrt{2}}  \, \virgp
\end{array}
\right. \end{equation} 
  where  $ \big(f^{j,\pm}_{k}\big)_{0   \leq j \leq \ell(k)}$ denotes the basis  of  solutions of $ \cS_{k} \, X =0 $ given by Lemma~\ref {basis},    where $X^{0} $ is given by Lemma \ref {propsourcestep2} $(1)$, and where in view of Formula~\refeq{formmatch22} the coefficients $a^k_{j,\pm}$ are determined by 
\begin{equation} \label{endsolcoef}  \quad  \left\{
\begin{array}{l}
\ds X_{0,  \ell, 0, -3}   + \sum_{\ell   \leq j \leq  \ell(k)}  \mu^{j,\ell} _{k,0} \big(a^k_{j,+} + a^k_{j,-} \big)=  c^{k,\ell} _{0, -3}   \, \virgp
 \\
\ds  X_{0,  \ell, 0, -2}  + \sum_{\ell   \leq j \leq  \ell(k)} \mu^{j,\ell} _{k,1} \big(a^k_{j,+} -a^k_{j,-} \big) = c^{k,\ell} _{0, -2} \, \virgp
\end{array}
\right. \end{equation} 
where respectively  $X_{0,  \ell, 0, -3}$  and $X_{0,  \ell, 0, -2}$,   $c^{k,\ell} _{0, -3}$ and $c^{k,\ell} _{0, -2}$,  denote the coefficients involved in~\eqref {0x_kl}  and~ \eqref{formmatch22}, \footnote{ with the convention $X_{0,  \ell, 0, \beta}=0$ if $\ell> \ell_1(k)$ and $c^{k,\ell} _{0, -3}=0$ if $\ell= \frac  {k -3} 2 \cdot$}  and where the coefficients $\mu^{j,\ell} _{k,0}$ and $\mu^{j,\ell} _{k,1}$ are defined so that 
\begin{equation} \label{defcoefmu}  f^{ j, +}_{k, \ell}(z)= \frac {\mu^{j,\ell} _{k,0}} {z^3} + \frac {\mu^{j,\ell} _{k,1}} {z^2} + \O\Big( \frac 1 z \Big) \virgp \,\, \mbox{as} \,\, \, z \to 0 \,\cdot \end{equation} 
By virtue of \eqref{systbasis}, we easily deduce  that 
\begin{equation} \label{coefmu}  \quad  \left\{
\begin{array}{l}
\ds    \mu^{j,\ell} _{k,0} =    \Big( \frac 1 {\sqrt{2}} \Big)^{\alpha(\nu, k)} \,   \begin{pmatrix} j  \\ \ell \end{pmatrix}  \Big( \log \big(  \frac 1 {\sqrt{2}} \big)\Big) ^{j-\ell} \,  \\
\ds    \mu^{j,\ell} _{k,1}   =   {\sqrt{2}}  \,      \Big( \alpha(\nu, k) -  \frac  {j-\ell}  {\log (\sqrt{2})}\Big) \,  \mu^{j,\ell} _{k,0} \cdot
\end{array}
\right. \end{equation} 
   By  Lemma \ref {propsourcestep2} $(2)$,  
$$ W_k= X^1+  \sum_{0   \leq j \leq \ell(k) } \,   \alpha^{j,+ }_{k}  \,  f^{j,+ }_{k, \ell}+  \alpha^{j,- }_{k}  \,  f^{j,- }_{k, \ell} \,  ,$$ with some coefficients $ \alpha^{j,\pm }_{k}$, which concludes   the proof of   the first part of Proposition \ref {STKL}.

\bigskip In order to establish  the  part $(2)$ of the proposition, we shall again proceed by induction. Firstly, let us  investigate  the uniqueness for the solutions to   \eqref{eqsyst}  near $\ds 0 $,  and consider the indexes $k=3, 4$ and $5$. By the computations carried out in Section \ref{step2sub} (see \eqref{coefcase50}),  we have in that case   $ w^i_{k,1}=0$  and 
$$   {\wt {\mathcal L}}_{k} w^i_{k,0}=0  \, ,$$  
which implies that 
$$  {\wt {\mathcal L}}_{k} (w^0_{k,0}-w^1_{k,0})=0 \,.$$ Invoking  Remark \ref {remself}  together with Hypothesis \eqref{cond0}, we easily gather that $w^0_{k,0}= w^1_{k,0}$ on a neighborhood of $0$, for $k=3, 4$ and $5$. 

\medskip Let us    assume now that for any index $k \leq k_0-1 \leq M-1$, the uniqueness for solutions to~\eqref{eqsyst}  near $0$  holds under  Hypothesis \eqref{cond0}.  Since $F^{{\rm
nl}}_{k_0,\ell} (\lam; w)$ only depends on $w_{j,\ell}$, $j \leq k_0-3$, this ensures that 
\begin{equation} \label{equniq}  \cS_{k_0}(\cW^0_{k_0}-\cW^1_{k_0})=0  \, ,\end{equation} 
  where \begin{equation} \label{vectequniq} 
\cW^i_{k_0}=
 \left(
\begin{array}{ccccccccc}
w^i_{k_0,0} \\
\vdots \\
w^i_{k_0,\ell} \\
 \vdots \\
w^i_{k_0,\ell(k_0)}  
\end{array}
\right)   \, .\end{equation}  In order to prove that $\cW^0_{k_0}=\cW^1_{k_0}$, we shall proceed   by induction on the index $\ell$  starting by $\ell(k_0)$. Taking into account   \eqref {reckl}   together with  \eqref{equniq}, we infer   that $$  {\wt {\mathcal L}}_{k_0} (w^0_{k_0,\ell(k_0)}-w^1_{k_0,\ell(k_0)})=0 \,.$$ 
Thanks to Lemma \ref {basis} and Condition  \eqref{cond0}, this implies that  
$$w^0_{k_0,\ell(k_0)}=w^1_{k_0,\ell(k_0)}  \, . $$
Assume now that for any integer $ \ell < q  \leq  \ell(k_0)$,  we have    on a neighborhood of $0$
$$w^0_{k_0,q}=w^1_{k_0,q} \,.$$
Therefore  in view of the  definition  of $\cS_{k_0}$  page \pageref {reckl},  we find that 
$$  {\wt {\mathcal L}}_{k_0} (w^0_{k_0,\ell }-w^1_{k_0,\ell})=0 \,,$$ 
which, due to Lemma \ref {basis} and Condition  \eqref{cond0}, easily ensures that  $$w^0_{k_0,\ell }=w^1_{k_0,\ell} \,.$$
This achieves the proof of the uniqueness for solutions to   \eqref{eqsyst}  near $0$.

\medskip Secondly, let us  investigate  the uniqueness for solutions to   \eqref{eqsyst}  around   $\ds  + \infty $.  Again, we shall proceed by induction starting with the indexes $k=3, 4$ and $5$.  In that case,  we have  $$   {\wt {\mathcal L}}_{k} w^i_{k,0}=0  \, ,$$ and the conclusion follows easily from  \eqref{condinfty}. Now, assuming that the uniqueness holds  under Hypothesis \eqref{condinfty}, for any index $k \leq k_0-1 \leq M-1$,  let  us consider the index $k_0$. Again, by the induction hypothesis, we have $$ \cS_{k_0}(\cW^0_{k_0}-\cW^1_{k_0})=0  \, .$$   This gives the result thanks  to Lemma \ref {basis} and Condition \eqref{condinfty}, which ends the proof of the proposition.

\medbreak

\begin{remark} \label{remuniq}
{\sl  By virtue of the uniqueness of the solutions to the system \eqref {eq:W_kl} near $0$ established above, we readily gather  from  the matching conditions  \eqref{exporigkl} coming out from the inner region that \beq
\label {equniq}d^{k,\ell} _{\alpha, \beta} =  c^{k,\ell} _{\alpha, \beta}(\lam)\,, \,\, \forall \, \, k,\ell, \alpha, \beta.\eeq }
\end{remark}

\bigbreak

%%%%%%%%%%%%%%%%%%%%%%%%%%%%%%%%%%%%%%%%%%%%
\subsection{Estimate of the approximate solution in the  self-similar region}\label{Est2}
Under the above notations, set for any integer~$N \geq 3$
\beq
\label{aprx2} V_{{\rm
ss}}^ {(N)}(t, y )=y + \lam^{(N)}(t)   t^{-\nu -1}  W_{{\rm
ss}}^ {(N)} \Big(t,y   \frac  {t^{\nu +1 }  }  {\lam^{(N)}(t) }\Big) \, ,  \, \, u_{{\rm
ss}}^ {(N)}(t, \rho) =     t^{\nu +1} V_{{\rm
ss}}^ {(N)}\Big(t, \frac {\rho} {t^{\nu +1}}  \Big)\,,\eeq 
with
\beq
\label{aprx2BIS}  W_{{\rm
ss}}^ {(N)}(t,z)=  \sum^{N}_{k= 3} t^ {\nu k} \sum^{\ell(k)}_{\ell  = 0}   \,(\log t)^\ell \,w_{k,\ell}  \big(z\big)\andf \lam^{(N)}(t)=  t \Big(1+ \sum^{N}_{k = 3} \sum^{\ell(k)}_{\ell  = 0}\lam_{k, \ell} \, t^ {\nu k} \,(\log t)^\ell \Big)\,.\eeq
The purpose of   this paragraph is firstly to estimate  the radial function $V_{{\rm
ss}}^ {(N)}$ defined by \eqref{aprx2}, in the   self-similar region:   
\beq \label{ssreg} \Omega_{{\rm
ss}}:=  \Big\{ Y \in  \R^4, \, \frac {t^ {  \epsilon_1- \nu }} {10  }  \leq |Y|  \leq  10 \,  t^ { - \epsilon_2- \nu }    \Big\}\,,
  \eeq
 and secondly to study, for $N$ sufficiently large,  
the remainder term.

\medskip
 Combining Identity \eqref{formmatch22} together with Lemma \ref  {STKL}, we firstly  get the following lemma:\begin{lemma}
\label {solVss}
{\sl There exist a positive constant $C$ and a small positive time $T=   T ( N) $ such that  the following $L^\infty $ estimates hold,  for  any time   $0 < t \leq T $:
\begin{equation}  \label {property1ss}  \qquad \qquad\| \langle \cdot \rangle^{|\alpha| -1} \nabla^{\alpha}   ( V_{{\rm ss}}^ {(N)}(t,  \cdot)-Q)\|_{L^{\infty}(\Omega_{{\rm
ss}})} \leq C  \,  [ t^{3(\nu -   \epsilon_1)}  +   t^{3\nu (1-  \epsilon_2)}],   \,  \forall \,  |\alpha| <  3 \nu + 4  \, ,\end{equation}
\beq\label{property1ssbis} 
\begin{split}
& \qquad  {}  \,  \,   \, \,    \| \langle \cdot \rangle^{\beta} \nabla^{\alpha}   ( V_{{\rm ss}}^ {(N)}(t,  \cdot)-Q)\|_{L^{\infty}(\Omega_{{\rm
ss}})} \leq C   \,   [ t^{2\nu + 2 (\nu -\epsilon_1)}  +  t^{(\nu - \epsilon_1)  (N+ 1)}       \\
&   \qquad  {}    {}  \,  \,     \qquad  {}   \qquad    \qquad  {}   \qquad {}  + t^{\nu +  1 + (3\nu - 1)  (1-\epsilon_2)} ], \,  \, \forall \, \beta \leq  |\alpha| - 2 \andf  1 \leq |\alpha| <  3 \nu + 4 \,  . \end{split}
\eeq
In addition $ \partial_t V_{{\rm ss}}^ {(N)}$ satisfies 
\begin{equation}  \label {property1ssdevt} \qquad \|  \partial_t V_{{\rm ss}}^ {(N)}(t,  \cdot) \|_{L^{\infty}(\Omega_{{\rm
ss}})} \leq C  \,  t^{-2- \nu } \, [ t^{1 +  \nu +  2(\nu -   \epsilon_1)}  +   t^{(1+ 3\nu) (1-  \epsilon_2)}]  \, , \end{equation}
\beq\label{property1ssbis} 
\begin{split}
&  {} \, \, \,  \quad  \|   \nabla^{\alpha}   \partial_t V_{{\rm ss}}^ {(N)}(t,  \cdot)\|_{L^{\infty}(\Omega_{{\rm
ss}})} \leq C  \,  t^{-1 } \,  [ t^{  3 (\nu -\epsilon_1)}  +  t^{3\nu (1-\epsilon_2)}    ], \,  \forall \, 1 \leq |\alpha| <    3 \nu + 3 \,  .    \end{split}
\eeq
Besides for any multi-index $\alpha$ of length $|\alpha|  <   3 \nu + 3$,  the function \footnote{ We recall that $\rho= y\, t^{\nu+1}$.}  $ V_{{\rm ss}, 1}^ {(N)}(t,y):= (\partial_t  u_{{\rm ss}}^ {(N)})(t,\rho)$ satisfies \beq
\label{asymptlinf}
\begin{split}
& \qquad  \| \langle \cdot \rangle^{\beta}  \nabla^{\alpha}   V_{{\rm ss}, 1}^ {(N)} (t,  \cdot)\|_{L^{\infty}(\Omega_{{\rm
ss}})} \leq   C \, t^{ \nu } \,   [ t^{3(\nu -  \epsilon_1)}  +   t^{3\nu (1-  \epsilon_2)} ],  \,  \,  \forall \, \beta   \leq    |\alpha| - 1  \, ,  \\
&   \qquad  {}  \| \langle \cdot \rangle^{\alpha}  \nabla^{\alpha}    V_{{\rm ss}, 1}^ {(N)} (t,  \cdot)\|_{L^{\infty}(\Omega_{{\rm
ss}})} \leq C \,  [ t^{3\nu - 2 \epsilon_1}  +   t^{3\nu (1-  \epsilon_2)}]   
\end{split}
\eeq
\beq 
\label{asymp44} \quad \|  \partial_t   V_{{\rm ss}, 1}^ {(N)} (t,  \cdot)\|_{L^{\infty}(\Omega_{{\rm
ss}})} \leq   C \,t^{-1 } \, [ t^{3 \nu -  2  \epsilon_1}  +   t^{  3\nu (1-  \epsilon_2)}]  \, .
\eeq
Finally for any multi-index $\alpha$  of length $|\alpha|   <   3 \nu + 2$ and any integer $\beta   \leq    |\alpha|$,  we have\beq
\label{asymptlinf}
\begin{split}
& \qquad  \| \langle \cdot \rangle^{\beta}  \nabla^{\alpha}   V_{{\rm ss}, 2}^ {(N)} (t,  \cdot)\|_{L^{\infty}(\Omega_{{\rm
ss}})} \leq   C \,    [ t^{2 \nu   + 2(\nu -  \epsilon_1)}  +   t^{\nu +1+ (3\nu -1)(1-  \epsilon_2)} ] 
  \,  , 
\end{split}
\eeq 
 where $V_{{\rm ss}, 2}^ {(N)}(t,y)= t^{\nu +1}   \, (\partial^2_t  u_{{\rm ss}}^ {(N)})(t,\rho)$. }
\end{lemma}

\bigbreak

 In the same spirit as Lemma \ref {solVinl2},  we have  the following   result.  \begin{lemma}
\label {solVssl2}
{\sl  The following $L^2$ estimates hold for all  $0 < t \leq T$:
\beq
 \label {property11l2ss}
\begin{split}
&     \|  \nabla^{\alpha}   (V_{{\rm
ss}}^ {(N)} (t,  \cdot)-Q)\|_{L^{2}(\Omega_{{\rm
ss}})} \leq   C\, [ t^{ \nu |\alpha|-  \epsilon_1( |\alpha|-2)}  +   t^{(\nu -  \epsilon_1)(N+|\alpha|-3)} \, ,  \\
&  \qquad {}  \qquad  {}   \qquad \, \, \qquad  {}  \qquad {}  \qquad  {}   \qquad \qquad {}  +   t^{ \nu  |\alpha| -\epsilon_2(3\nu +3 - |\alpha|)}  ]  \, \, ,\, \forall \, 1 \leq  |\alpha| < 3 \nu + 4 + \frac 1 2\, \virgp
\end{split}
\eeq
\begin{equation}   
\begin{split}
 \label {property4l2ss}  \quad \|   \nabla^{\alpha}   (V_{{\rm ss}, 1}^ {(N)}) (t,  \cdot)\|_{L^{2}(\Omega_{{\rm
ss}})} \leq  C   \,   t^{ \nu   ( |\alpha| +1)} \, [ t^{ -\epsilon_1 |\alpha|}  +   t^{- \epsilon_2 (3\nu +2 - |\alpha| )}  ] \,,   \,  \, \forall \, 0 \leq  |\alpha| < 3 \nu + 3 + \frac 1 2\, \virgp \end{split}\end{equation}
\begin{equation}   
 \label {property4l2sstt}  \begin{split}
&  \|   \nabla^{\alpha}   (V_{{\rm ss}, 2}^ {(N)}) (t,  \cdot)\|_{L^{2}(\Omega_{{\rm
ss}})} \leq  C   \,   t^{ \nu   ( |\alpha| +2)} \, [ t^{ -\epsilon_1 |\alpha|}  +   t^{- \epsilon_2 (3\nu +1 - |\alpha| )} (1+ t^{3\nu - 2 \epsilon_2)} ) ] \,,   \\
&  \qquad {}  \qquad  {}   \qquad  {} \qquad  {}   \qquad  {}   \qquad \qquad  {} \qquad  {}  \qquad  {}   \qquad  {} \qquad  {}  \qquad  {}   \qquad  {}\,  \, \forall \, 0 \leq  |\alpha| < 3 \nu + 2 + \frac 1 2\, \cdot \end{split}\end{equation}}
\end{lemma}

\bigbreak

 Let us now  consider the remainder 
 $${\cR}_{{\rm
ss}}^ {(N)}(t,y) := \big[\eqref{eqpart1} V_{{\rm ss}}^ {(N)}\big](t,y)\, .$$
Clearly, 
 $${\cR}_{{\rm
ss}}^ {(N)}(t,y) = \frac  {t^{\nu +1 }  }  {\lam^{(N)}(t) } \wt{\cR}_{{\rm
ss}}^ {(N)}  \Big(t, y   \frac  {t^{\nu +1 }  }  {\lam^{(N)}(t) }\Big)\virgp$$
where 
$$ \wt{\cR}_{{\rm
ss}}^ {(N)}(t,z)= \big[\eqref{eqregion2} W_{{\rm
ss}}^ {(N)}\big] (t,z)\, .$$
By construction,  
$$\wt{\cR}_{{\rm
ss}}^ {(N)}(t,z)= \sumetage {k \geq N+1} { \ell \leq   \frac  {k -6} 2} t^ {\nu k}   (\log t)^\ell r_{k, \ell }(z)\with r_{k, \ell }(z)=F^{{\rm
nl}}_{k,\ell}(W_{{\rm
ss}}^ {(N)}, \lam^ {(N)})\, .$$
In view of computations carried out in Section \ref{step2sub}, we have 
\beq
\label{formrem2}
\begin{split}
& \qquad   r_{k, \ell }  (z )= r^{\rm
reg}_{k,\ell} (z)+  \big(\frac 1 {\sqrt{2}} - z\big)^{k \nu + 6}\,  
\sumetage {0   \leq \alpha \leq \frac  {k -6} 2- \ell}  \,  r^{\rm
reg}_{k,\ell, \alpha }(z) 
\,\big(\log \big(\frac 1 {\sqrt{2}} -z\big) \big)^\alpha   
\,\chi_{]0,   \frac 1 {\sqrt{2}}]}(z )\\
& \qquad  {}  \qquad {}  + \sumetage { 3 \leq \beta \leq k-3}  {0   \leq \alpha \leq \frac  {k -6} 2- \ell}  \, \,  r^{\rm
reg}_{k,\ell, \alpha, \beta }(z)  \big(\frac 1 {\sqrt{2}} - z\big)^{\beta \nu + 2}\,  \big(\log \big(\frac 1 {\sqrt{2}} -z\big) \big)^\alpha    \, \chi_{]0,   \frac 1 {\sqrt{2}}]}(z ) \,  \cdot
\end{split}
\eeq
Furthermore, as   $z  \to 0$ and as $z  \to \infty$, $r_{k, \ell } $ satisfies \eqref{0F_kl}, \eqref{infF_kl} respectively. 
%$3\nu   +2-K_0>-1/2$ to ensure that R is in $L^2$ near $\frac 1 {\sqrt{2}} $

\medskip

As a direct consequence of these properties, we obtain the following lemma:
\begin{lemma}
\label {estap2}
{\sl There exist  a small positive  time $T=T(N)$ and a positive constant~$C_{ N }$  such that for all  $0 < t \leq T$,  the  remainder term~$\cR_{{\rm
ss}}^ {(N)}$   satisfies the following estimate 
\begin{equation}\label {Condremss} \|\langle \cdot \rangle^{\frac 3 2}  \,  \cR_{{\rm
ss}}^ {(N)}(t, \cdot)\|_{H^{K_0}(\Omega_{{\rm
ss}})} \leq C_{ N} \,  [ t^{  (\nu -  \epsilon_1)(N- \frac 3 2)} +   t^{ \nu (1-  \epsilon_2)    (N + 1) -  \frac  5 2 (\nu  + 1) }   ]  
 \, ,\end{equation}
 where $\ds K_0= [ 3\nu   +  \frac 5 2 ] \cdot$
}
\end{lemma}

\medbreak

Let us end this section by investigating $V_{{\rm in} }^ {(N)} - V_{{\rm ss} }^ {(N)}$ in the intersection of the inner and self-similar regions,  namely in 
$$\Omega_{{\rm
in}} \cap \Omega_{{\rm
ss}}=\Big\{ Y \in  \R^4, \, \frac {t^ {  \epsilon_1- \nu }} {10  }  \leq |Y|   \leq    t^ {  \epsilon_1- \nu }    \Big\}\, \cdot
  $$ 
  In view of \eqref{formmatch22}, \eqref{formwkl} and Remark \ref{remuniq}, we have   for any multi-index $\alpha$ and any integer $m$
$$ \big|\partial_y^\alpha \partial_t^m( V_{{\rm in} }^ {(N)} - V_{{\rm ss} }^ {(N)})(t, y) \big|\lesssim t^ { 2\nu(N+1) -m} \, y^ { 2 N - |\alpha|} +   t^{ -m}  y^ {- N -|\alpha|} \, ,  $$
provided that $y$ belongs to $\Omega_{{\rm
in}} \cap \Omega_{{\rm
ss}}$,    and $t$ is sufficiently small, which leads to the following result:
\begin{lemma}
\label {difestap2}
{\sl For any integer $m$ and any multi-index $\alpha$,  the following estimate holds \begin{equation}\label {estdiff} \| \nabla^\alpha \partial_t^m (V_{{\rm in} }^ {(N)} - V_{{\rm ss} }^ {(N)})(t, \cdot)\|_{L^{\infty}(\Omega_{{\rm
in}} \cap \Omega_{{\rm
ss}})} \leq C_{N, \alpha, m} t^{  -m +|\alpha|(\nu -\epsilon_1)}\, \big(t^{2\nu  + 2 N  \epsilon_1 } + t^{  N  (\nu -\epsilon_1) } \big)\,,\end{equation}
for all   $0 < t \leq T=T(\alpha, m, N)$.}
\end{lemma}

\bigbreak

%%%%%%%%%%%%%%%%%%%%%%%%%%%%%%%%%%%%%%%%%%%%
%%%%%%%%%%%%%%%%%%%%%%%%%%%%%%%%%%%%%%%%%%%%

\section{Approximate solution in the   remote region }\label {step3}
\subsection{General scheme of the construction of the approximate solution  in the   remote region}\label{st3}
 In the  previous section,  we built in the self-similar region an approximate solution $u_{{\rm
ss}}^ {(N)}$   which extends the approximation solution $u_{{\rm
in}}^ {(N)}$ constructed in Section
~\ref{step1} in the inner region. Our goal  here is to extend  $u_{{\rm
ss}}^ {(N)}$ to the whole space.

\medskip Recall that   the approximate solution $u_{{\rm
ss}}^ {(N)}$ built in in Section
~\ref{step2} assumes the following form:
$$ u_{{\rm
ss}}^ {(N)}( t, \rho)= \rho +  \lam^ {(N)}(t) \sum^{N}_{k=3} t^ {\nu k } \sum^{ \ell(k)}_{\ell=0}   \,\big(\log t\big)^\ell \,  w_{k,\ell} \Big(\frac {\rho} {\lam^ {(N)} (t)}\Big) \,,$$
where
 $ \ds \ell(k)   = \big[\frac  {k -3} 2 \big]  \virgp$   and where $\lam^ {(N)}(t)$ is the perturbation of $t$ defined by 
\eqref{aprx2BIS}.

\medskip  To achieve our goal, let us start by introducing the function $u^{{\rm
lin}, (N)}$ defined by 
 \beq
\label{linpartsol} u^{{\rm
lin}, (N)}( t, \rho) :=  t \sum^{N}_{k = 3} t^ {\nu k } \sum^{ \ell(k)}_{\ell=0}   \,\big(\log t\big)^\ell \,  w^{{\rm
lin}}_{k,\ell} \Big(\frac {\rho} {t}\Big) \,,\eeq 
where  $ w^{{\rm
lin}}_{k,\ell}$ denotes the linear part of the function $w_{k,\ell}$ involved in the asymptotic expansion \eqref{formreg2} and given by \eqref{linw_kl}.  

\medskip The function $u^{{\rm
lin}, (N)}$ solves the Cauchy problem:\begin{equation}
\label {Cauchyout}
\left\{
\begin{array}{l}
\ds  (2 \partial^2_t - l_{\rho}) \,u^{{\rm
lin}, (N)}=0 \\
{u^{{\rm
lin}, (N)}}_{|t=0}= u_0^{{\rm
lin}, (N)}  \\
(\partial _t u^{{\rm
lin}, (N)})_{|t=0}= u_1^{{\rm
lin}, (N)} \,, 
\end{array}
\right.
\end{equation} where   $l_{\rho}$ is defined by \eqref {deflinrho}, and where
\begin{equation} \label {defCauchydataout}
\left\{
\begin{array}{l}
\ds u_0^{{\rm
lin}, (N)}(\rho)=   \sum^{N}_{k = 3}   \sum^{ \ell(k)}_{\ell=0}  {\mu}^0_{k,\ell}   \,  \rho^{k\, \nu  +1}  \,  \big(\log \rho\big)^\ell      \, , \\
\ds u_1^{{\rm
lin}, (N)}(\rho)=   \sum^{N}_{k = 3}   \sum^{ \ell(k)}_{\ell=0}   {\mu}^1_{k,\ell}    \,  \rho^{k\, \nu }  \,  \big(\log \rho\big)^\ell   \,,   
\end{array}
\right.\end{equation}
with  under notations \eqref{linw_kl}  
 \begin{equation} \label{eqform}\left\{
\begin{array}{l}
\ds  {\mu}^0_{k,\ell} =   \alpha^{\ell,+ }_{k}  +\alpha^{\ell,- }_{k}     \, , \\ 
\ds  {\mu}^1_{k,\ell} =   \frac 1 {\sqrt{2}}  \Big((\nu k+4) \big(\alpha^{\ell,+ }_{k}  -\alpha^{\ell,- }_{k} \big)+ (\ell+1) \big(\alpha^{\ell+1,+ }_{k}  -\alpha^{\ell+1,- }_{k} \big)\Big)\,,     
\end{array}
\right.  \end{equation}
using again the convention that $\alpha^{\ell+1,\pm }_{k}=0$ if $ \ds \ell+1 >  \frac {k-3} 2  \cdot$

 \medskip  \noindent Indeed, combining \eqref{linw_kl}  together with   \eqref{linpartsol}, we infer that   
$$
 \begin{aligned}
  u^{{\rm
lin}, (N)}(t, \rho) &=   \sum^{N}_{k = 3} t^ {\nu k+1 } \sum^{ \ell(k)}_{\ell=0}   \,\big(\log t\big)^\ell \,  \sum_{0   \leq j \leq \frac  {k -3} 2 }  \, \Big(  \alpha^{j,+}_{k}  \,  f^{j,+}_{k, \ell}\Big(\frac {\rho} {t}\Big)   + \alpha^{j,-}_{k}  \,  f^{j,-}_{k, \ell}\Big(\frac {\rho} {t}\Big) \Big)  \, \cdot
 \end{aligned}
 $$
Taking advantage of \eqref{formulause}, this gives rise to
$$
 \begin{aligned}
 \ds & u^{{\rm
lin}, (N)}(t, \rho) =  \sum^{N}_{k = 3} t^ {\nu k +1} \sum_{0   \leq j \leq \frac  {k -3} 2 }   \Big(  \alpha^{j,+}_{k}    \big(\log t + \log \big(  \big(\frac {\rho} {t}\big)  +\frac 1 {\sqrt{2}}\big)\big)^j \, \frac {\big(\big(\frac {\rho} {t}\big)  +\frac 1 {\sqrt{2}}\big)^{\nu  k+4}} {\big(\frac {\rho} {t}\big)^3} \\
& \qquad  \qquad  \qquad   \qquad  \qquad  \qquad  +  \alpha^{j,-}_{k}    \big(\log t + \log \big(  \big(\frac {\rho} {t}\big)  -\frac 1 {\sqrt{2}}\big)\big)^j \, \frac {\big(\big(\frac {\rho} {t}\big)  -\frac 1 {\sqrt{2}}\big)^{ \nu k +4}} {\big(\frac {\rho} {t}\big)^3} \Big) \\
&   =   \sum^{N}_{k = 3}  \sum_{0   \leq j \leq \frac  {k -3} 2 } \Big( \alpha^{j,+}_{k}    \big( \log \big(  \rho  + \frac t {\sqrt{2}}    \big)\big)^j   \frac {\big(   \rho + \frac t{\sqrt{2}}    \big)^{\nu k+4}} { \rho ^3} + \alpha^{j,-}_{k}    \big( \log \big(  \rho  - \frac t {\sqrt{2}}    \big)\big)^j  \frac {\big(   \rho - \frac t{\sqrt{2}}    \big)^{ \nu  k +4}} { \rho ^3}    \Big) \, \virgp
 \end{aligned}
 $$
which ensures the result.

\medskip
Let now~$\chi_0$ be a radial smooth cutoff function   on $\R^4$  equal to $1$ on the unit ball centered at  the origin  and  vanishing outside the   ball of radius $2$ centered at  the origin, and consider for a small positive real number $\delta$, the compact support functions:
\begin{equation}
\label {defCauchydataoutbis}\left\{
\begin{array}{l}
\ds {\mathfrak g}_0(\rho)=  \chi_\delta (\rho) \,   u_0^{{\rm
lin}, (N)}(\rho)   \, ,\\
\ds {\mathfrak g}_1(\rho)=  \chi_\delta (\rho) \,   u_1^{{\rm
lin}, (N)}(\rho) \,,   
\end{array}
\right.
\end{equation}
where  $u_0^{{\rm
lin}, (N)}$ and  $u_1^{{\rm
lin}, (N)}$  are the functions defined above by \eqref {defCauchydataout}, and where $\ds  \chi_\delta (\rho)= \chi_0 \Big(\frac \rho  \delta\Big)\, \cdot$ 
\begin{remark} \label{g} {\sl    
  Invoking \eqref   {defCauchydataout} together with \eqref {defCauchydataoutbis}, we infer  that there \footnote{ In what follows, the parameter $\delta_0(N)$ may vary from line to line.} is   $\delta_0(N)>0$ such that for all positive real $  \delta \leq \delta_0(N)$ and any integer $m < 3 \nu+2$,   the above functions $ {\mathfrak g}_0$ and $ {\mathfrak g}_1$ belong  respectively   to the Sobolev spaces $\dot H^{m+1 } (\R^4) $ and $\dot H^{m } (\R^4)$, 
and satisfy 
$$  \| {\mathfrak g}_0\|_{ \dot H^{m+1 } (\R^4)} \leq C \, \delta^{3 \nu-m+2} \andf \| {\mathfrak g}_1\|_{\dot  H^{m } (\R^4)} \leq C \, \delta^{3 \nu-m+2} \, . $$
 }
\end{remark}

\bigbreak

 We shall look for the solution in the remote region under the form: 
\begin{equation}
\label {formaout} u_{{\rm
out}} (t, \rho)= \rho+ {\mathfrak g}_0(\rho)+  t  {\mathfrak g}_1(\rho) +  \sum_{k \geq 2 } t^k  {\mathfrak g}_k(\rho)\,.   \end{equation}
To this end, we shall apply the lines of reasoning of Sections ~\ref{step1} and ~\ref{step2} and determine by induction the functions $ {\mathfrak g}_k$, for $k  \geq 2$, making use of the fact  that the function $u_{{\rm
out}}$ is a formal solution to the Cauchy problem:
\begin{equation}
\label {Cauchypbout}\left\{
\begin{array}{l}
\ds \eqref {eq:NW} \, u_{{\rm
out}}=0  \\
\ds  {u_{{\rm
out}}}_{|t=0}= \rho+ {\mathfrak g}_0 \\
\ds  (\partial_t u_{{\rm
out}})_{|t=0}= {\mathfrak g}_1   \,.    
\end{array}
\right.
\end{equation}

\medskip
For that purpose, we substitute  \eqref {formaout}  into   \eqref {eq:NW}, which by straightforward computations leads to the following recurrent relation for $k \geq 2$
\begin{equation} \label{receq3} {\mathfrak g_k}= 
\frac 1 {k(k-1)(2+2 ({\mathfrak g}_0)_{\rho} +({\mathfrak g}_0)^2_{\rho} )}  \, {\cH}_k\big({\mathfrak g_j}, j\leq k-1\big) \,\cdot   \end{equation}
The source term $ {\cH}_k$ involved in the above identity  can be splitted into three parts as follows:
\begin{equation} \label{Hreceq3} {\cH}_k= {\cH}^{(1)}_k+ {\cH}^{(2)}_k +  {\cH}^{(3)}_k \,,\end{equation}
with  
\begin{equation} \label{H1receq3} {\cH}^{(1)}_k= l_\rho \, {\mathfrak g}_{k-2} \,,\end{equation}
 \begin{equation} \label{H2receq3} \begin{aligned}
 &  {\cH}^{(2)}_k= - 2 \sumetage {k_1+k_2=k } { k_2> 0 } k_1   (k_1-1-k_2){\mathfrak g}_{k_1}  ({\mathfrak g}_{k_2})_{\rho}    \\
&\qquad  \qquad  \qquad \qquad + 6\sum_{k_1+k_2=k-2 } \Big( -    {\mathfrak g}_{k_1} \frac  {\check   {u}_{k_2}} {\rho^3}   +  ( {\mathfrak g}_{k_1})_\rho \Big(\frac  {\check {u}_{k_2}} {\rho^2} + \frac     {({\mathfrak g}_{k_2})_\rho} {\rho} \Big) \Big)
   \, , \end{aligned}
 \end{equation} 
\begin{equation} \label{H3receq3} \begin{aligned}
 &  {\cH}^{(3)}_k=  \sum_{k_1+k_2+k_3=k } k_1  k_2\Big(-{\mathfrak g}_{k_1} {\mathfrak g}_{k_2} ({\mathfrak g}_{k_3})_{\rho\rho} +2 {\mathfrak g}_{k_1} ({\mathfrak g}_{k_2})_{ \rho}({\mathfrak g}_{k_3})_{\rho}-3 {\mathfrak g}_{k_1}  {\mathfrak g}_{k_2}\Big(\frac  {\check   {u}_{k_3}} {\rho^2} + \frac     {({\mathfrak g}_{k_3})_\rho} {\rho} \Big)  \Big)  \\
& - \sumetage{k_1+k_2+k_3=k }{2 \leq k_1 < k } k_1(k_1-1) {\mathfrak g}_{k_1} ({\mathfrak g}_{k_2})_{\rho}({\mathfrak g}_{k_3})_{\rho}+3 \sum_{k_1+k_2+k_3=k-2 } ({\mathfrak g}_{k_1})_{ \rho} ({\mathfrak g}_{k_2})_{ \rho} \Big(\frac  {\check   {u}_{k_3}} {\rho^2} + \frac     {({\mathfrak g}_{k_3})_\rho} {\rho} \Big)
   \, \virgp  \end{aligned}
 \end{equation} 
where $\check   {u}_{k}$ is given by  
\begin{equation} \label{defutchec}   \ds \check   {u}  = \frac   {u- \rho}{1+ \frac {u- \rho}  \rho}= \sum_{k \geq 0 } t^k  {\check   {u}}_k    \, \cdot \end{equation}
Note that $ \check   {u}_{k} $   only depends on  ${\mathfrak g}_{k_i }$, with $k_i \leq  k$.

%%%%%%%%%%%%%%%%%%%%%%%%%%%%%%%%%%%%%%%%%%%%

%%%%%%%%%%%%%%%%%%%%%%%%%%%%%%%%%%%%%%%%%%%%
%%%%%%%%%%%%%%%%%%%%%%%%%%%%%%%%%%%%%%%%%%%%

\subsection{Analysis of the fuctions ${\mathfrak g}_{k}$}\label{st33}
The aim of the present  paragraph is to investigate the fuctions~${\mathfrak g}_{k}$ defined above by \eqref{receq3}-\eqref{H3receq3}. To this end, let us start by introducing the following definition. 
\begin{definition} \label {defspace3} 
{\sl We denote by $\cA$ the set of functions $a$ in $\cC^ {\infty}(\R_+^* )$ supported in $\{ 0< \rho \leq 2 \delta \}$, where $\delta$ is the positive parameter introduced in \eqref {defCauchydataoutbis}, and admitting for $\rho < \delta$ an  absolutely convergent expansion of the form: 
\begin{equation} \label {asydefspace3} a( \rho)=  \sum_{j \geq  3}   \sum_{0 \leq \ell \leq \frac {j-3} 2}  a_{j,\ell}   \,  \rho^{ \nu j}  \,  \big(\log \rho\big)^\ell  \, . \end{equation} }
 \end{definition}
 
\medbreak

\begin{remark} 
{\sl  The functional space  $\cA$ given by Definition \ref{defspace3}  is an algebra,  and we have for any function $a$ in $\cA$ and any integer $m$,   \begin{equation} \label {propA}\partial^m a \in 
\rho^{-m } \cA \, .\end{equation} 
}
\end{remark}

\medbreak

Our aim now is to establish   the following  key result which describes  the behavior of the functions~${\mathfrak g}_{k}$. 
\begin{lemma} \label {propG}
{\sl There exists $\delta_0(N)>0$ such that for all positive  real  $\delta \leq \delta_0(N)$,  we have, under the above notations,  for any integer $k$     $${\mathfrak g}_{k} \in \rho^{1-k} \cA \, .$$ }\end{lemma}

\medbreak

\begin{proof}
 Firstly note that  in view of \eqref {defCauchydataout} and \eqref {defCauchydataoutbis},   ${\mathfrak g}_{0} \in \rho \, \cA$ and ${\mathfrak g}_{1} \in   \cA$ for any  $\delta > 0$, and there exists  $\delta_0(N)> 0$ such that 
 \begin{equation} \label{eqG}  \frac 1 {1+ (1+({\mathfrak g_0})_\rho)^2 } \, \cA \subset \cA \, \virgp  \quad \ds \frac 1 {1+ \ds \frac {\mathfrak g_0} \rho  } \, \cA \subset \cA   \, \virgp \end{equation}
 for any $\delta \leq \delta_0(N)$. 
 
\medskip  Let us now  show that for any  $\delta \leq \delta_0(N)$, ${\mathfrak g}_{k} \in \rho^{1-k} \cA  $, for all $k \geq 2$. To this end, we shall proceed by induction assuming that, for any integer $j \leq k-1$, the function ${\mathfrak g}_{j}$ belongs to  $\rho^{1-j} \cA$.

\medskip

Recalling that
$$l_\rho  v=   v_{\rho\rho} +  6  \Big(\frac v {\rho^2} + \frac {v_\rho} {\rho}\Big)\virgp$$
we infer taking into account  \eqref {propA} that the function  ${\cH}^{(1)}_k$ given by \eqref{H1receq3}  belongs to   $\rho^{1-k} \cA$. 

\medskip Since ${\check   {u}}_k $ is defined by 
$$\ds \check   {u}  = \frac   {u- \rho}{1+ \frac {u- \rho}  \rho} =  \sum_{k \geq 0 } t^k  {\check   {u}}_k \, \virgp$$
it   readily follows from the induction hypothesis that 
  for any integer $j \leq k-1$,  ${\check   {u}}_{j}$ belongs to  the functional space $\rho^{1-j} \cA$. 

\medskip
Combining the fact that $\cA$ is an algebra together with  \eqref {propA} and \eqref{eqG}, we deduce  that the function ${\cH}^{(2)}_k$ defined by \eqref{H2receq3} belongs to  $\rho^{1-k} \cA$.
 
 \medskip Along the same lines, taking into account   \eqref{H2receq3}, we readily gather  that  $ {\cH}^{(3)}_k \in \rho^{1-k} \cA$. This concludes   the proof of the result thanks to \eqref{receq3}, \eqref{Hreceq3} and \eqref{eqG}.  
 
\end{proof}

\bigbreak

\begin{remark} \label{remcoef} 
{\sl    Combining   Definition \ref {defspace3} together with  Lemma \ref {propG}, we infer   that for any integer~$k$,  the function  ${\mathfrak g}_{k}$ involved in the asymptotic formula \eqref {formaout}  admits an absolutely convergent  expansion of the form: \begin{equation} \label {beh0gk} {\mathfrak g}_{k}(
\rho) = \rho^{1-k } \sum_{j \geq  3}   \sum_{0 \leq \ell \leq \frac {j-3} 2}  a^k_{j,\ell}   \,  \rho^{ \nu j}  \,  \big(\log \rho\big)^\ell  \, ,\end{equation} 
for $\rho < \delta$, with some coefficients   $a^k_{j,\ell}$ satisfying
\beq
\label{coefas} a^0_{j,\ell}= \mu^0_{j,\ell}\, , \, \, a^1_{j,\ell}= \mu^1_{j,\ell}\,  \,  \mbox{if} \,\,  \, 3 \leq j \leq N \andf  a^0_{j,\ell}= a^1_{j,\ell}= 0 \,  \,  \mbox{if} \,\, \,  j \geq N+1 \, .\eeq }
\end{remark}

\medbreak

%%%%%%%%%%%%%%%%%%%%%%%%%%%%%%%%%%%%%%%%%%%%

%%%%%%%%%%%%%%%%%%%%%%%%%%%%%%%%%%%%%%%%%%%%

\subsection{Estimate of the approximate solution in the remote region}\label{Est3}
Under the above notations, set for any integer~$N \geq 3$\beq
\label{aprx3}u_{{\rm
out}}^ {(N)}(t, \rho) = \rho +  \sum^N_{k = 0} t^k  {\mathfrak g}_k(\rho)\, , \quad  V_{{\rm
out}}^ {(N)}(t,y) =   t^ {-(\nu+1)} \, u_{{\rm
out}}^ {(N)}(t,  t^ {\nu+1}  \, y ) \, . \eeq
  Invoking Lemma \ref {propG},  and recalling that for any integer $k$ the function ${\mathfrak g}_{k}$   is compactly supported in $\{ 0 \leq  \rho \leq  2 \delta \}$, we infer that  the function $V_{{\rm
out}}^ {(N)}$ defined by~\eqref{aprx3} satisfies the following~$L^{\infty}$   estimates in  the  remote region $$ \Omega_{{\rm
out}}:=  \big\{ Y \in  \R^4, \,y= |Y| \geq t^ {-\epsilon_2- \nu} \big\}  \, . $$
\begin{lemma}
\label {solVout}
{\sl For any multi-index $\alpha$, there exists $\delta_0(\alpha, N)>0$ such that for all positive  real  number $\delta \leq \delta_0(\alpha, N)$,   we have    \begin{eqnarray}  \label {property1out} \qquad \| \langle \cdot \rangle^{|\alpha|} \nabla^{\alpha}   (V_{{\rm
out}}^ {(N)} (t,  \cdot)-Q)\|_{L^{\infty}(\Omega_{{\rm
out}})} &\leq &C_{\alpha }  \, t^{-(\nu+1)} \,\delta^{3\nu+1}\,, \\ \label {property12out} \qquad \| \langle \cdot \rangle^{|\alpha|-1} \nabla^{\alpha}   (V_{{\rm
out}}^ {(N)} (t,  \cdot)-Q)\|_{L^{\infty}(\Omega_{{\rm
out}})} &\leq &C_{\alpha}  \, \delta^{3\nu}\,, \\ \label {property13out} \qquad \| \langle \cdot \rangle^{\beta} \nabla^{\alpha}   (V_{{\rm
out}}^ {(N)} (t,  \cdot)-Q)\|_{L^{\infty}(\Omega_{{\rm
out}})} &\leq &C_{\alpha, \beta}  \, \big(t^{3\nu(\nu+1)}+t^{\nu+1} \big)\,, \, \forall  \, \beta \leq  |\alpha|-2 \, ,\\  \label {property4out} \| \partial_t  V_{{\rm
out}}^ {(N)} (t,  \cdot)\|_{L^{\infty}(\Omega_{{\rm
out}})} &\leq  & C \, t^{-(\nu+2)} \,\delta^{3\nu+1} \,,   
\\ \label {property2out}\|   \langle \cdot \rangle^{|\alpha|} \nabla^{\alpha}V_{{\rm
out}, 1}^ {(N)}(t,  \cdot)\|_{L^{\infty}(\Omega_{{\rm
out}})} &\leq &C_{\alpha}  \, \delta^{3\nu}\,,\\ \label {property13out} \qquad \| \langle \cdot \rangle^{\beta} \nabla^{\alpha}   V_{{\rm
out}, 1}^ {(N)}(t,  \cdot)\|_{L^{\infty}(\Omega_{{\rm
out}})} &\leq &C_{\alpha, \beta}  \, \big(t^{3\nu(\nu+1)}+t^{\nu+1} \big)\,, \, \forall \, \beta \leq  |\alpha|-1 \, , \\ \label {property2tout}\|   \partial_tV_{{\rm
out}, 1}^ {(N)}(t,  \cdot)\|_{L^{\infty}(\Omega_{{\rm
out}})} &\leq &C  \, t^{-1} \,\delta^{3\nu}\, , \\  \label {property6out} \| \langle \cdot \rangle^{\beta}  \nabla^{\alpha}     V_{{\rm out}, 2}^ {(N)}(t,  \cdot)\|_{L^{\infty}(\Omega_{{\rm
out}})}   &\leq  & C_{\alpha, \beta}   \, \big(t^{3\nu(\nu+1)}+ \delta^{3\nu-1} \,t^{\nu+1} \big)\,, \, \forall \,\beta \leq  |\alpha|  \, ,
\end{eqnarray}
for all  $0 < t \leq T$ with  $T= T(\alpha,   \delta, N)$, and where $$V_{{\rm
out}, 1}^ {(N)}(t, y):= (\partial_t  u_{{\rm out}}^ {(N)})(t, \rho ) \, ,\quad V_{{\rm
out}, 2}^ {(N)}(t,y):= t^{\nu +1}   \,  (\partial^2_t  u_{{\rm out}}^ {(N)})(t, \rho)  \,.$$ 

\medskip \noindent Besides, for any multi-index $| \alpha|  \geq 1$
\begin{equation} \label {property5out} \|  \nabla^{\alpha} \partial_t  V_{{\rm
out}}^ {(N)} (t,  \cdot)\|_{L^{\infty}(\Omega_{{\rm
out}})} \leq  C_{ \alpha} \, t^{-1} \,\delta^{3\nu} \,,\end{equation}
for all  $0 < t \leq T$.}
\end{lemma}

\medbreak

Denote
\beq \label{outreg} \Omega^x_{{\rm
out}}:=  \big\{ x \in  \R^4, \, |x|   \geq t^ {1-\epsilon_2}  \big\} \, .
  \eeq
One has the following estimates  in $L^{2}$ framework:
\begin{lemma}
\label {solVoutl2}
{\sl  Under the above notations,   the following estimates occur  for  any  $0< \delta \leq \delta_0(\alpha, N)$ and all~$0 < t \leq T= T(\alpha,    \delta, N)$: \begin{eqnarray*}   \label {g0out} \qquad \big\|  \nabla_x^{\alpha}   \big(u_{{\rm
out}}^ {(N)} (t,  \cdot)- t^ {\nu+1} Q\Big(\frac {\cdot} {t^ {\nu+1}} \Big)-  {\mathfrak g}_0\big)\big\|_{L^{2}(\Omega^x_{{\rm
out}})} &\leq &C_\alpha  t    (1+t^{(1-\epsilon_2)(3 \nu+2-|\alpha|)}),    \, \forall \,  |\alpha|   \geq 1, \\ \label {g1out}   \big\|  \nabla_x^{\alpha}   \big(\partial^\ell_t u_{{\rm
out}}^ {(N)} (t,  \cdot) -  {\mathfrak g}_\ell\big)\big\|_{L^{2}(\Omega^x_{{\rm
out}})} &\leq & C_\alpha   t   (1+t^{(1-\epsilon_2)(3 \nu+2-\ell-|\alpha|)}),  \, \forall \,   |\alpha|  \geq 0, 
\end{eqnarray*}
for all $\ell=1,2$.}
\end{lemma}

\medbreak

\begin{remark} 
{\sl  Combining Formula \eqref{aprx3}   with Lemma  \ref {solVoutl2}, we infer  that $V_{{\rm
out}}^ {(N)} (t,  \cdot)$ satisfies, for all~$0 < t \leq T$
 \begin{equation} \label {property11l2out}  \|  \nabla^{\alpha}   (V_{{\rm
out}}^ {(N)} (t,  \cdot)-Q)\|_{L^{2}(\Omega_{{\rm
out}})} \leq C_\alpha  t^{(|\alpha|-3)\, (\nu+1)} [ \delta^{3\nu+3-|\alpha|} +   t^{(1- \epsilon_2) (3\nu +3 - |\alpha| )}  ],\,  \forall \alpha   \geq  1  \,, \end{equation}
 \begin{equation}  \label {property11l22out}  \| \nabla^{\alpha}   V_{{\rm
out},\ell}^ {(N)} (t,  \cdot)\|_{L^{2}(\Omega_{{\rm
out}})} \leq C_\alpha    t^{(|\alpha|-3+\ell)\, (\nu+1)} [\delta^{3\nu+3-\ell-|\alpha|}+  t^{(1- \epsilon_2) (3\nu +3 -\ell - |\alpha| )}],  \forall \alpha   \geq  0, 
\end{equation}for all $\ell=1,2$.} 
\end{remark} 

\bigbreak

Let us now consider  the remainder
  \begin{equation}
\label {remestap3}\ds {\cR}_{{\rm
out}}^ {(N)}:=  \eqref{eqpart1} V_{{\rm
out}}^ {(N)} \, .  \end{equation}
We have
$${\cR}_{{\rm
out}}^ {(N)}(t,y) = t^{\nu +1 }      \wt{\cR}_{{\rm
out}}^ {(N)}  (t,      t^{\nu +1 } y ) $$
where 
$$ \wt{\cR}_{{\rm
out}}^ {(N)}(t,y)=  [\eqref {eq:NW} u_{{\rm
out}}^ {(N)}] (t, t^{\nu +1 } y)\, .$$
It follows readily from the proof of Lemma \ref {propG} that 
\begin{equation}\label {Condremout} \|| \cdot |^{\frac 3 2}  \,  \nabla_x^{\alpha}  \wt \cR_{{\rm
out}}^ {(N)}(t, \cdot)\|_{L^2(\Omega^x_{{\rm
out}})} \leq C_{ \alpha, N} \, t^{N-1 - (1-\epsilon_2)(| \alpha | +N-3  \nu - \frac 7 2)}   \, ,\end{equation}
for any $| \alpha | \geq 0$, provided that~$\ds N \geq 3  \nu + \frac 7 2 $ which leads to the following lemma:
  \begin{lemma}
\label {estap3}
{\sl For any multi-index $\alpha$,     the  following estimate holds : \begin{equation}\label {Condremout} \|\langle \cdot \rangle^{\frac 3 2}  \,  \nabla^{\alpha}   \cR_{{\rm
out}}^ {(N)}(t, \cdot)\|_{L^2(\Omega_{{\rm
out}})} \leq t^{\epsilon_2 N  - \frac 5 2 (\nu+1) }   \, ,\end{equation}
for all  $0 < t \leq T=T(\alpha,      \delta, N)$, provided that~$\ds N \geq 3  \nu + \frac 7 2 \cdot$}
\end{lemma}

\medbreak

We next  investigate $V_{{\rm out} }^ {(N)} - V_{{\rm ss} }^ {(N)}$  in 
$\Omega_{{\rm
out}} \cap \Omega_{{\rm
ss}}$. Assuming $\rho < \delta$, and rewriting $u^ {(N)}_{{\rm out} }$ in terms of the variable $\ds z=\frac {\rho} {\lam^{(N)}(t)}\virgp$ we get:
$$
 \begin{aligned}
 \ds & u_{{\rm
out}}^{(N)}(t, \rho) =  \lam^{(N)}(t) \biggl[ z+  \sum^{N}_{k = 3}  \sum_{0   \leq \ell \leq \frac  {k -3} 2 } t^{\nu k} (\log t)^\ell \biggl(\sumetage {3  \leq  \beta  \leq k-3}{0  \leq  \alpha  \leq  \frac {k -6} 2 -\ell , \,  p \geq 0} w_{k, \ell, \alpha, \beta, p}^{{\rm
out}} (\log z)^\alpha z^{ \nu \beta+1-p}\\
& \qquad  \qquad  \qquad   \qquad  \qquad  \qquad  +  z^{ \nu k+1 } \sum_{0  \leq  \alpha  \leq  \frac {k -3} 2 -\ell , \,  p \geq 0} w_{k, \ell, \alpha,   p}^{{\rm
out}} (\log z)^\alpha z^{-p} \biggr)\biggr] \virgp
 \end{aligned}
 $$
with some coefficients $w_{k, \ell, \alpha, \beta, p}^{{\rm
out}}$, $w_{k, \ell, \alpha,   p}^{{\rm
out}}$ that can be expressed explicitly in terms of the coefficients~$(\lam_{j, \ell}
)$, for~$3\leq j   \leq N, 0 \leq \ell \leq \ell(j)$ and of the constants $(a^k_{j, \ell}
)$,  $k \geq 0,   j   \geq 3, 0 \leq \ell \leq \ell(j)$    introduced in Remark~\ref{remcoef}. 

\medskip

In particular
\begin{equation}\label {coefimpout} w_{k, \ell, \alpha,   p}^{{\rm
out}}=  \begin{pmatrix} \alpha+ \ell  \\ \alpha \end{pmatrix} a_{k,  \alpha+ \ell}^p \, ,\end{equation}
for all $k \geq 3$, $\ds \ell \leq \frac {k -3} 2\virgp$ $\ds \alpha \leq \frac {k -3} 2 - \ell \virgp$ $p\geq 0\cdot$

\medskip \noindent Combining \eqref{eqform}, \eqref{coefas} with \eqref {coefimpout}, we infer that
$$ \sum_{0   \leq j \leq \ell(k) }   \big(  \alpha^{j,+}_{k}    f^{j,+}_{k, \ell}  + \alpha^{j,-}_{k}    f^{j,-}_{k, \ell}\big)= \sumetage {0   \leq \alpha \leq \frac {j -3} 2- \ell  } {p=0,1 } z^{ \nu k+1 -p} (\log z)^\alpha w_{k, \ell, \alpha,   p}^{{\rm
out}}+ \cO(z^{ \nu k-1} (\log z)^{\ell(j)-\ell} )  \,,$$
as  $z  \to 0$, which by Proposition \ref {STKL} $(2)$ (uniqueness around infinity) implies that 
\begin{equation}\label {coefimpoutbis} w_{k, \ell, \alpha, \beta, p}^{{\rm
out}}=  w_{k, \ell, \alpha, \beta,  p} \,, \, \,w_{k, \ell, q,   p}^{{\rm
out}}=  w_{k, \ell, q,   p} \, ,\end{equation}
for any $3 \leq k \leq N$, $0 \leq \ell \leq \ell(k)$, $0 \leq \alpha \leq \frac {k -6} 2- \ell \virgp$ $0 \leq q \leq \frac {k -3} 2- \ell \virgp$ $3  \leq  \beta  \leq k-3$, $p\geq 0$, where~$w_{k, \ell, \alpha, \beta,  p} $, $w_{k, \ell, q,   p}$ are the coefficients involved in \eqref{w_klasyin}.

\medskip As a direct consequence of \eqref {coefimpoutbis}, we obtain
\begin{lemma}
\label {difestap3}
{\sl For any multi-index $\alpha \in \N^4$ and any integer $m$,  we have  \begin{equation}\label {estdiffout} \|\partial_t^m \nabla^\alpha  (V_{{\rm out} }^ {(N)} - V_{{\rm ss} }^ {(N)})(t, \cdot)\|_{L^\infty(\Omega_{{\rm
out}} \cap \Omega_{{\rm
ss}})} \leq  t^{ -m -\nu+ |\alpha |(\nu+\epsilon_2)}\, \big( t^{ \epsilon_2 N  } +t^{- \epsilon_2+    (1-\epsilon_2) \nu N }   \big) \, ,\end{equation}
for all $0<t<T=T(\alpha, m, N)$.}
\end{lemma}

\medbreak

\section{Approximate solution in the whole space} \label{apsol}
Let ~$\Theta$ be a radial  function in $\cD(\R)$ satisfying$$\Theta(\xi)= \left\{
 \begin{array}{l}1  \quad \mbox{if}  \quad |\xi | \leq \frac 14\\
 0  \quad \mbox{if}  \quad |\xi | \geq \frac 12 \,  \cdot 
\end{array}
\right.$$

Set
\beq
\label {genap}
\begin{split}
&   V^ {(N)}(t,y):= \Theta 
\big( y \, {t^{\nu- \epsilon_1}}\big)V_{{\rm in} }^ {(N)})(t,y) \\
& \qquad\qquad + \big(\Theta 
\big( y \, {t^{\nu+ \epsilon_2}}\big)-\Theta 
\big( y \, {t^{\nu- \epsilon_1}}  \big)\big)V_{{\rm ss} }^ {(N)})(t,y) +\big(1- \Theta 
\big( y \, {t^{\nu+ \epsilon_2}}\big)\big) V_{{\rm out} }^ {(N)})(t,y)
\,  \virgp\\
& u^ {(N)}(t,\rho):= t^{\nu+ 1} V^ {(N)}\big(t,\frac \rho {t^{\nu+ 1}}\big) \,  \cdot
\end{split}
\eeq

Combining Lemmas \ref   {solVin},  \ref {solVss} and  \ref {solVout} together with Lemmas  \ref {difestap2} and \ref {difestap3}, we infer that for $N$ sufficiently large there exists a positive parameter $\delta_0(N)$ such that for any $\delta \leq \delta_0(N)$ there exists   a positive time $T=T(\delta, N)$ so that the above approximate solution $V^ {(N)}$ defined by \eqref {genap} satisfies  the  following $L^{\infty}$ estimates:

\begin{lemma} \label {ap3}
{\sl  The following estimates hold for $V^ {(N)}$, for all   $0 < t \leq T$
\begin{eqnarray} \label{use0} \qquad  \qquad\| \langle \cdot \rangle^{|\alpha| -1} \, \nabla^{\alpha} \,    (V^ {(N)}-Q)(t,  \cdot)\|_{L^{\infty}(\R^4)} &\leq& C \,    \delta^{3\nu}\,,  \,  \forall   \, 0\leq |\alpha| < 3 \nu+4 \, .\\ \label{use1} \qquad \qquad \| \langle \cdot \rangle^{\beta} \nabla^{\alpha}   (V^ {(N)}-Q) (t,  \cdot)\|_{L^{\infty}(\R^4)} &\leq& C   \,   t^{\nu }  \,,  \,  \forall   \, 1\leq |\alpha| < 3 \nu +4 \,\, \mbox{and} \,\, \beta \leq  |\alpha|-2. \\ \label{use11} \|\nabla^{\alpha}  \frac {y } {|y|^2} \cdot \nabla (V^ {(N)}-Q) (t, \cdot)  \|_{L^{\infty}(\R^4)} &\leq& C   \, t^{\nu } \,, \,  \forall   \, 0\leq |\alpha| < 3 \nu+3 \,   .\end{eqnarray}
Besides the time derivative of $V^ {(N)}$ satisfies 
\begin{eqnarray}  \label{use6} \|   \partial _tV^ {(N)}(t,  \cdot)\|_{L^{\infty}(\R^4)} &\leq& C\, t^{-2-\nu}  \, \delta^{3\nu+1}\\ \label{use7} \|  \nabla^{\alpha} \partial _t V^ {(N)}(t,  \cdot)\|_{L^{\infty}(\R^4)} &\leq& C  \,t^{-1}  \, \delta^{3\nu}\,, \,  \forall  \,  \, 1\leq |\alpha| < 3 \nu +3 \,. \end{eqnarray} In addition for any multi-index $\alpha$ of length $|\alpha|   < 3 \nu + 3$,  the function $V_1^ {(N)}(t,y):= (\partial_t u^ {(N)})(t,  \rho )$  and its time derivative verify \begin{eqnarray} \label{use2} \| \langle \cdot \rangle^{|\alpha|}  \nabla^{\alpha}    V_1^ {(N)}\|_{L^{\infty}(\R^4)} &\leq& C \,   \delta^{3\nu}\,,   \\ \label{use3} \| \langle \cdot \rangle^{\beta}  \nabla^{\alpha}    V_1^ {(N)}(t,  \cdot)\|_{L^{\infty}(\R^4)} &\leq& C  \, t^{\nu} \,, \,  \forall   \, \beta \leq  |\alpha|-1\, , \\  \label{use6*} \|   \partial _t V_1^ {(N)}(t,  \cdot)\|_{L^{\infty}(\R^4)} &\leq& C\, t^{-1}  \delta^{3\nu} \,.\end{eqnarray}  Finally for any multi-index $\alpha$ of length $|\alpha|   <3 \nu+ 2$ and  any integer $\beta \leq  |\alpha|$, we have  \begin{eqnarray}  \label{use4} \| \langle \cdot \rangle^{\beta}  \nabla^{\alpha}    V_2^ {(N)}(t,  \cdot)\|_{L^{\infty}(\R^4)} &\leq& C  t^{ \nu} \,, 
\end{eqnarray} where $V_2^ {(N)}(t,y):= t^{\nu +1 }(\partial^2_t u^ {(N)})(t, \rho)$. }\end{lemma}

\medbreak

 Along the same lines,  taking advantage of Lemmas \ref {solVinl2},  \ref {solVssl2} and \ref {solVoutl2},  we get  the following $L^2$ estimates, as before for $N$ sufficiently large, $\delta \leq \delta_0(N)$ and $0<t \leq T(\delta, N)$:  \begin{lemma}
\label {solVapl2}
{\sl  For  any $ 1\leq |\alpha| < 3 \nu+3$, we have \begin{equation}   \label {property11l2ss}  \|\nabla^{\alpha}   (u^ {(N)} (t,  \cdot)-t^ {\nu+1} Q\big( \frac {\cdot} {t^ {\nu+1}} \big)-  {\mathfrak g}_0 )\|_{L^{2}((\R^4)} \leq C  \big(t+ t^{(1-\epsilon_2)(3\nu+3-|\alpha|)}+t^{3+5\nu-|\alpha|(1+ \nu)}\big)  , \end{equation} 
and  for any  $0\leq |\alpha| < 3 \nu+2$ \begin{equation} \label {property5l2ss} \|   \nabla^{\alpha}   (u_{ t}^ {(N)} (t,  \cdot)-{\mathfrak g}_1) \|_{L^{2}(\R^4)} \leq   C  \big(t+ t^{(1-\epsilon_2)(3\nu+2-|\alpha|)}+t^{2+3\nu-|\alpha|(1+ \nu)}\big)  \,.  
\end{equation}
  Besides,  we have  
\begin{eqnarray} \label {property11l2ssV} \|  \nabla^{\alpha}   (V^ {(N)} (t,  \cdot)-Q)\|_{L^{2}(\R^4)} &\leq&  C  t^{2\nu },  \, \, \forall \, 3 \nu+3<  |\alpha| < 3 \nu+4 + \frac 1 2 \virgp \\   \label {property11l2ssV1} \|  \nabla^{\alpha}   V_1^ {(N)} (t,  \cdot)\|_{L^{2}(\R^4)} &\leq&  C \,    t^{\nu }, \, \, \forall \, 3 \nu+2<  |\alpha| < 3 \nu+3 + \frac 1 2 \virgp \\   \label {property11l2ssV2} \|  \nabla^{\alpha}   V_2^ {(N)} (t,  \cdot)\|_{L^{2}(\R^4)} &\leq&  C \,    t^{2\nu },  \, \, \forall \, 3 \nu+1<  |\alpha| < 3 \nu+2 + \frac 1 2 \cdot\end{eqnarray} }
\end{lemma}

\medbreak

\begin{remark} \label{rem1}
{\sl  Lemma \ref {solVapl2} implies that
\begin{eqnarray}   \label {property11l2out} \qquad \|  \nabla^{\alpha}   (V^ {(N)} (t,  \cdot)-Q)\|_{L^{2}(\R^4)} &\leq &C \, \big(   t^{2\nu } + t^{(1+\nu)(|\alpha|-3) }\delta^{3\nu+3-|\alpha| }\big)\,,\, \, \forall  1 \leq  |\alpha| < 3\nu+4 +\frac 1 2 \virgp \\ \label {property11l22out} \qquad \| \nabla^{\alpha}  V_1^ {(N)} (t,  \cdot)\|_{L^{2}(\R^4)} &\leq &C\,  \big(   t^{\nu } + t^{(1+\nu)(|\alpha|-2) }\delta^{3\nu+2-|\alpha| }\big)\,,\, \, \forall 0 \leq  |\alpha| < 3\nu+3 +\frac 1 2 \virgp
\end{eqnarray}
and \begin{eqnarray*} \|\nabla^{\alpha}   (u^ {(N)} (t,  \cdot)-t^ {\nu+1} Q\big( \frac {\cdot} {t^ {\nu+1}} \big)-  {\mathfrak g}_0 )\|_{L^{2}(\R^4)} &\stackrel{t\to0}\longrightarrow 0&\,,\, \, \forall\, \,  1 \leq  |\alpha| < 3+ \frac {2 \nu} {\nu +1} \, \virgp \\ \| \nabla^{\alpha}   (u_{ t}^ {(N)} (t,  \cdot)-{\mathfrak g}_1)\|_{L^{2}(\R^4)} &\stackrel{t\to0}\longrightarrow 0&\,,\, \, \forall\, \,  0 \leq  |\alpha| < 2+ \frac {  \nu} {\nu +1} \, \cdot\end{eqnarray*}} 
\end{remark} 

\medbreak

\bigskip

Finally, if we denote by  
$$\cR^ {(N)}:= \eqref{eqpart1}\, V^ {(N)}\,,$$
 then invoking Lemmas \ref {estap1},  \ref {estap2}, \ref {difestap2},  \ref {estap3} and \ref{difestap3}, we infer that  the following result holds:
\begin{lemma} \label {genremest}
{\sl There exist   $N_0 \in \N$ and $\kappa>0$ such that \begin{equation}\label {CondRNap} \|\langle \cdot \rangle^{\frac 3 2}  \,  \cR^ {(N)}(t, \cdot)\|_{H^{K_0} (\R^4)} \leq t^{\kappa N + \nu}  \, ,\end{equation} 
for all~$0< t \leq T(\delta, N )$, where $K_0= [ 3\nu   +  \frac 5 2 ]$ denotes the integer introduced in Lemma \ref {estap2}.  }\end{lemma}

\medbreak

Re-denoting $N$, one can always assume that the approximate solutions $u^ {(N)}$ are defined and satisfy Lemmas \ref {ap3}-\ref  {solVapl2} for any integer $N \geq 1$, and that \eqref{CondRNap} holds with $\kappa=1$ for all $N \geq 1$.

%%%%%%%%%%%%%%%%%%%%%%%%%%%%%%%%%%%%%%%%%%%%

\section{Proof of the blow up result} \label{end}

\subsection{Key estimates} \label{keyest}

The approximated solutions $u^ {(N)}$ constructed in the previous sections verify, for any integer $N \geq 1$
$$ (\nabla (u^ {(N)} -Q), \partial_t u^ {(N)}) \in C(]0,T], H^{K_0+1}(\R^4))\,  , $$
for \footnote{ In what follows,  $\delta$ is assumed to be less than $\delta_0(N)$, which may vary from line to line.} some $T=T(\delta, N)>0$. Furthermore, by \eqref{use0}, \eqref{use3},
 there are positive constants~$c_0$ and~$c_1$  such that 
\begin{equation}
\label {Cond1} u^ {(N)}(t,\cdot)\geq  c_0\,  t^{\nu+1} \,  \andf \end{equation}
\begin{equation}
\label {Cond2}  (1+|\nabla u^ {(N)}|^2-(\partial_t u^ {(N)})^2) (t,\cdot) \geq  c_1 \, \virgp\end{equation}
for any $N \geq 1$, and  all $t$ in~$]0,T]$. This ensures that    $$(u^ {(N)}(t,\cdot), \partial_t u^ {(N)}(t,\cdot))\in X_{K_0+2} \, , \forall t \in ]0,T] \, .$$

\bigskip The goal of this paragraph is to  achieve the proof of Theorem \ref {main} by complementing these approximate solutions  $u^ {(N)}$ to an actual solution ~$u$ to   the quasilinear wave equation  \eqref {eq:NW} which blows up at $t=0$, and  which for~$N$   fixed large enough is close to   $u^ {(N)}$, in the sense that  there is  a positive time~$T=T(\delta, N)$ such that  the following estimate holds  \begin{equation} \label{eq:exact}\|\langle \cdot \rangle^\frac 3 2 \, \partial_t(u -  u^ {(N)})(t,\cdot)\|_{H^{L_0-1} (\R^4)}+ \|\langle \cdot \rangle^\frac 3 2 \,  \nabla(u -  u^ {(N)})(t,\cdot)\|_{H^{L_0-1} (\R^4)} \leq t^{\frac N 2} \, , \end{equation}
 for all time $t$ in~$]0,T]$, where  
 the regularity index\footnote{ we take L to be odd just to make  the estimates we are dealing more easier, but it is not important.} 
$L_0=2M+1$, with $\ds M= [\frac  {K_0} 2] \cdot$ 

\smallskip \noindent Note that since $\ds \nu > \frac 1 2\virgp$ we have~$M \geq 2$, and thus $L_0\geq 5$.

 \bigskip The mechanism for achieving this will rely on 
  the following crucial  result:
\begin{proposition}
\label {propestap4}
{\sl There is $N_0$ in $\N$ such that for any integer $N \geq N_0$, there exists a small positive time $T=T(\delta, N)$  such that, for any time~$0< t_1 \leq T$, the Cauchy problem:
\begin{equation}
\label {Cauchy1}
{\rm(NW)}^{(N)} \left\{
\begin{array}{l}
\eqref {geneq} \,u=0 \,, \\
{u}_{|t=t_1}= u^ {(N)}(t_1, \cdot)\\
(\partial _t u)_{|t=t_1}= \partial_t u^ {(N)}(t_1, \cdot)\, 
\end{array}
\right.
\end{equation}
admits a unique solution $u$ on the interval ~$[t_1,T]$  which satisfies
\begin{equation}
\label {eq:key}\|\langle \cdot \rangle^\frac 3 2 \,   \partial_t(u -  u^ {(N)})(t,\cdot)\|_{H^{L_0-1} (\R^4)}+ \|\langle \cdot \rangle^\frac 3 2 \,  \nabla(u -  u^ {(N)})(t,\cdot)\|_{H^{L_0-1} (\R^4)} \leq t^{\frac N 2} \, , \end{equation}
 for all $t_1 \leq t \leq T$.}
\end{proposition}

\medbreak

\begin{proof}
As mentioned   above, for any $t_1$ sufficiently small,  the initial data  $(u^ {(N)}(t_1, \cdot), \partial_t u^ {(N)}(t_1, \cdot))$ belongs to $X_{K_0+2}$, and thus satisfy the hypothesis of Theorem \ref {Cauchypb}.  By construction $u^ {(N)}(t, \rho) - \rho$ is compactly supported. Thus,   to prove Proposition~\ref {propestap4}, it is enough to show that there exists  a time~$T=T(\delta, N)>0$ such that the solution to the Cauchy problem  \eqref {Cauchy1} satisfies the energy estimate \eqref {eq:key}, for any time $t_1\leq t < \min \big\{T(\delta, N), T^*\big\}$, where $T^*$ is the maximal time of existence.
    This will be achieved in  two steps:
\begin{enumerate}
\item First writing $u(t,x)= t^{\nu+1} V(t,y)$,  $V(t,y)= V^ {(N)}(t,y)+ \varepsilon^ {(N)}(t,y)$, with $\ds y= \frac x {t^{\nu+1}}\virgp$ $x \in \R^4$,  we derive the equation satisfied by the remainder term $\varepsilon^ {(N)}$. 
 We next set
 $$\ds \varepsilon^ {(N)}(t, y)= H (y) \, r^ {(N)}(t, y) \, \virgp $$ where $H$ is the function defined by \refeq {changfunct},  and rewrite the obtained equation in terms of~$r^ {(N)}$. As we will see later, the equation for $r^ {(N)}$ involves  the operator ${\mathfrak L}$ introduced in~\eqref {redlineartrans}.
\smallskip
\item We deduce the desired result (inequalities  \eqref {eq:key})  by suitable energy estimates by making use of the behavior  of the approximate solution $u^ {(N)}$
described by Lemmas  \ref{ap3}, \ref {solVapl2}, and the spectral properties of the operator ${\mathfrak L}$ which turns out to be   close to the Laplace operator. 
\end{enumerate}

\medskip In order to make notations as light as possible, 
 we shall    omit  in the sequel the dependence of the functions $\varepsilon^ {(N)}$ and  $r^ {(N)}$  on $N$. 
 
 \medskip
 Denote by 
\begin{eqnarray}\label{V1}  V_1 (t, y): &=& a(t) \, V_t (t, y) +a'(t)\, \Lambda V  (t, y)=u_{t} (t,x) \, 
\\ \label{V2} V_2 (t, y): &=& a(t) \, (V_1)_t (t, y)- a'(t) \,  (y \cdot \nabla V_1) (t, y)= t^ {\nu+1} \, u_{tt} (t,x) \, \virgp
\end{eqnarray}
with $a(t)=t^ {\nu+1}$ and $\Lambda V=V- y \cdot \nabla V$.

\medskip
By straightforward computations,  we readily gather 
 that   the quasilinear wave equation
~\eqref {geneq} multiplied  by $a(t)$ undertakes the following form    in terms of the function $V$  with respect to  the  variables~$\ds (t,y)=\Big(t, \frac  {x} {  t^{1+\nu}} \Big)$ \beq  \begin{aligned}
 \label {eq:NW*} (1+ |\nabla V|^2)V_{2}& - 2 (\nabla V \cdot  \nabla V_1) V_{1}   -(1 - V^2_{1} + |\nabla V|^2) \,\Delta V  \\
& \qquad   \qquad+ \sum^{4}_{j, k=1} V_{y_j} V_{y_k} \partial^2_{y_j y_k} V+ \frac 3 V (1 - V^2_{1}+  |\nabla V|^2)= 0 \,\cdot  \end{aligned} \eeq 
Thus recalling   that the approximate solution  $V^ {(N)}$ satisfies \eqref {eq:NW*} up to a remainder term $\cR^ {(N)}$, we infer that
 saying that the function $u$ solves the equation $\eqref {geneq} \,u=0$ is equivalent to say that the remainder term  $\varepsilon$ satisfies the following equation:
 \begin{equation}\label{defeqtaubis} \begin{aligned}
  \ds & (1+|\nabla V|^2)\, \varepsilon_2 - {\mathcal
L} \varepsilon  - 2 V_1 \nabla V \cdot  \nabla \varepsilon_1    + (V^2_{1}- |\nabla \varepsilon|^2) \,\Delta \varepsilon 
 \\
&  \ds \qquad \qquad \qquad \qquad \qquad \qquad \qquad \qquad \qquad \qquad + \sum^{4}_{j, k=1} \varepsilon_{y_j} \varepsilon_{y_k} \partial^2_{y_j y_k} \varepsilon+\cF+  \cR^ {(N)}=0 \,,\end{aligned} \end{equation}
where \begin{equation}\label{def2}
\ds \varepsilon_2= a(t) \, (\varepsilon_1)_t - a'(t) \,   (y \cdot \nabla \varepsilon_1) \virgp \quad \varepsilon_1=  a(t) \, \varepsilon_t  +a'(t) \, \Lambda \varepsilon\, ,\end{equation}   with   ${\mathcal
L}$    the linearized operator  introduced in  \eqref{linQ}:
$$ {\mathcal
L} \varepsilon = \Delta \varepsilon  +3 \Big(  \frac {\big( 3 \, y \cdot \nabla Q\big) \, \nabla Q} {|y|^2}- \frac {2 \,  \nabla Q} {Q}\Big) \cdot  \nabla \varepsilon+ 3 \frac {1+ |\nabla Q|^2} {Q^2}  \,  \varepsilon   \,   \virgp$$
and where   the term $ \cF$  is given by:
$$\begin{aligned}
 & \ds  \cF= (|\nabla V|^2-|\nabla V^ {(N)} |^2) \,V_2^ {(N)}- 2  \,(V_1 \nabla V-   V_1^ {(N)}\nabla V^ {(N)})\cdot \nabla V^ {(N)}_1 +  (V_1^2- (V_1^ {(N)})^2) \Delta V^ {(N)}\\
& \qquad \qquad  \ds  -    \frac 3 {V\, V^ {(N)}} \big(V_1^2 \, V^ {(N)}- (V_1^ {(N)})^2 \,V\big) + 3  \Big[\frac 1 {V^ {(N)}} \big(|\nabla V|^2-|\nabla V^ {(N)} |^2\big) - \frac 2 Q \,\nabla Q \cdot \nabla \varepsilon \Big]\\
& \qquad  \ds - 3 \,\varepsilon \Big[\frac {(1+ |\nabla V|^2)} {VV^ {(N)}} - \frac {(1+ |\nabla Q|^2)} {Q^2} \Big]    - 9\, (|\nabla V^ {(N)} |^2- |\nabla Q |^2) \frac {y \cdot \nabla \varepsilon} {|y|^2} - 9  \frac {y \cdot \nabla V^ {(N)}} {|y|^2} |\nabla \varepsilon|^2\,. \end{aligned}$$

 \medskip 
Next,   set
\begin{equation}\label{change} \varepsilon(t,y)= H( y)\, r(t, y) \, ,\end{equation}
with
\begin{equation}\label{defH}  H=  \frac  {(1+  |\nabla Q|^2)^{\frac  1 4}} { Q^{\frac  3 2}} \, \cdot \end{equation}
 Let us emphasize that in view of Lemma \ref {ST}, the above function
~$H$ enjoys the following property: for any multi-index $\al$ in $\N^4$,  there exists a   positive constant $C_\alpha $ such that,  for any $y$ in $\R^4$ the following estimate holds
\begin{eqnarray} \label{H1} \frac 1 {C_\alpha \,\langle y \rangle^{\frac 3  2+|\alpha|} }   \leq  |\, \nabla^{\alpha} H(y)| \leq \frac  {C_\alpha} {\langle y \rangle^{\frac 3  2+|\alpha|}}  \,  \cdot \end{eqnarray} 
Now in light of the definitions introduced in \eqref{def2}, we have
$$ \varepsilon_1(t,y)= H( y) \, r_1(t, y) \virgp \quad \varepsilon_2(t,y)= H( y) \, r_2(t, y) \, ,$$
where
\begin{equation}\label{def12}
\ds r_1= a\, r_t  +  a'\, \Lambda r-   a'  \, \frac  { y \cdot  \nabla H} H  \,  r \virgp \quad r_2=  a\, (r_1)_t  - a'\, y \cdot  \nabla r_1 - a'  \, \frac  { y \cdot  \nabla H} H  \, r_1  \, \cdot
\end{equation}
Thus taking advantage of \eqref{defeqtaubis}, we readily gather that  the remainder term  $r$ given by  \eqref{change} satisfies
\begin{equation}\label{defeqfinal0}
 \begin{aligned}
 & \ds
(1+|\nabla V|^2)\,  r_2 + (1+|\nabla Q|^2)\,{\mathfrak L} r -\frac{2 V_1}{H}\nabla V\cdot\nabla (H r_1) +   
  (V^2_{1}-|\nabla (Hr)|^2) \,\Delta r \\
&\ds -\frac{2 V_1}{H}\nabla V\cdot\nabla (H r_1)  +\frac  {V^2_{1}-|\nabla (Hr)|^2}{H}  [\Delta, H] \, r
+ \sum^{4}_{j, k=1} (Hr)_{y_j} (Hr)_{y_k} \partial^2_{y_j y_k} r\, \\
&  \ds  \qquad \qquad \qquad \qquad \qquad  \ds+  \sum^{4}_{j, k=1} \frac{(Hr)_{y_j} (Hr)_{y_k}}{H} [\partial^2_{y_j y_k},
H]r+
\frac  {\cF}{H}   + \frac  { \cR^ {(N)}}{H}=0\, \virgp\end{aligned}
\end{equation}
where $[A, B]= AB-BA$ denotes the commutator of the operators $A$ and $B$, and where \begin{equation}\label{rel2L} {\mathfrak L}= - \frac 1 {H (1+  |\nabla Q|^2)} \, {\mathcal
L}  H \,\cdot\end{equation}
Let us recall that in view of \eqref{redlineartrans}
$${\mathfrak L} = - q \, \Delta  \, q +   \cP \, \virgp$$
with $\ds q=\frac  1 {(1+ |\nabla Q|^2)^ {\frac 1 2}}\virgp$ and   $ \cP$   a  radial $\cC^\infty$ function which  satisfies $$\ds \cP   =  - \frac 3 {8 \rho^2} \big(1 + \circ(1)\big) \, \virgp $$ as $\rho$ tends to infinity.

\bigskip Now dividing the  equation  at hand  by $(1+|\nabla V|^2)$,  
we infer that the function $r$ solves the following equation: 
\begin{equation}\label{defeqfinal}
 \begin{aligned}
 & \ds
  {r}_2  + \frac  {1+|\nabla Q|^2}{1+|\nabla V|^2} \,{\mathfrak L} r -  \frac  {2 \,V_1}{1+|\nabla V|^2}  \nabla V \cdot  \nabla r_1
+   
\frac  {V^2_{1}-|\nabla (Hr)|^2}{1+|\nabla V|^2}  \,\Delta r \\
&\ds +  \frac{1}{1+|\nabla V|^2} 
\sum^{4}_{j, k=1} (Hr)_{y_j} (Hr)_{y_k} \partial^2_{y_j y_k} r + {\wt \cF} + { \wt \cR^ {(N)}}=0\, \virgp\end{aligned}
\end{equation}
where  \beq  
 \label {eq:tildR} \ds { \wt \cR^ {(N)}}:= \frac  { \cR^ {(N)}}{(1+|\nabla V|^2)\,  H}  \,  \virgp \eeq 
and
\beq  \begin{aligned}
 \label {eq:finalF}  {\wt \cF}:= &   \frac  {\cF}{(1+|\nabla V|^2)\,  H} -\frac{2 V_1}{(1+|\nabla V|^2)\,H}\nabla V\cdot (\nabla H) \,r_1
\\
&+
\frac  {V^2_{1}-|\nabla(Hr)|^2}{(1+|\nabla V|^2)\,  H}  [\Delta, H] \, r  +
\sum^{4}_{j, k=1} \frac{(Hr)_{y_j} (Hr)_{y_k}}{(1+|\nabla V|^2)\,H} [\partial^2_{y_j y_k},
H]r \, \cdot \end{aligned} \eeq 

Note that we split Equation \eqref{defeqfinal0}  into a first  part which behaves as a quasilinear  wave equation and a second part  depending only on the remainder term $r$ and its   first derivatives. This achieves the goal of the first step.

\medskip

The proof of  energy inequalities \eqref {eq:key} is based on   suitable priori estimates. These priori estimates are established by combining the  key properties of  the operator ${\mathfrak L}$  stated  page \pageref  {potspect}  (and established in  Appendix \ref{end1}) together with the asymptotic formula~\eqref {4} as well as some properties of  the approximate solution  $V^ {(N)}$.  
We shall argue by bootstrap argument by proving the following key result: 
\begin{lemma}
\label {keybootstrap}
{\sl There is  $N_0$ in $\N$  such that for any integer $N \geq N_0$, there exists~$T=T(N,\delta)>0$  such that   for any $t_1\in ]0, T]$, and any $t_2\in [t_1, T]$ the following property holds. 

\medskip  \noindent 
If we have for all time $t$ in $[t_1,t_2]$, 
 \begin{equation}    \label{eqbootst} \| r_1 (t, \cdot )\|^2_{H^{L_0-1}(\R^4)} + \| \nabla r (t, \cdot )\|^2_{H^{L_0-1}(\R^4)} \leq t^{2N}  \, \virgp\end{equation}
then 
\begin{equation}    \label{eqbootst2} \| r_1 (t, \cdot )\|^2_{H^{L_0-1}(\R^4)} + \| \nabla r (t, \cdot )\|^2_{H^{L_0-1}(\R^4)} \leq \frac C {N}  \, t^{2N}  \, \virgp\end{equation} for any time $t$ in $[t_1,t_2]$,   where $C$ is an absolute constant. }
\end{lemma}

\medbreak

\begin{proof}[Proof of Lemma \ref {keybootstrap}] 
In order to establish  Inequality \eqref{eqbootst2}, let us  start by  applying the operator~ ${\mathfrak L}^M$ to Equation  \eqref{defeqfinal}. This  gives rise to   \begin{equation}\label{defeqfinalM}
 \begin{aligned}
 & \ds
 {\mathfrak L}^M \, {r}_2  + \frac  {1+|\nabla Q|^2}{1+|\nabla V|^2} \,{\mathfrak L}^{M+1} r -  \frac  {2 \,V_1}{1+|\nabla V|^2}  \nabla V \cdot  \nabla {\mathfrak L}^M r_1
+   
\frac  {V^2_{1}-|\nabla \varepsilon|^2}{1+|\nabla V|^2}  \,\Delta {\mathfrak L}^M r \\
&\ds +  \frac{1}{1+|\nabla V|^2} 
\sum^{4}_{j, k=1} (Hr)_{y_j} (Hr)_{y_k} \partial^2_{y_j y_k} ({\mathfrak L}^M r)+ {\wt \cF}_M +{\mathfrak L}^M { \wt \cR^ {(N)}}=0\, \virgp\end{aligned}
\end{equation}
with ${\wt \cF}_M=   {\mathfrak L}^M \, {\wt \cF} + \cG_M$, where 
\beq  \begin{aligned}
 \label {eq:finalFM}  \cG_M:= &   \Big[{\mathfrak L}^M, \frac  {1+|\nabla Q|^2}{1+|\nabla V|^2}\Big]\,{\mathfrak L} r  - 2 \, \Big[{\mathfrak L}^M, \frac  {V_1}{1+|\nabla V|^2}  \nabla V \cdot  \nabla \Big]\,r_1 
\\
&+ \Big[{\mathfrak L}^M, \frac  {V^2_{1}-|\nabla \varepsilon|^2}{1+|\nabla V|^2}  \,\Delta\Big]\,r  + \sum^{4}_{j, k=1}  \Big[{\mathfrak L}^M, \frac{(Hr)_{y_j} (Hr)_{y_k}}{1+|\nabla V|^2}  \partial^2_{y_j y_k}\Big]\,r \, .\end{aligned} \eeq

\medskip

Now let us respectively    multiply Equation \eqref{defeqfinal} by $a^{-1}  \, r_1$ and Equation \eqref{defeqfinalM} by $a^{-1}  \, {\mathfrak L}^M r_1$, and then    integrate over $\R^4$.  This easily  gives rise to the following identity
\begin{equation}    \label{eqengen} a^{-1} (t)  \int_{\R^4} \Big[r_1 \eqref{defeqfinal} + ({\mathfrak L}^M r_1) \eqref{defeqfinalM}\Big] (t,y)\,  dy =0\, .\end{equation}
Making use of formulae \eqref{defeqfinal} and \eqref{defeqfinalM}, we deduce that \eqref{eqengen} can be splited in several parts as follows:
$$ \begin{aligned}  & (I)+ (II) + (III) + (IV) =  \\
&-a^{-1} (t)  \int_{\R^4} \Big[r_1 \big({\wt \cF}+ {\wt \cR}^ {(N)}\big) + ({\mathfrak L}^M r_1) \big({\wt \cF}_M+ {\mathfrak L}^M{\wt \cR}^ {(N)}\big)\Big] (t,y) \,   dy \, ,\end{aligned} $$
with 
\begin{eqnarray*}
(I)&=& a^{-1} (t)\int_{\R^4} \big(r_2 \,  r_1  + {\mathfrak L}^M r_2  \,  {\mathfrak L}^M r_1 \big)(t,y) \,   dy  \, , \\
(II)&=& a^{-1} (t) \int_{\R^4} \frac  {1+|\nabla Q|^2}{1+|\nabla V|^2} \,\Big[ ({\mathfrak L} r)  \,  r_1  + ( {\mathfrak L}^{M+1} r )\, ({\mathfrak L}^M r_1)\Big] (t,y) \,   dy  \,  \virgp\\
(III)&=& -2  a^{-1} (t)\int_{\R^4} \frac  {V_1}{1+|\nabla V|^2} \,\nabla V \cdot \Big[ (\nabla r_1)  \,  r_1  + (\nabla {\mathfrak L}^{M} r_1 )\, ({\mathfrak L}^M r_1)\Big] (t,y) \,   dy\andf
\\
(IV)&=& a^{-1}(t)\sum^{4}_{i, j=1} \int_{\R^4} g_{i,j} \Big[ \partial_{y_iy_j}^2 r \,  r_1  + (\partial_{y_iy_j}^2 {\mathfrak L}^{M}r)\, ({\mathfrak L}^M r_1)\Big] (t,y) \,   dy \,   ,
\end{eqnarray*}
where for all $1 \leq i, j \leq 4$ the coefficients $g_{i,j}$ in the latter integral are defined by 
\begin{equation}    \label{defg} g_{i,j} = \frac  {V^2_{1}-|\nabla \varepsilon|^2}{1+|\nabla V|^2} \, \delta_{i,j}+ \frac{ (Hr)_{y_i} (Hr)_{y_j}}{1+|\nabla V|^2} \,   \virgp \end{equation}  
and obviously  satisfy the symmetry relations $ g_{i,j} = g_{j,i}$.

\medskip
Firstly, let us   investigate  the  term $(I)$. By definition 
$$ \ds r_2=  a\, (r_1)_t  - a'\, y \cdot  \nabla r_1 - a'  \, \frac  { y \cdot  \nabla H} H  \, r_1  \, \virgp$$
and thus
$$ \ds {\mathfrak L}^M  r_2=  a\, ({\mathfrak L}^M  r_1)_t  - a'\, y \cdot  \nabla ({\mathfrak L}^M r_1) - a'  \, X  \, \virgp$$
with
$$ X= [{\mathfrak L}^M, y \cdot  \nabla] \,  r_1 + {\mathfrak L}^M \, \frac  { y \cdot  \nabla H} H  \, r_1 \,  \cdot$$
We deduce that
$$ \begin{aligned}  & \ds (I)=  \frac12\frac{d}{dt}\Big[ \| r_1 (t)\|^2_{L^2 (\R^4)}+ \| {\mathfrak L}^M r_1 (t)\|^2_{L^2 (\R^4)}\Big]-\frac{1+\nu}t
\int_{\R^4} \Big[   r_1\, y \cdot  \nabla r_1 + ({\mathfrak L}^M r_1) \,  y \cdot  \nabla ({\mathfrak L}^M r_1) \\
&\ds \qquad \qquad \qquad \qquad \qquad  \qquad \qquad \qquad +  r_1  \,  \frac  { y \cdot  \nabla H} H  \, r_1  + \, ({\mathfrak L}^M r_1) \,  X\Big](t, y)  \,dy\, \cdot \end{aligned} $$
Integrating by parts and taking into account that
$$
\|X\|_{L^2(\R^4)}\lesssim\| r_1 \|_{H^{L_0-1} (\R^4)},$$
we find 
\begin{equation}    \label{term1} (I)=  \frac12\frac{d}{dt}\Big[ \| r_1 (t)\|^2_{L^2 (\R^4)}+ \| {\mathfrak L}^M r_1 (t)\|^2_{L^2 (\R^4)}\Big]
+ \frac1t\,\O\big ( \| r_1 (t, \cdot )\|^2_{H^{L_0-1}  (\R^4)}\big) \, , \end{equation} in the sense that (and all along this proof)
$$  \Big|\O \big(\| r_1 (t, \cdot )\|^2_{H^{L_0-1} (\R^4)}\big)  \Big|  \lesssim  \| r_1 (t, \cdot )\|^2_{H^{L_0-1} (\R^4)} \, .$$

\medskip
Let us now estimate the part $(II)$. Firstly, 
let us    point out that it stems from Hardy inequality and the asymptotic  expansion  \eqref {4} that for any function $f$ in $\dot H^1(\R^4)$ the following inequality holds
\begin{equation}    \label{hardyus}     \|\nabla  (q \, f)\|_{L^2(\R^4)} \leq C \|\nabla f \|_{L^2(\R^4)} \, . \end{equation} 
Therefore performing an integration by parts, we get
$$
 \begin{aligned}
& \ds (II)= \int_{\R^4} \nabla \Big(\frac  {1+|\nabla Q|^2}{1+|\nabla V|^2}\Big) \cdot \big[ qr_1   \nabla (qr)  +  q {\mathfrak L}^M r_1   \nabla (q{\mathfrak L}^M r)\big] (t,y)dy \\
& \ds   +  \int_{\R^4} \frac  {1+|\nabla Q|^2}{1+|\nabla V|^2} \big[ \nabla (qr_1) \cdot  \nabla (qr)  + \nabla (q {\mathfrak L}^M r_1)\cdot  \nabla (q{\mathfrak L}^M r)+ \cP (r  r_1+  {\mathfrak L}^M r {\mathfrak L}^M r_1)\big] (t,y)dy   \, 
 \cdot
 \end{aligned}
 $$
A straightforward computation gives
$$ \nabla\Big(\frac  {1+|\nabla Q|^2}{1+|\nabla V|^2}\Big)= \nabla\Big(1+ \frac  {|\nabla Q|^2- |\nabla V|^2}{1+|\nabla V|^2}\Big)= \nabla\Big(\frac  {(\nabla Q-\nabla V)(\nabla Q+\nabla V)}{1+|\nabla V|^2}\Big) \, \cdot$$
We claim that there is a positive constant~$C$ such that the following estimate holds for any  time~$t$ in~$[t_1,t_2]$, with
~$0<  t_1 \leq  t_2 \leq T$:
\begin{equation}    \label{term2*}     \Big\|\nabla\Big(\frac  {1+|\nabla Q|^2}{1+|\nabla V|^2}\Big)(t, \cdot )\Big\|_{L^\infty (\R^4)} \leq C t^\nu \, . \end{equation} 
In order to establish the above estimate,  let us start by observing  that  for any  time $t$ in~$[t_1,t_2]$  we have \footnote{ Here and bellow, we assume that  $N> \nu$.}
$$ \|\nabla^2(V-Q)(t, \cdot )\|_{L^\infty(\R^4)} \leq C \, t^\nu \, .$$
Indeed by definition $$V= V^ {(N)}+ \varepsilon  \,, \with \varepsilon = H\, r \, ,$$ 
which gives the result by applying the triangle inequality and making use   of Lemma \ref{ap3}, Hardy inequality,  the estimates \eqref{use1},  \eqref{H1} and the bootstrap assumption \eqref{eqbootst}.   

\medskip \noindent  Along the same lines, we find that 
 $$ \begin{aligned}
& \ds \qquad \qquad   \| \langle \cdot \rangle^{-1} \nabla(V-Q) (t, \cdot )\|_{L^{\infty}}\leq C t^\nu, \,  \, \| \langle \cdot \rangle  \nabla^2 V (t, \cdot )\|_{L^{\infty}} \leq C \andf  \|  \nabla V (t, \cdot )\|_{L^{\infty}}\leq C  \, ,\end{aligned}$$
which achieves the proof of the claim   \eqref{term2*}. 

\medskip
We deduce that  $$
 \begin{aligned}
& \ds (II)=  \frac1t\O \Big(\| r_1 (t, \cdot )\|^2_{H^{L_0-1} (\R^4)}+ \| \nabla r (t, \cdot )\|^2_{H^{L_0-1} (\R^4)}\Big)  \,  \\
& \ds   +\int_{\R^4} \frac  {1+|\nabla Q|^2}{1+|\nabla V|^2} \big[ \nabla (qr_1) \cdot  \nabla (qr)  + \nabla (q {\mathfrak L}^M r_1)\cdot  \nabla (q{\mathfrak L}^M r)+ \cP (r  r_1+  {\mathfrak L}^M r {\mathfrak L}^M r_1)\big] (t,y)dy  \, 
 \cdot
 \end{aligned}
 $$
Besides, remembering that  $$
\ds r_1= a\, r_t  +  a'\, \Lambda r-   a'  \, \frac  { y \cdot  \nabla H} H  \,  r \,  \virgp$$
we obtain $$\nabla (qr_1)=  a\, \partial_t(\nabla qr)+ a'\,\Lambda \nabla q  r - a'\, \cY_0\,  \virgp$$
with 
$$ \cY_0=  \nabla\, \Big(q \frac  { y \cdot  \nabla H} H  \,  r\Big) - [\nabla q, \Lambda] r \, \cdot$$
Invoking \eqref {4} together with \eqref{H1} and Hardy inequality, we infer that  
\begin{equation}    \label{Y0} \|\cY_0\|_{L^2 (\R^4)}  \lesssim  \| \nabla r\|_{L^{2} (\R^4)}  \,  .\end{equation}    
Along the same lines, we readily gather that 
$$ {\mathfrak L}^M r_1 = a\, \partial_t({\mathfrak L}^M r) +  a'\, \Lambda {\mathfrak L}^M r - a'\, \cY_1\,  \virgp \quad \nabla q {\mathfrak L}^M r_1=  a\, \partial_t(\nabla q {\mathfrak L}^M r) +  a'\, \Lambda \nabla q {\mathfrak L}^M r - a'\, \cY_2\,  \virgp$$
with 
$$ \cY_1=   {\mathfrak L}^M  \frac  { y \cdot  \nabla H} H  \,  r  - [{\mathfrak L}^M, \Lambda] r\, \virgp \quad \cY_2=  - [\nabla q, \Lambda]  {\mathfrak L}^M  r  + \nabla (q  \cY_1)\
 \virgp$$
that clearly satisfy:
\begin{equation}    \label{Y1} \|\cY_1\|_{H^1 (\R^4)}  +\|\cY_2\|_{L^2 (\R^4)} \lesssim  \| \nabla r\|_{H^{L_0-1} (\R^4)}  \,  .\end{equation}    
Taking advantage of \eqref{Y0} and \eqref{Y1}, we infer that 
 $$
 \begin{aligned}
& \ds (II)=  \frac{d}{dt}\cE_1(t)+ (II)_1+(II)_2+ \frac1t\O \Big(\| r_1 (t, \cdot )\|^2_{H^{L_0-1} (\R^4)}+ \| \nabla r (t, \cdot )\|^2_{H^{L_0-1} (\R^4)}\Big)
\,  
 \virgp
 \end{aligned}
 $$
with
 \begin{equation}  \label{en1}\cE_1(t):= \frac 1 2 \int_{\R^4} \frac  {1+|\nabla Q(y)|^2}{1+|\nabla V(t,y)|^2} \Big[ |\nabla (qr)|^2  + |\nabla (q{\mathfrak L}^M r)|^2+ \cP (r^2 +  ({\mathfrak L}^M r)^2)\Big] (t,y)dy \,  , \end{equation}
$$(II)_1:= -\frac 1 2 \int_{\R^4} \partial_t\Big(\frac  {1+|\nabla Q(y)|^2}{1+|\nabla V(t,y)|^2}\Big) \Big[ |\nabla (qr)|^2  + |\nabla (q{\mathfrak L}^M r)|^2+ \cP (r^2 +  ({\mathfrak L}^M r)^2)\Big] (t,y)dy \,  \virgp$$
and 
$$(II)_2:=\frac {1+\nu}t\int_{\R^4}\frac  {1+|\nabla Q|^2}{1+|\nabla V|^2} \big[\nabla (qr) \cdot \Lambda \nabla (qr) + \nabla (q{\mathfrak L}^M r) \cdot \Lambda \nabla (q{\mathfrak L}^M r) \big] (t,y)dy \,  \cdot$$
Again combining the bootstrap assumption \eqref{eqbootst}  with Estimate \eqref{use7},  we claim that for any  time~$t$ in $[t_1,t_2]$, with  $0<  t_1 \leq  t_2 \leq T$
\begin{equation}    \label{term2*sV}\Big\|\partial_t \Big(\frac  {1+|\nabla Q(y)|^2}{1+|\nabla V(t,y)|^2}\Big)(s, \cdot )\Big\|_{L^\infty(\R^4)} \leq C t^{-1} \, \cdot \end{equation}
It is obvious that  \eqref{term2*sV}   reduces to the  following inequality
\begin{equation}    \label{derivs} \| \partial_t \nabla V (t, \cdot )\|_{L^\infty(\R^4)} \leq C t^{-1} \, \cdot\end{equation}
Now to establish  \eqref{derivs}, let us first recall that
 $$V= V^ {(N)}+ \varepsilon   \with \varepsilon = H\, r \, .$$ 
Applying  the triangle inequality and invoking Estimate \eqref{use7}, we deduce that 
$$  \begin{aligned}
& \ds \| \partial_t \nabla V (t, \cdot )\|_{L^\infty(\R^4)} \leq \| \partial_t \nabla V^ {(N)} (t, \cdot )\|_{L^\infty(\R^4)}+ \| \partial_t \nabla (H \, r) (t, \cdot )\|_{L^\infty(\R^4)} \,\\
& \ds  \qquad \qquad  \qquad \qquad  \qquad  \qquad  \qquad \qquad  \qquad\qquad \leq C t^{-1}  + \|\nabla (H \,  \partial_t r) (t, \cdot )\|_{L^\infty(\R^4)}\cdot   \end{aligned}
$$
But in view of \eqref{def12}, we have 
$$\ds a \, r_t = a\, r_1  -  a'\, \Lambda r +  a'  \, \frac  { y \cdot  \nabla H} H  \,  r \, ,$$
which ends the proof of  the result thanks to the bootstrap hypothesis \eqref{eqbootst}.   

\medskip \noindent
Consequently, we get 
\begin{equation} \label{term21}  (II)_1 =    \frac1t\O \Big(\| \nabla r (t, \cdot )\|^2_{H^{L_0-1} (\R^4)}  \Big) \,  
 \cdot \end{equation}  
 To end the estimate of the second part, it remains to investigate the term $(II)_2$.  For that purpose we 
perform an integration by parts, which implies that 
$$ (II)_2=  \frac  {(1+\nu)} 2  \int_{\R^4} \Big((\nabla \cdot y) \frac  {1+|\nabla Q|^2}{1+|\nabla V|^2}\Big) \Big[ |\nabla (qr)|^2  + |\nabla (q{\mathfrak L}^M r)|^2\Big] (t,y)dy  \,   \cdot$$
Taking into account   Lemma \ref {ap3} and the bootstrap assumption \eqref{eqbootst}, this gives  rise to 
\begin{equation} \label{term22}  (II)_2 =    \frac1t\O \Big( \| \nabla r (t, \cdot )\|^2_{H^{L_0-1} (\R^4)}  \Big) \,  
 \cdot \end{equation}  
In summary, we have
\begin{equation}    \label{term2} (II)= \frac{d}{dt}\cE_1(t) + \frac1t\O\Big(\| r_1 (t, \cdot )\|^2_{H^{L_0-1} (\R^4)}+ \| \nabla r (t, \cdot )\|^2_{H^{L_0-1} (\R^4)}\Big)  \,  
 \cdot \end{equation} 
Besides, it stems from the definition of the operator ${\mathfrak L} $  and the estimates \eqref{use0},
\eqref{eqbootst}
 that there is a positive constant $C$ such that 
\begin{equation}    \label{apterm2} \Big|\cE_1(t) - \frac 1 2 \bigl({\mathfrak L}r(t, \cdot ) |r (t, \cdot )\bigr)_{L^2} - \frac 1 2 \bigl({\mathfrak L}^{M+1}r(t, \cdot ) |{\mathfrak L}^M r (t, \cdot )\bigr)_{L^2} \Big| \leq C   \delta^{3\nu} \| \nabla r (t, \cdot )\|^2_{H^{L_0-1}}\,  
 \cdot \end{equation} 

\medskip
Let us now estimate the third term $(III)$.  Integrating again by parts, we easily get 
$$ \begin{aligned}
& (III)= -2  a^{-1} (t) \int_{\R^4} \frac  {V_1}{1+|\nabla V|^2} \,\nabla V \cdot \Big[ (\nabla r_1)  \,  r_1  + (\nabla {\mathfrak L}^{M} r_1 )\, ({\mathfrak L}^M r_1)\Big] (t,y) \,   dy \\ & \ds \qquad  \qquad \qquad = a^{-1} (t) \int_{\R^4}  \partial_{y_j} \Big(\frac  {V_1}{1+|\nabla V|^2} \,\partial_{y_j} V\Big) \Big[ (r_1)^2     + ({\mathfrak L}^{M} r_1)^2\Big] (t,y) \,   dy\, 
 \cdot\end{aligned}$$
 Arguing as above, we infer that for any  time $t$ in $[t_1,t_2]$, with  $0<  t_1 \leq  t_2 \leq T$, we have 
\begin{equation}    \label{term2*s}  \Big\|\nabla \Big(\frac  {V_1}{1+|\nabla V|^2} \,\nabla V \Big) (t, \cdot )\Big\|_{L^\infty} \leq C t^\nu \, \cdot \end{equation}
The latter estimate is a direct consequence of the following inequalities  
$$ \begin{aligned}
& \ds \qquad \qquad  \|\nabla V_1(t, \cdot )\Big\|_{L^\infty} \leq C t^\nu, \qquad  \| \langle \cdot \rangle^{-1} V_1 (t, \cdot )\|_{L^{\infty}}\leq C t^\nu  \andf  \| \langle \cdot \rangle  \nabla^2 V (t, \cdot )\|_{L^{\infty}}   \leq C  \, ,\end{aligned}$$
which readily stem from   the bootstrap assumption \eqref{eqbootst} and  Lemma \ref {ap3}. 

\medskip \noindent
It proceeds to say that
\begin{equation} \label{term3}  (III) =    \frac1t\O \Big( \|  r_1 (t, \cdot )\|^2_{H^{L_0-1} (\R^4)}  \Big) \,  
 \cdot \end{equation}  
 
 \medskip
Finally the last term $(IV)$ can be  dealt with along the same lines as the second term $(II)$. Firstly performing an integration by parts, we get
$$ \begin{aligned}
& (IV)= - \sum^{4}_{i, j=1}a^{-1}(t)\int_{\R^4} g_{i,j} \Big[ (\partial_{y_i} r) \,  (\partial_{y_j} r_1)  + (\partial_{y_i} {\mathfrak L}^{M}r)\, (\partial_{y_j}{\mathfrak L}^M r_1)\Big] (t,y) \,   dy\\ & \ds \qquad  \qquad \qquad  - \sum^{4}_{i, j=1} a^{-1} (t) \int_{\R^4} (\partial_{y_j} g_{i,j}) \Big[ (\partial_{y_i} r) \,   r_1  + (\partial_{y_i} {\mathfrak L}^{M}r)\, ({\mathfrak L}^M r_1)\Big] (t,y) \,   dy\, 
 \virgp\end{aligned}$$
where the coefficients $g_{i,j}$ are defined by \eqref{defg}.

\medskip
 For any  time $t$ in $[t_1,t_2]$, with  $0<  t_1 \leq  t_2 \leq T$, the functions~$g_{i,j}$
 for~$1 \leq i, j \leq 4$ enjoy the following properties
 \begin{equation} \label{term4coef}  \|g_{i,j}(t)\|_{L^\infty(\R^4)} \leq C \delta^{6 \nu} \andf \|\nabla g_{i,j}(t )\|_{L^\infty(\R^4)} \leq Ct^{\nu}  
 \, .\end{equation} 
 Indeed by definition
 $$ g_{i,j} = \frac  {V^2_{1}-|\nabla \varepsilon|^2}{1+|\nabla V|^2} \, \delta_{i,j}+ \frac{ (Hr)_{y_i} (Hr)_{y_j}}{1+|\nabla V|^2} \, , $$  which leads to the result thanks to the   following estimates 
 $$ \begin{aligned}  &\Big\|\Big(\frac  {V^2_{1}-|\nabla \varepsilon|^2}{1+|\nabla V|^2}\Big)  (t, \cdot )\Big\|_{L^\infty(\R^4)} \leq C \delta^{6 \nu} \, ,  \quad \Big\|\nabla  \Big(\frac  {V^2_{1}-|\nabla \varepsilon|^2}{1+|\nabla V|^2}\Big)  (t, \cdot )\Big\|_{L^\infty(\R^4)} \leq C t^{\nu}  \, , \\
&\ds \qquad \qquad \qquad \qquad  \Big\|\nabla^ \ell  \Big(\frac{ (Hr)_{y_j} (Hr)_{y_k}}{1+|\nabla V|^2} \Big)  (t, \cdot )\Big\|_{L^\infty(\R^4)} \leq C t^{2N}  \, ,   \, \, \ell=0, 1 \, , \end{aligned} $$
that can be proved  by the same way as \eqref{term2*}, making use 
of  the bootstrap assumption \eqref{eqbootst} and  Lemma \ref {ap3}.  

\medskip \noindent
Now remembering that  $$
\ds r_1= a\, r_t  +  a'\, \Lambda r-   a'  \, \frac  { y \cdot  \nabla H} H  \,  r \,  \virgp$$
we find that
$$\nabla (r_1)=  a\, \partial_t(\nabla r)+ a'\,\Lambda \nabla   r - a'\, \wt {\cY}_0\,  \virgp$$
with 
$$ \wt {\cY}_0=  \nabla\, \Big(\frac  { y \cdot  \nabla H} H  \,  r\Big) - [\nabla, \Lambda] r \, \cdot$$
Along the same lines as for $\cY_0$, we have 
\begin{equation}    \label{tildY0} \|\wt {\cY}_0\|_{L^2 (\R^4)}  \lesssim  \| \nabla r\|_{L^2 (\R^4)}  \,  .\end{equation} 
Similarly, we easily check that  
$$ \nabla  {\mathfrak L}^M r_1=  a\, \partial_t(\nabla  {\mathfrak L}^M r) +  a'\, \Lambda \nabla  {\mathfrak L}^M r - a'\, \wt {\cY}_2\,  \virgp $$
with 
$$\wt {\cY}_2=  - [\nabla, \Lambda]  {\mathfrak L}^M  r  + \nabla (\cY_1)\, ,$$
that clearly satisfies
\begin{equation}    \label{tildY}\|\wt {\cY}_2\|_{L^2 (\R^4)}  \lesssim  \| \nabla   r\|_{H^{L_0-1} (\R^4)} \,  .\end{equation}   
Therefore
$$ \begin{aligned}
& (IV)=  \frac{d}{dt}\cE_2(t)
+ \frac 1 2 \sum^{4}_{i, j=1} \int_{\R^4} (\partial_t g_{i,j} ) \Big[ (\partial_{y_i} r )\,  (\partial_{y_j} r ) + (\partial_{y_i} {\mathfrak L}^{M}r)\, (\partial_{y_j}{\mathfrak L}^M r)\Big] (t,y) \,   dy\, \\ & \ds
\qquad \qquad \qquad   \qquad  - \frac{(\nu+1)}t \sum^{4}_{i, j=1} \int_{\R^4} g_{i,j} \Big[ \partial_{y_i} r \,  \Lambda \partial_{y_j} r + (\partial_{y_i} {\mathfrak L}^{M}r)\, (\Lambda \partial_{y_j}{\mathfrak L}^M r)\Big] (t,y) \,   dy
  \\ & \ds   \qquad \qquad \qquad \qquad  \qquad \qquad \qquad  \qquad \qquad+ \frac1t\O\Big(\| r_1 (t, \cdot )\|^2_{H^{L_0-1} (\R^4)}+ \| \nabla r (t, \cdot )\|^2_{H^{L_0-1} (\R^4)}\Big) \,  ,
 \end{aligned}$$
with
 \begin{equation}  \label{en2}\cE_2(t)=- \frac 1 2 \sum^{4}_{i, j=1}\int_{\R^4} g_{i,j} \Big[ (\partial_{y_i} r) \,  (\partial_{y_j} r ) + (\partial_{y_i} {\mathfrak L}^{M}r)\, (\partial_{y_j}{\mathfrak L}^M r)\Big] (t,y) \,   dy \, \cdot  \end{equation}  
 In view of the bootstrap assumption \eqref{eqbootst} and  Lemma~\ref {ap3},  we have   the following estimates 
 \begin{equation} \label{term4coef*}  \|\langle \cdot \rangle  \nabla g_{i,j}\|_{L^\infty(\R^4)} \leq C 
 \, ,\end{equation} 
 which  follow easily from the fact  that  we  have  for any  time $t$ in $[t_1,t_2]$
$$ \begin{aligned}  &\Big\|\langle \cdot \rangle  \nabla  \Big(\frac  {V^2_{1}-|\nabla \varepsilon|^2}{1+|\nabla V|^2}\Big)  (t, \cdot )\Big\|_{L^\infty(\R^4)} \leq C  \andf \Big\|\langle \cdot \rangle  \nabla  \Big(\frac{ (Hr)_{y_j} (Hr)_{y_k}}{1+|\nabla V|^2} \Big)  (t, \cdot )\Big\|_{L^\infty(\R^4)} \leq C t^{2N}   \, .\end{aligned} $$
An  integration  by parts thus gives rise to 
$$ \begin{aligned}
&  (IV) = \frac{d}{dt}\cE_2(t)
+ \frac 1 2 \sum^{4}_{i, j=1} \int_{\R^4} (\partial_t g_{i,j} ) \Big[ (\partial_{y_i} r )\,  (\partial_{y_j} r ) + (\partial_{y_i} {\mathfrak L}^{M}r)\, (\partial_{y_j}{\mathfrak L}^M r)\Big] (t,y) \,   dy\, \\ & \ds
  \qquad \qquad \qquad \qquad  \qquad \qquad \qquad  \qquad \qquad+ \frac1t\O\Big(\| r_1 (t, \cdot )\|^2_{H^{L_0-1} (\R^4)}+ \| \nabla r (t, \cdot )\|^2_{H^{L_0-1} (\R^4)}\Big) \,  
 \cdot\end{aligned}$$
Now we claim that 
\begin{equation} \label{term4coef**}  \| \partial_t g_{i,j}\|_{L^\infty(\R^4)} \leq C  \, t^{-1}
 \, .\end{equation} 
 The latter estimate  is shown  making use again of the bootstrap assumption~\eqref{eqbootst} and  Lemma~\ref {ap3} which assert that there is a positive constant $C$ such that  for any  time $t$ in $[t_1,t_2]$, we have 
$$ \begin{aligned}  &\Big\|\partial_t  \Big(\frac  {V^2_{1}-|\nabla \varepsilon|^2}{1+|\nabla V|^2}\Big)  (t, \cdot )\Big\|_{L^\infty(\R^4)} \leq C  \,t^{-1} \andf  \Big\|\partial_t   \Big(\frac{ (Hr)_{y_j} (Hr)_{y_k}}{1+|\nabla V|^2} \Big)  (t, \cdot )\Big\|_{L^\infty(\R^4)} \leq C t^{N}   \, \cdot \end{aligned} $$ 
Therefore, we  obtain
\label{estIV0} 
\beq
(IV) =  \frac{d}{dt}\cE_2(t)+\frac1t\O\Big(\| r_1 (t, \cdot )\|^2_{H^{L_0-1} (\R^4)}+ \| \nabla r (t, \cdot )\|^2_{H^{L_0-1} (\R^4)}\Big) \,  
 \cdot
\eeq
Observe also that by
 \eqref{term4coef},  we have
\beq
\label{estIV} 
\big|\cE_2(t)\big|\leq C\delta^{6 \nu} \Big(\| r_1 (t, \cdot )\|^2_{H^{L_0-1} (\R^4)}+ \| \nabla r (t, \cdot )\|^2_{H^{L_0-1} (\R^4)}\Big)\,  \cdot
\eeq

We finally   address the terms $ {\wt \cF}$,  $ {\wt \cF}_M$ and ${\wt \cR}^ {(N)}$. Using the bootstrap assumption \eqref{eqbootst}  and Lemmas ~\ref{ap3}, \ref{solVapl2} it is not difficult to show that they admit the following estimates.

\begin{lemma}
\label {estF}
{\sl There is a positive constant   $C$ such that under 
Assumption \eqref{eqbootst}, 
 the following estimates occur  for any  time $t$ in $[t_1,t_2]$, where  $0<  t_1 \leq  t_2 \leq T$}
\begin{equation}    \label{eqbootstF} \| {\wt \cF}(t, \cdot )\|_{H^{L_0-1}(\R^4)} \leq C t^{\nu} \Big( \| \nabla r (t, \cdot )\|_{H^{L_0-1}(\R^4)} + \| r_1 (t, \cdot )\|_{H^{L_0-1}(\R^4)}\Big)  \, ,  \end{equation}
\begin{equation}    \label{eqbootstR} \| {\wt \cR}^ {(N)}(t, \cdot )\|_{H^{L_0-1}(\R^4)} \leq C \, t^{N+\nu}  \, ,  \end{equation}
\begin{equation}    \label{eqbootstFM} \| {\wt \cF_M}(t, \cdot )\|_{L^2(\R^4)} \leq C\, t^{\nu} \Big( \| \nabla r (t, \cdot )\|_{H^{L_0-1}(\R^4)} + \| r_1 (t, \cdot )\|_{H^{L_0-1}(\R^4)}\Big)  \, \cdot \end{equation}

\end{lemma}

\medbreak  

 We now gather the latter lemma  with the bootstrap hypothesis \eqref{eqbootst} and 
the above estimates~\eqref{term1},   \eqref{apterm2},  \eqref{term3}, \eqref {estIV}, \eqref{eqbootstR} and  \eqref{eqbootstFM}. This yields the following estimate 
\beq \label{*}
\frac{d}{dt}\cE(t)\leq C t^{2N-1}\, ,\eeq
with 
$$\cE(t)=\frac12\Big[ \| r_1 (t)\|^2_{L^2 (\R^4)}+ \| {\mathfrak L}^M r_1 (t)\|^2_{L^2 (\R^4)}\Big]+ \cE_1(t)+
\cE_2(t),$$
and where  $\cE_1$ and $\cE_2$ are respectively given by   \eqref{en1} and \eqref{en2}. 

\medskip 
It follows from \eqref{4}, \eqref{positivity*}, \eqref{apterm2} and \eqref{estIV} that
$$\cE(t)\geq C(\| r_1 (t, \cdot )\|^2_{H^{L_0-1}}+ \| \nabla r (t, \cdot )\|^2_{H^{L_0-1}},$$
for  some positive constant $C$, provided that $\delta$ is taken sufficiently small.  Therefore, integrating inequality
\eqref{*} and taking into account that $r(t_1)=r_1(t_1)=0$, we get
$$\| r_1 (t, \cdot )\|^2_{H^{L_0-1}(\R^4)} + \| \nabla r (t, \cdot )\|^2_{H^{L_0-1}(\R^4)} \leq  \frac C  {N} t^{2N} \, \virgp$$
which  achieves the proof of  Lemma \ref {keybootstrap}. 
\end{proof}

\bigbreak

Since by construction, we have
\begin{eqnarray*}  (u -  u^ {(N)})(t,x) &= & t^ {\nu+1} (V-V^ {(N)})\Big(t, \frac {x} {t^ {\nu+1}} \Big)= t^ {\nu+1}  (H\, r)\Big(t, \frac {x} {t^ {\nu+1}}\Big) \, , \\  \partial_t(u -  u^ {(N)})(t,x) &= &(V_1-V_1^ {(N)})\Big(t, \frac {x} {t^ {\nu+1}} \Big)=  (H\, r_1)\Big(t, \frac {x} {t^ {\nu+1}}\Big) \, \virgp\end{eqnarray*} Proposition \ref {propestap4} follows readily from \eqref{H1} and Lemma \ref {keybootstrap}  by standard continuity arguments.
\end{proof}

\bigbreak
\begin{remark} \label{rem}{\sl
Combining Proposition \ref {propestap4} with  the bounds \eqref{Cond1} and \eqref{Cond2}, we get that
for any time~$t\in [t_1, T]$,
\begin{equation}
\label {Cond11} u(t,\cdot)\geq  \tilde c_0\,  t^{\nu+1} \andf  (1+|\nabla u|^2-(\partial_t u)^2(t,\cdot) \geq  \tilde c_1 \, \virgp \end{equation}
 with some positive constants $\tilde c_0$ and $\tilde c_1$, provided that $N_0$ is sufficiently large.

\medskip
Furthermore, injecting the bounds \eqref{eqbootst}  into \eqref{defeqfinal} and taking into account Lemma \ref{estF}, one easily deduces that  the solution to the Cauchy problem  \eqref {Cauchy1} satisfies
\begin{equation}
\label {u_tt}\|\langle \cdot \rangle^\frac 3 2 \,   \partial_t^2(u -  u^ {(N)})(t,\cdot)\|_{H^{L_0-2} (\R^4)}\leq t^{\frac N 2} \, , \end{equation}
 for all $t\in [t_1, T]$.}
\end{remark}

\medbreak

\subsection{End of the proof} \label{keyest}

We are now in position to finish the proof of Theorem \ref {main}.
Let $(t_n)_{n \in \N}$ be a sequence of positive real numbers in $]0,T]$ converging to $0$, and
consider   the   Cauchy problem~$ {\rm(NW)}_{n,N}$ defined by: 
$$   {\rm(NW)}_{n,N}
\left\{
\begin{array}{l}
 \eqref {geneq} \,u_n=0 \\
{u_n}_{|t=t_n}=u^ {(N)}(t_n, \cdot)\\
(\partial _t u_n)_{|t=t_n}= (\partial _t u^ {(N)})(t_n, \cdot)\,\cdot 
\end{array}
\right.
$$
In view of Proposition \ref {propestap4} and Remark \ref{rem},  we straightforward have  the following uniform result of local well-posedness:
\begin{cor}\label{unifwell} {\sl  There exists an integer $N_0$ such that the Cauchy problem   ${\rm(NW)}_{n,N_0}$  admits a unique solution $u_n$ on~$[t_n,T]$ which satisfies the following energy estimates
\begin{equation}\label{estseq}\|\langle \cdot \rangle^\frac 3 2 \,   \partial_t(u_n -  u^ {(N_0)})(t,\cdot)\|_{H^{L_0-1} (\R^4)}+ \|\langle \cdot \rangle^\frac 3 2 \,  \nabla(u_n -  u^ {(N_0)})(t,\cdot)\|_{H^{L_0-1} (\R^4)} \leq t^{\frac {N_0} 2}   \, , \end{equation}
 for any time $t_n \leq t \leq T$. 

\medskip
Furthermore,
\begin{equation}
\label {Cond111} u_n(t,x)\geq  \tilde c_0\,  t^{\nu+1}, \,\,\, 1+|\nabla u_n(t,x)|^2-(\partial_t u_n(t,x))^2 \geq  \tilde c_1, \quad
\forall \, (t,x)\in [t_n, T]\times\R^4\,.\end{equation}

}

\end{cor}

\medbreak 

By Ascoli theorem, the bounds \eqref{estseq}, \eqref{Cond111}  imply that 
there exists a solution $u$ to the Cauchy problem \eqref {Cauchy}  on $]0,T]$ satisfying
$(u, \partial_t u)  \in  C(]0,T], X_{L_0})$  and such that
after passing to a subsequence, the sequence 
 $((\nabla u_n, \partial_t u_n))_{n \in \N}$  converges to $(\nabla u, \partial_tu)$ in $C([T_1,T], H^{s-1}(\R^4))$
for any~$T_1\in ]0, T]$ and any $s<L_0$. Clearly the solution $u$ satisfies:
\begin{equation*}\|\partial_t(u-  u^ {(N_0)})(t,\cdot)\|_{H^{L_0-1} (\R^4)}+ \|\nabla(u -  u^ {(N_0)})(t,\cdot)\|_{H^{L_0-1} (\R^4)} \leq t^{\frac {N_0} 2} ,\quad \forall t\in ]0, T]  \, , \end{equation*}
\begin{equation*}
u(t,x)\geq  \tilde c_0\,  t^{\nu+1}, \,\,\,
1+|\nabla u(t,x)|^2-(\partial_t u(t,x))^2 \geq  \tilde c_1 \, \quad \forall (t,x)\in ]0, T]\times\R^2.
\end{equation*}
Taking into account Lemma \ref{solVapl2} and Remaks \ref{g}, \ref{rem1}, this concludes the 
proof of Theorem \ref {main}.

%%%%%%%%%%%%%%%%%%%%%%%%%%%%%%%%%%%%%%%%%%%%

%%%%%%%%%%%%%%%%%%%%%%%%%%%%%%%%%%%%%%%%%%%%

\appendix

\section{Derivation of the equation}\label {deerq}
It is well-known that if we consider  in the Minkowski space $\R^ {1,m}$ regular time-like hypersurfaces with vanishing mean curvature which for fixed $t$ are graphs of functions  $\varphi(t,x)$  over $\R^ {m}$, then  $\varphi$ satisfies the following quasilinear wave equation:
$$ \partial_t \Big( \frac {\varphi_t} {\sqrt{1-(\varphi_t)^2+|\nabla \varphi|^2}} \Big)
- \sum^{m}_{j=1}\partial_{x_j} \Big( \frac {\varphi_{x_j}} {\sqrt{1-(\varphi_t)^2+|\nabla \varphi|^2}} \Big)= 0 \, . $$
Our purpose in this appendix is to carry out  the computations for the equation  in the case of  time-like surfaces with vanishing mean curvature  that for fixed $t$ are  parametrized as follows
\begin{equation}\label {param*}\R^n\times \SS^{n-1}\ni (x,\omega)\to(x, u(t,x)\omega)\in \R^{2n}\,  ,
\end{equation} with some positive function $u$.  An elementary computation shows that in that case, the pull-back metric is: 
\begin{equation} \label{metric}
g= - dt^2+ dx^2+u^2 d\omega^2+ du^2\, .
\end{equation}
Recalling the obvious identities
$$ dx^2=  \sum^{n}_{j=1} dx_j^2 \andf du= \sum^{n}_{j=1} \frac {\partial u} {\partial x_j}dx_j + \frac {\partial u} {\partial t} dt \,,$$
we infer that the associated Lagrangian density is  given by 
\begin{equation} \label{volume}
\cL(u, u_t, \nabla u)= u^{n-1}\,\sqrt{1-(u_t)^2+|\nabla u|^2}\, .
\end{equation}
Using that the mean curvature is the first variation of the volume form, we can determine  the equation of motion by considering formally the Euler-Lagrange equation associated to the Lagrangian density $\cL$, which gives rise to
$$ \frac {\partial \cL} {\partial u}- \sum^{n}_{j=1} \frac {\partial } {\partial x_j} \frac{\partial \cL} {\partial u_{x_j}} - \frac {\partial } {\partial t} \frac{\partial \cL} {\partial u_{t}}=0 \, .$$
According to \eqref{volume}, this leads to $$ \begin{aligned} (n-1)\,u^{n-2}\,\sqrt{1-(u_t)^2+|\nabla u|^2}& - \sum^{n}_{j=1} \frac {\partial } {\partial x_j} \frac{u^{n-1} \, u_{x_j}} {\sqrt{1-(u_t)^2+|\nabla u|^2}} \\
&\qquad \qquad \qquad \qquad + \frac {\partial } {\partial t} \frac{u^{n-1} \, u_{t}} {\sqrt{1-(u_t)^2+|\nabla u|^2}} \, \cdot\end{aligned} $$
 Therefore the quasilinear wave  equation at hand undertakes the following form: 
\begin{equation} \label{simgeneq} \partial_t \Big( \frac {u_t} {\sqrt{1-(u_t)^2+|\nabla u|^2}} \Big)- \sum^{n}_{j=1}\partial_{x_j} \Big( \frac {u_{x_j}} {\sqrt{1-(u_t)^2+|\nabla u|^2}} \Big) + \frac {n-1} {u \,\sqrt{1-(u_t)^2+|\nabla u|^2}}=0 \, . \end{equation}
Straightforward computations show that  the above equation \eqref{simgeneq} rewrites as follows:
\begin{equation} \label{eq:gen*}\begin{aligned} u_ {tt}(1+|\nabla u|^2) &-\Delta u \, (1-(u_t)^2+|\nabla u|^2) + \sum^{n}_{j, k=1} u_{x_j} u_{x_k} u_{x_j x_k}\\
&\qquad \qquad \qquad \qquad \qquad   - 2 {u_t} (\nabla u \cdot \nabla u_t)+ \frac {(n-1)} {u}\, (1-(u_t)^2+|\nabla u|^2) =0 \, ,\end{aligned} \end{equation}
which achieves the proof of \eqref {geneq}.

%%%%%%%%%%%%%%%%%%%%%%%%%%%%%%%%%%%%%%%%%%%%

\section{Study of the linearized operator of the quasilinear wave equation around the ground state} \label{end1}
The aim of this section is to investigate  the linearized operator ${\mathcal
L}$ introduced in  \eqref{linQ}.
 To this end, consider under notations \eqref{linQcoef}   the change of function  $$ \ds w(\rho)= H(\rho) f(\rho) \with H= \frac  {(1+ Q^2_{\rho})^{\frac  1 4}} { Q^{\frac  3 2}}  \,\cdot$$ 
By easy computations,   we deduce that 
$$ {\mathcal
L}w =  - H (1+Q^2_\rho) {\mathfrak L} f \with {\mathfrak L} =  - q\, \Delta  \, q +\cP    \,  ,$$
where $\ds q=\frac  1 {(1+Q^2_\rho)^ {\frac 1 2}} $ 
and    $ \ds \cP= \frac {V^ \flat}   {(1+ Q^2_{\rho})} \virgp$ with \beq
 \label {potential}V^ \flat =  \frac   {- 3 \, \big(1+ Q^2_{\rho}\big)} {Q^2} +  \frac  1 2   ({  B}_1){\rho}  -  \frac  1 4   {  B}^2_1 -  \frac   {3} {2}  {  B}_1 \Big(-\frac   {(1+ Q^2_{\rho})} {\rho} + 2 Q_\rho \big(\frac 1 Q- \frac {Q_\rho} {\rho}\big)\Big) \,  \cdot \eeq
In view  of Lemma \ref {ST},   the potential $\cP$ belongs to $\cC_{rad}^\infty(\R^4)$ and satisfies
\beq
\label {infpotential}\cP =  - \frac 3 {8 \rho^2} \big(1 + \circ (1)\big) \, \virgp  \, \, \mbox{as} \, \, \rho \to \infty \,   \cdot \eeq
 The  operator ${\mathfrak L}$    is at the heart of the analysis carried out in this article. The  following lemma summarizes   some of its useful properties.

\medskip
\begin{lemma}
\label {opL}
{\sl Under the above notations, we have 
\begin{itemize}
\item The operator ${\mathfrak L}$ with domain $H^2(\R^4)$ is self-adjoint on $L^2(\R^4)$.
\item There is a positive constant $c$ such that for any function $f$ in $\dot H_{rad}^1(\R^4)$, the following inequality holds:  
\beq
\label {positivity}\bigl({\mathfrak L}f |f\bigr)_{L^2(\R^4)}  \geq c  \, \|\nabla f\|^2_{L^2(\R^4)} \, .\eeq
\end{itemize}
}
\end{lemma}

\medbreak

\begin{remark}  
{\sl  Taking into account \eqref{4}, one easily deduces from \eqref{positivity} that for any integer $m$, there exists a positive constant $c_m$ such that
$$\bigl({\mathfrak L}^{m+1} f |f\bigr)_{L^2(\R^4)} + \bigl({\mathfrak L}f |f\bigr)_{L^2(\R^4)} \geq c_m  \, \|\nabla f\|^2_{H^m(\R^4)} \, ,  \,\, \forall f \in \dot H^1(\R^4) \cap \dot H^{m+1}(\R^4)\, , $$
and $$\bigl({\mathfrak L}^{m+1}f |f\bigr)_{L^2(\R^4)} + \bigl(f |f\bigr)_{L^2(\R^4)} \geq c_m  \, \|  f\|^2_{H^m(\R^4)} \,, \,\,  \forall f \in  H^{m+1}(\R^4)\, .$$}
\end{remark}

\medbreak

\begin{proof} The fact that ${\mathfrak L}$ is self-adjoint on $L^2(\R^4)$ stems easily from  \eqref {redlineartrans}. Consequently  the spectrum of ${\mathfrak L}$ which will be denoted in what follows  by $\sigma({\mathfrak L})$ is real. Since  the potential~$\cP$ is a regular function which behaves as $\ds - \frac 3 {8 \rho^2} $ as $\rho$ tends to infinity, we deduce that~$\sigma({\mathfrak L}) \cap  \R_-^*$ is a discrete set. Besides if $\sigma({\mathfrak L}) \cap  \R_-^* \neq \emptyset$, then it admits   a minimum $\lam_0 < 0$ which is an eigenvalue of~${\mathfrak L}$ and an associated eigenfunction $u_0$ in $\cS(\R^4)$ which is positive. 

\medskip
Recalling that the positive function $\Lambda  Q= Q-\rho \,  Q_\rho$ solves   the homogeneous equation~${\mathcal L} w=0$,  we infer that the function \beq
\label {G}  G:=  \frac  {\Lambda  Q} { H} \with  H= \frac  {(1+ Q^2_{\rho})^{\frac  1 4}} { Q^{\frac  3 2}}  \,\virgp \eeq
defines a regular positive solution to  the homogeneous equation~${\mathfrak L}f=0$.  
We deduce that  
$$ 0= \bigl({\mathfrak L}G |u_0 \bigr)_{L^2(\R^4)} = \bigl(G |{\mathfrak L}u_0 \bigr)_{L^2(\R^4)}= \lam_0 \bigl(G |u_0 \bigr)_{L^2(\R^4)} < 0 \, ,$$ which yields a contradiction. This implies that   $\sigma({\mathfrak L}) \cap  \R_-^* = \emptyset$ and ends the proof of the fact the operator ${\mathfrak L}$ is positive in the sense that for any function $f$ in $H_{rad}^1(\R^4)$, we have
\beq
\label {posLL}\bigl({\mathfrak L}f |f\bigr)_{L^2(\R^4)}  \geq 0  \, . \eeq

\medskip
In order to prove  Inequality  \eqref {positivity}, we shall proceed by contradiction assuming that there is a sequence $(u_n)_{n \in \N}$ in $\dot H_{rad}^1(\R^4)$ satisfying $\|\nabla u_n\|_{L^2(\R^4)}=1$ and 
\begin{equation}
\label{hyp} \bigl({\mathfrak L}u_n|u_n\bigr)_{L^2(\R^4)} \stackrel{n\to\infty}\longrightarrow 0 \, . \end{equation}
Since the sequence $(u_n)_{n \in \N}$ is bounded  in
$\dot H^1_{rad}(\R^4)$, there is a function $u$ in  $\dot H^1_{rad}(\R^4)$ such that, up to  a subsequence extracting  (still denoted by $u_n$ for simplicity) \begin{equation}
\label{weaklim}u_n \stackrel{n\to\infty}\rightharpoonup u  \quad \mbox{in} \quad \dot H^1(\R^4) \, .\end{equation}
We claim that the function $u \neq 0$ and satisfies $ {\mathfrak L}u=0$.  Indeed    by definition, we have 
$${\mathfrak L} u_n= (- q\, \Delta  \, q +\cP) u_n  \,  \virgp$$
 which, with Notation  \eqref{potential}, gives by integration
$$ \bigl({\mathfrak L}u_n|u_n\bigr)_{L^2(\R^4)} = \int_{\R^4} |\nabla ( q u_n)(x)|^2 \, dx + \int_{\R^4} | (q u_n)(x)|^2 \, V ^ \flat(x) \,dx   \,.$$
Firstly,  let us observe that there is a positive constant $C$  such that, for any integer $n$ we have 
\begin{equation}
\label{estweight} \|\nabla   (q u_n ) \ \|_{L^2(\R^4)}>C  \, \cdot \end{equation}
Indeed, one has  $$\begin{aligned}  \|\nabla u_n\|_{L^2(\R^4)}  & \leq \big\|\frac 1 q \nabla  			(q  u_n) \big\|_{L^2(\R^4)}+   \big\|\nabla\big(\frac 1 q\big)  q u_n \big\|_{L^2(\R^4)}  \,, \end{aligned}$$ 
which in view of Hardy inequality  and  Lemma \ref{ST} leads to \eqref{estweight}.

\medskip  \noindent Secondly,  consider  $\theta$    a smooth radial function valued in $[0,1]$ and satisfying 
$$\left\{
 \begin{aligned}
  &\theta (x) = 0 \,  \, \mbox{for} \, \,  |x| \leq  1 \, \\
  &\ds  \theta (x) = 1 \, \,  \mbox{for} \, \, |x| \geq  2 \,,
 \end{aligned}
\right. $$
and   write 
$$ \bigl({\mathfrak L}u_n|u_n\bigr)_{L^2(\R^4)} = \int_{\R^4}  |\nabla ( q u_n)(x)|^2  dx  - \frac 3 4 \int_{\R^4} | (q u_n)(x)|^2 \, \frac  {\theta (x)} { |x|^ {2}} dx   +  \int_{\R^4} | (q u_n)(x)|^2  \tilde V(x) dx    \, , $$ where of course
$$ \tilde V(x)= V ^ \flat(x) + \frac 3 4  \frac  {\theta (x)} { |x|^ {2}} \, \cdot$$
Invoking  Formula \eqref   {infpotential}, we infer that there is a positive constant $\delta$ such that $\tilde V$ satisfies at infinity
$$ |\tilde V (x) |  \lesssim \frac 1 {\langle x \rangle^{2 + \delta} } \, \cdot$$
 Invoking Rellich theorem and Hardy inequality, we deduce that
\begin{equation} \label{form1} \int_{\R^4}  |  ( q u_n )(x) |^2 \, \tilde V(x) \,dx \stackrel{n\to\infty}\longrightarrow \int_{\R^4}  |  ( q u )(x) |^2 \, \tilde V(x) \,dx \, \cdot\end{equation}
Now for any functions $f$ and $g$ in $\dot H^1(\R^4)$,  denote   by 
$$a(f,g):=\int_{\R^4} \nabla (qf )(x) \cdot \overline {\nabla (qg) (x)}  dx - \frac 3 4 \int_{\R^4}  \frac  {\theta (x)} { |x|^ {2}} \, (qf)(x) \overline {(qg) (x)}  dx\, \cdot$$ Combining  Hardy inequality  with Lemma \ref {ST},  we easily gather  that there exist two positive constants $\alpha_0< \alpha_1$ such that for any function $f$   in $\dot H^1(\R^4)$, we have
$$ \alpha_0  \|\nabla   f  \|^2_{L^2(\R^4)} \leq a(f,f) \leq \alpha_1  \|\nabla   f  \|^2_{L^2(\R^4)} \, , $$
which ensures that $a(f,g)$ is a scalar product  on $\dot H^1(\R^4)$ and that the norms $\sqrt {a(\cdot, \cdot)}$ and  $\|\cdot \|_{\dot H^1(\R^4)}$ are equivalent. 

\medskip 
Since 
$ u_n \stackrel{n\to\infty}\rightharpoonup  u$ in $ \dot H^1(\R^4)$,
we deduce that 
$ \ds a(u,u) \leq \liminf_{n\to\infty} a(u_n, u_n)$, 
and thus 
$$ \begin{aligned}  & \qquad    \bigl({\mathfrak L}u|u\bigr)_{L^2(\R^4)} \leq \liminf_{n\to\infty}  \Big(\bigl({\mathfrak L}u_n|u_n\bigr)_{L^2(\R^4)} \,  .\end{aligned} $$
Taking into account   \eqref {posLL},  \eqref{hyp} and   \eqref{form1}, we deduce that  \begin{equation}    \label{L0} \big({\mathfrak L}u|u\big)_{L^2(\R^4)}=0\, , \end{equation} which according to the fact that   ${\mathfrak L}$ is  positive   implies that  ${\mathfrak L}u=0$.

\medskip
To end the proof of the claim, it remains to establish that $u \neq 0$. For that purpose, let us start by observing that by virtue of  \eqref{hyp},     \eqref{form1} and \eqref{L0}, we have   
\begin{equation}  \begin{aligned} \label{form2}    \int_{\R^4}   |  \nabla( q u_n )(x) |^2   dx &- \frac 3 4 \int_{\R^4} | ( q u_n)(x)|^2 \frac  {\theta (x)} { |x|^ {2}} dx    \\ \qquad \qquad \qquad \qquad \qquad \qquad  \qquad &\stackrel{n\to\infty}\longrightarrow  \int_{\R^4}  |  \nabla( q u )(x) |^2  dx - \frac 3 4 \int_{\R^4}  | ( q u_n)(x)|^2\frac  {\theta (x)} { |x|^ {2}} dx  \,. \end{aligned}\end{equation}
But in view of Hardy inequality and the bound \eqref{estweight}, we have 
\begin{equation} \begin{aligned} \label{Hardy} & \int_{\R^4}  |  \nabla( q u_n )(x) |^2 \, dx - \frac 3 4 \int_{\R^4}  | ( q u_n)(x)|^2 \, \frac  {\theta (x)} { |x|^ {2}} \,dx  \\ & \qquad  \qquad \qquad \qquad  \qquad \qquad  \qquad  \qquad  \qquad   \geq \frac 1 4  \int_{\R^4}  |  \nabla( q u_n )(x) |^2   \, dx  \geq \frac C  4\, \cdot \end{aligned}\end{equation}
By passing to the limit, we obtain
$$\int_{\R^4}  |\nabla ( q u )(x) |^2 \, dx - \frac 3 4 \int_{\R^4}  | (q  {u}  )(x) |^2 \, \frac  {\theta (x)} { |x|^ {2}} \,dx \geq \frac C  4 \, \virgp$$
which achieves  the proof of the fact that $u$ is not null.

\medskip  By construction the function $u$ belongs to $\dot H^1_{rad}(\R^4) $ and satisfies 
$$ {\mathfrak L}u= - q\, \Delta  \, q u +\cP   u =0 \with \ds \cP   =  - \frac 3 {8 \rho^2}  (1 +\circ(1)) \, \virgp  \, \, \mbox{as} \, \, \rho \to \infty \, \cdot$$
Therefore in view  of Hardy inequality,  $\cP   u \in  L^2_{rad}(\R^4) $ and thus $q\, \Delta  \, q u$ belongs to ~$L^2_{rad}(\R^4) $, which ensures that $u \in \dot H^2_{rad}(\R^4)$.

\medskip 
  Now the homogeneous equation ${\mathfrak L}u=0$ admits a   basis of solutions $\ds \big\{f_{1}, f_{2} \big\}$ given by \footnote{ see Appendix \ref {ap:genLres2}   for a proof of this fact.}:
$$ \quad  \left\{
\begin{array}{l}
\ds f_{1}(\rho) = G (\rho)   \andf \\
\ds f_{2} (\rho) =  G(\rho) \int^\rho_1 \frac {(1+ (Q_r (r))^2)^ {\frac 3 2} } {Q^3(r) \, r ^3 \,(\Lambda  Q)^2(r)} \, dr \,\virgp
\end{array}
\right.$$
where $G$ denotes the function defined by \eqref{G}.  
 By  Lemma \ref {ST}, one then has 
$$\quad  \left\{
\begin{array}{l}
 \ds f_1(\rho) \sim 1  \\
\ds  f_2(\rho) \sim \frac 1 {\rho^2}\, \virgp
\end{array}
\right.$$
   near $\rho=0$.
 Since $f_2$ does not belong to $\dot H^1_{rad}(\R^4)$, we deduce  that $u$ is collinear to  $G$.   This yields a contradiction because in view  of \eqref {4}, the function $G$ behaves as $\ds\frac 1 {\sqrt{\rho}}$ when $\rho$  tends to infinity and thus it does not  belong to $\dot H^1_{rad}(\R^4)$. 
This finally completes    the proof of the lemma.
\end{proof}

\bigbreak

 %%%%%%%%%%%%%%%%%%%%%%%%%%%%%%%%%%%%%%%%%%%%

%%%%%%%%%%%%%%%%%%%%%%%%%%%%%%%%%%%%%%%%%%%%
\section{Proof of the local well-posedness result }\label {welCauchy}
The aim of this appendix is to give an outline of the proof of Theorem \ref {Cauchypb}. Since the subject is so well known, we only indicate the main arguments. One can proceed on  three steps:
\begin{enumerate}
\item First,  one proves that for some positive time sufficiently small  $$T=T \big( \|\nabla (u_0-Q)\|_{H^{L-1}}, \|u_1\|_{H^{L-1}}, \inf u_0, \inf \, (1+|\nabla u_0|^2-(u_1)^2)\big),$$
the  Cauchy problem  \eqref {Cauchy} admits a solution $u$ such that $(u,u_t)$ belongs to  the functional space $C([0,T], X_L)$,   $u_t \in C^1([0,T], H^{L-1})$, and which satisfies for all $t$ in $[0,T]$
$$ \|\nabla (u-Q) (t,\cdot)\|_{H^{L-1}} + \|u_t (t,\cdot)\|_{H^{L-1}}  \leq C \big(\|\nabla (u_0-Q)\|_{H^{L-1}} + \|u_1\|_{H^{L-1}}   \big) \, ,$$
for some positive constant $$C=C\big(\|\nabla (u_0-Q)\|_{H^{L-1}},  \|u_1\|_{H^{L-1}},  \inf u_0, \inf \, (1+|\nabla u_0|^2-(u_1)^2)\big).$$ 
\item Second, one shows the uniqueness of solutions thanks to a continuity argument. \smallskip \item Third, one establishs the blow up criterion \eqref{star}.
\end{enumerate}

\bigskip 
 Let us then   consider the  Cauchy problem  \eqref {Cauchy}
and assume that~$\nabla (u_0 - Q)$  and $ u_1$ belong  to~$H^{L-1} (\R^4)$, with $L$ an integer strictly  larger than $4$, and that there is $\varepsilon>0$ such that  
$$u_0 \geq  2  \varepsilon \andf  \frac{1-(u_1)^2+|\nabla u_0|^2}
 {1+|\nabla u_0|^2}
\geq  2  \varepsilon \, .$$
Defining for $1 \leq i, j \leq 4$
\begin{equation} \label{defcoef} \left\{
\begin{array}{l}
\ds a_{ij} \big(\nabla u, u_t \big) =  \delta_{ij} \Big( 1- \frac {(u_t)^2} {(1+|\nabla u|^2)}\Big)-  \frac { u_{x_j} u_{x_i}} {(1+|\nabla u|^2)} \, \virgp  \\
\ds \, \, \, b_{i} \big(\nabla u, u_t \big) =  \frac {2 \, u_{x_i} u_{t}} {(1+|\nabla u|^2)}\, \virgp   \\ \ds c\big(u,\nabla u, u_t\big) = - \frac {3 \,(1-(u_t)^2+|\nabla u|^2)} {u (1+|\nabla u|^2)}\,\virgp \end{array}
\right.
\end{equation}
we readily gather that \eqref {geneq}  takes the form
\begin{equation} \label{genformulation} u_ {tt} - \sum^{4}_{i, j=1} a_{i, j} \big(\nabla u, u_t \big) u_ {x_i x_j}-\sum^{4}_{i=1} b_{i} \big(\nabla u, u_t \big) u_ {tx_i}-c\big(u,\nabla u, u_t\big)=0 \, .\end{equation}

  To prove existence, we shall  use an iterative scheme. To this end,  under the above notations, introduce  the sequence~$(u^{(n)})_{n \in \N}$ defined by $u^{(0)}= Q$  which according to \eqref {eq:ST} satisfies
$$ c\big(Q,\nabla Q, 0\big)+ \sum^{4}_{i, j=1} a_{i, j} \big(\nabla Q, 0 \big) Q_ {x_i x_j}= 0  \, ,$$and 
$$ {\rm(W)}_{n+1}\left\{
\begin{array}{l}
\ds u^{(n+1)}_ {tt} - \sum^{4}_{i, j=1} a_{i, j} \big(\nabla u^{(n)}, u^{(n)}_t \big) u^{(n+1)}_ {x_i x_j}-\sum^{4}_{i=1} b_{i} \big(\nabla u^{(n)}, u^{(n)}_t \big) u^{(n+1)}_ {tx_i}-c\big(u^{(n)},\nabla u^{(n)}, u^{(n)}_t\big)=0 \\
\ds {u^{(n+1)}}_{|t=0}= u_0 \\
\ds (\partial _t u^{(n+1)})_{|t=0}= u_1 \,. 
\end{array}
\right.$$

\medskip
In order to investigate the sequence~$(u^{(n)})_{n \in \N}$ defined above by induction, let us begin by proving  that this sequence of functions is well defined for any time~$t$ in some fixed interval $[0,T]$ which depends only on $\|\nabla (u_0-Q)\|_{H^{L-1}}$, $\|u_1\|_{H^{L-1}}$ and  $\varepsilon$.  This will be deduced from the following result.
\begin{proposition}
\label {keylem}
{\sl Let $u$ be such that $(u,u_t) \in C([0,T], X_L)$, $u_{tt} \in C([0,T], H^{L-2})$ for some integer $L>4$ and some $0<T \leq1$. Assume that 
\beq  
\label {encauchyeq}   \|u_t \|_{L^\infty([0,T],  H^{L-1} )} + \|\nabla (u-Q) \|_{L^\infty([0,T],  H^{L-1} )}   \leq A\, ,   \eeq
\beq  
\label {encauchyeqbis}  \|u_{tt} \|_{L^\infty([0,T],  H^{L-2} )}   \leq A_1\, , \eeq
 \beq \label {condeq} u(t,x) \geq    \varepsilon \andf  \frac{1-(u_t(t,x))^2+|\nabla u(t,x)|^2}
 {1+|\nabla u(t,x)|^2}
\geq    \varepsilon \, ,  \,  \,  \forall (t,x) \in [0,T] \times \R^d \, \cdot \eeq

\smallskip \noindent 
Consider the Cauchy problem\begin{equation} \label{wave}\left\{
\begin{array}{l}
\ds \Phi_ {tt} - \sum^{4}_{i, j=1} a_{i, j} \big(\nabla u, u_t \big) \Phi_ {x_i x_j}-\sum^{4}_{i=1} b_{i} \big(\nabla u, u_t \big) \Phi_ {tx_i}= c\big(u,\nabla u, u_t\big)  \\
\ds {\Phi}_{|t=0}= \Phi_0 \\
\ds (\partial _t \Phi)_{|t=0}= \Phi_1 \,, 
\end{array}
\right.\end{equation}
 assuming that $\nabla (\Phi_0-Q)$ and $\Phi_1$ belong to $H^{L-1} (\R^4)$.  Then the Cauchy problem\refeq{wave} admits a unique solution $\Phi$ on $[0,T]$ and the following energy inequalities hold:
\beq  
\label {encauchyeqsol} \begin{aligned} & \|\Phi_t (t,\cdot)\|_{H^{L-1} } + \|\nabla (\Phi(t,\cdot)-Q)\|_{H^{L-1}}\leq C_ \varepsilon\,  {\rm
e}^{t \, C_ {\varepsilon, A, A_1}}\big(\|\Phi_1\|_{H^{L-1} }  + \|\nabla (\Phi_0-Q)\|_{H^{L-1}} \big)  \\ &  \qquad \qquad \qquad \qquad \qquad \qquad + C_ {\varepsilon, A, A_1} \, \int^t_0 \bigl(\|u_t (s,\cdot)\|_{H^{L-1} } + \|\nabla (u-Q) (s,\cdot)\|_{H^{L-1} }\bigr) ds\, , \end{aligned}\eeq 
 and 
 \beq  
\label {encauchyeqsolbis} \begin{aligned} & \|\Phi_{tt} (t,\cdot)\|_{H^{L-2} }   \leq C_ A\,  \big(\|\Phi_t (t,\cdot)\|_{H^{L-1} } + \|\nabla (\Phi(t,\cdot)-Q)\|_{H^{L-1}} \big)  \\ &  \qquad \qquad \qquad \qquad \qquad \qquad \qquad \qquad+ C_ {\varepsilon, A} \,  \bigl(\|u_t (t,\cdot)\|_{H^{L-1} } + \|\nabla (u-Q) (t,\cdot)\|_{H^{L-1} }\bigr) \, . \end{aligned}\eeq 
}
\end{proposition}

\medbreak

\begin{proof}
Invoking Hypothesis \eqref {condeq}, we easily check  that for any $\ds \xi \in \R^4 \setminus \{0\}$, the  characteristic polynomial of   the wave equation \eqref{wave}
\beq  
\label {chpoly}\tau^2 - \tau  \sum^{4}_{i=1} b_{i} \big(\nabla u, u_t \big) \xi_ {i}- \sum^{4}_{i, j=1} a_{i, j} \big(\nabla u, u_t \big) \xi_i\xi_j \eeq
has two distinct real roots $\tau_1$ and $\tau_2$. Indeed taking account  of \eqref{defcoef}, we find that $\Delta$ the discriminant  of \eqref {chpoly} is given by
$$ \begin{aligned} \Delta & =  \frac {4 (u_t)^2} {(1+|\nabla u|^2)^2}   \, \big(  \sum^{4}_{i=1} u_{x_i} \xi_ {i}\big)^2 +  \frac {4(1-(u_t)^2+|\nabla u|^2)} {(1+|\nabla u|^2)}  \,|\xi|^2-   \frac {4} {(1+|\nabla u|^2)} \big(  \sum^{4}_{i=1} u_{x_i} \xi_ {i}\big)^2 \\
&\qquad \qquad  = \frac {4 (1-(u_t)^2+|\nabla u|^2)} {(1+|\nabla u|^2)^2}  \,|\xi|^2 +\frac {4 (1-(u_t)^2+|\nabla u|^2)} {(1+|\nabla u|^2)^2} \big( |\nabla u|^2 \,|\xi|^2 -  \big(  \sum^{4}_{i=1} u_{x_i} \xi_ {i}\big)^2\big)\, , \end {aligned}$$
which implies that 
\beq  
\label {chypest}  \Delta  \geq  \frac {4 (1-(u_t)^2+|\nabla u|^2)} {(1+|\nabla u|^2)^2}  \,|\xi|^2 \, \cdot\eeq  
  This ends the proof of the claim and ensures   that  \eqref{wave} is  strictly hyperbolic  as long as 
$$(1-(u_t)^2+|\nabla u|^2) >0 \, ,$$ and thus, in view of \eqref {condeq},  on $[0,T] \times \R^4$.

\bigskip Let us emphasize that  under the above notations, the function ${\wt \Phi}:= \Phi - Q$ satisfies :\begin{equation} \label{tildeeq} \left\{
\begin{array}{l}
\ds {\wt \Phi}_ {tt} - \sum^{4}_{i, j=1} a_{i, j} \big(\nabla u, u_t \big) {\wt \Phi}_ {x_i x_j}-\sum^{4}_{i=1} b_{i} \big(\nabla u, u_t \big) {\wt \Phi}_ {tx_i}= f \big(u,\nabla u, u_t\big) \\
\ds {{\wt \Phi}}_{|t=0}= \Phi_0-Q \\
\ds (\partial _t {\wt \Phi})_{|t=0}= \Phi_1 \,, 
\end{array}
\right.\end{equation} 
with
\begin{equation} \label{tildeeqsource}f\big(u,\nabla u, u_t\big)= c\big(u,\nabla u, u_t\big)+ \sum^{4}_{i, j=1} a_{i, j} \big(\nabla u, u_t \big) Q_ {x_i x_j}  \, . \end{equation} 

Firstly  note   that   the  source term~$f$ belongs to the functional space $L^\infty([0,T], H^{L-1} (\R^4))$ and thus to~$L^1([0,T], H^{L-1} (\R^4))$.  Let us start by establishing that $f \in L^\infty([0,T], L^{2} (\R^4))$. Recalling that by virtue  of  \eqref{eq:ST}, we have $$\ds c\big(Q,\nabla Q, 0\big)+ \sum^{4}_{i, j=1} a_{i, j} \big(\nabla Q, 0 \big) Q_ {x_i x_j}=0 \, ,$$ we deduce   that  $f$  rewrites on the following way: 
$$ \begin{aligned} f & = c\big(u,\nabla u, u_t\big)- c\big(Q,\nabla Q, 0\big)+ \sum^{4}_{i, j=1} \big(a_{i, j} \big(\nabla u, u_t \big)- a_{i, j} \big(\nabla Q, 0 \big)\big) Q_ {x_i x_j}= - 3 \, \Big(\frac 1 u - \frac 1 Q \Big)+ \wt f\,\virgp \end{aligned}$$
where 
\begin{equation} \label{defsourctild} \wt f= c\big(u,\nabla u, u_t\big) +\frac 3 u + \sum^{4}_{i, j=1} \big(a_{i, j} \big(\nabla u, u_t \big)- a_{i, j} \big(\nabla Q, 0 \big)\big) Q_ {x_i x_j} \, . \end{equation} 
 Combining Lemma \ref {ST} together with the  hypotheses \eqref {encauchyeq} and \eqref  {condeq}, we obtain making use of Taylor's formula: $$  |\wt f| \leq   C_ {\varepsilon, A} \,  \big( |u_t  |+ |\nabla  (u-  Q)|\big)\, , $$ which easily ensures that for all  $t$ in $[0,T]$, we have
\begin{equation} \label{sourctild} \|\wt f(t,\cdot)\|_{L^2 } \leq   C_ {\varepsilon, A} \,  \big(\|u_t  (t,\cdot)\|_{L^2 }+ \|\nabla  (u-  Q) (t,\cdot)\|_{L^2 }\big)  \,  . \end{equation}
Therefore, we are  reduced to the study of the part  $$\ds - 3 \Big( \frac 1 u - \frac 1 Q \Big)=  3 \,  \frac {u-Q} {u \, Q} \, \cdot$$
We claim that 
\begin{equation} \label{fundest2} \Big| \frac 1 u - \frac 1 Q  \Big|  \leq  C_ {\varepsilon, A} \,  \frac { |u-Q| }  {Q^2}  \, \cdot\end{equation}
Indeed on the one hand according to  Estimate \eqref {encauchyeq},  the function $u-Q$ is bounded on $[0,T] \times \R^4$. Then  writing $$u= Q+ (u-Q) \, , $$ and recalling  that the stationary solution $Q$ behaves as $\rho$ at infinity,  we infer that there is a positive real number $R_0= R_0(A)$ such that for any $|x| \geq R_0$ and  any $t$ in $[0,T]$, we have $$\ds u (t, x) \geq   \frac {Q(x)}  2  \,  \cdot $$
On the other hand, invoking      \eqref {condeq} together with  Lemma \ref {ST}, we infer   that there is a positive constant~$C(\varepsilon, R_0)$ such that if $|x| \leq R_0$, then we have for all $0 \leq  t \leq T$
$$ \frac 1 {u(t,x)} \leq \frac {C(\varepsilon, R_0)}  {Q(x)} \, \cdot$$
 Now taking advantage of  the Sobolev embedding $\ds \dot H^1(\R^4)\hookrightarrow L^{4}(\R^4)$, we deduce that $$ \begin{aligned} \Big\|\Big(\frac 1 u - \frac 1 Q\Big)(t, \cdot) \Big\|_{L^2 } &\leq   C_ {\varepsilon, A} \,  \|(u-Q)(t, \cdot)\|_{L^4 } \Big\| \frac 1 {Q^2} \Big\|_{L^4} \\
&\qquad \qquad \qquad \qquad \qquad \qquad \leq    C_ {\varepsilon, A} \,  \, \|\nabla (u-Q)(t, \cdot) \|_{L^2 } \Big\| \frac 1 {Q^2} \Big\|_{L^4}  \virgp \end {aligned}$$
which according to the fact that $\ds  \frac1 {Q(\rho)}   \lesssim  \frac1 {\langle \rho \rangle}$  ensures that 
\begin{equation} \label{sourc2} \Big\|\Big(\frac 1 u - \frac 1 Q\Big)(t, \cdot) \Big\|_{L^2 }  \leq    C_ {\varepsilon, A} \,  \|\nabla (u-Q)(t, \cdot) \|_{L^2 } \, \cdot \end{equation}
Together with   \eqref{sourctild}, this implies that  for all  $t$ in $[0,T]$ \begin{equation} \label{sourc10}  \| f(t,\cdot)\|_{L^2 } \leq  C_ {\varepsilon, A}  \big( \|u_t  (t,\cdot)\|_{L^2 }+ \|\nabla  (u-  Q) (t,\cdot)\|_{L^2 }\big)\, .\end{equation}
 Thanks to   the bound \eqref {encauchyeq},  this ends the proof of the fact that $f $ belongs to $ L^\infty([0,T], L^2 (\R^4))$.

\medskip

In order to establish that $f \in L^\infty([0,T], H^{L-1} (\R^4))$, let us firstly observe that by virtue of the assumption \eqref {encauchyeq}, the functions  $(b_i\big(\nabla u, u_t \big))_{1 \leq i \leq 4}$,~$\big(a_{i, j} \big(\nabla u, u_t \big)- a_{i, j} \big(\nabla Q, 0 \big)\big)_{1 \leq i,j \leq 4}$  as well as the function  $\ds c\big(u,\nabla u, u_t\big) +\frac 3 u$ belong to~$L^\infty([0,T], H^{L-1} (\R^4))$.  

\medskip Thus taking advantage of Lemma \ref {ST} and recalling that $L>4$, we find  that the function~$\wt f$  belongs to $L^\infty([0,T], H^{L-1} (\R^4))$ and  satisfies the following estimate uniformly on  $[0,T]$:
$$ \|\wt f (t,\cdot)\|_{ H^{L-1} } \leq  C_ {\varepsilon, A} \,    \big( \|u_t (t,\cdot)\|_{H^{L-1} } + \|\nabla (u-Q) (t,\cdot)\|_{H^{L-1} }\big)  \, . $$ 
Besides applying  Leibniz's  formula to  the term $\ds \frac {u-Q} {u \, Q} $ and taking account~\eqref{sourc2},  we  infer that there is a positive constant $C_ {\varepsilon, A}$ such that for all $t$ in $[0,T]$, we have
$$ \Big\|\Big(\frac 1 u - \frac 1 Q\Big) (t,\cdot)\Big\|_{ H^{L-1} } \leq  C_ {\varepsilon, A} \,    \|\nabla (u-Q) (t,\cdot)\|_{H^{L-1} }   \, .$$
Combining the two last inequalities, we get 
\begin{equation} \label{sourcL}  \| f (t,\cdot)\|_{ H^{L-1} } \leq  C_ {\varepsilon, A} \,    \big( \|u_t (t,\cdot)\|_{H^{L-1} } + \|\nabla (u-Q) (t,\cdot)\|_{H^{L-1} }\big)  \, .  \end{equation} This  concludes the proof of the desired result. 

\medskip

Finally,  since  the coefficients of  Equation~\eqref{tildeeq}   as well as their time and spatial derivatives  are bounded on~$[0,T] \times \R^4$, 
 applying  classical arguments, we infer that    the Cauchy problem \eqref{wave} admits  a unique solution  on~$[0,T] \times \R^4$. 

\medskip 
To avoid heaviness, we shall notice all along this proof by $\cA$ the matrix $(a_{i, j})_{1 \leq i, j \leq 4}$ and by $b$ the vector $(b_1, \cdots, b_4)$, and omit on what follows the dependence of all the functions $a_{i, j}$ and $b_{i}$ on~$(\nabla u, u_t)$ and the source term  $f$ on~$(u, \nabla u, u_t)$.

\bigskip Now   to establish the energy inequality \eqref {encauchyeqsol}, we can proceed as follows. First we take the~$L^2$-scalar product of \eqref{tildeeq} with $\ds ({\wt \Phi}_t - \frac b 2 \cdot \nabla {\wt \Phi})$, which gives rise to  $$
\begin{aligned}  
\qquad &
\big ([ \partial_t({\wt \Phi}_t - \frac b 2 \cdot \nabla {\wt \Phi}) - \sum^{4}_{i, j=1} a_{i, j}  {\wt \Phi}_ {x_i x_j} - \frac b 2 \cdot \nabla {\wt \Phi}_t+ \frac {b_t } 2 \cdot \nabla {\wt \Phi}] (t, \cdot) \big| ({\wt \Phi}_t - \frac b 2 \cdot \nabla {\wt \Phi})(t, \cdot)\big)_{L^2}  \\
&\qquad \qquad \qquad \qquad \qquad \qquad \qquad \qquad \qquad  \qquad  = \big (f (t, \cdot)\big|({\wt \Phi}_t - \frac b 2 \cdot \nabla {\wt \Phi})(t, \cdot)\big)_{L^2} \, .\end{aligned}
$$ 
Performing integrations by parts, we deduce that
\begin{equation} \label{enl2}
  \frac12  \frac d {dt} \cE({\wt \Phi}) (t, \cdot) =  I_0(t)+\big (f (t, \cdot)\big|({\wt \Phi}_t - \frac b 2 \cdot \nabla {\wt \Phi})(t, \cdot)\big)_{L^2} \, ,
\end{equation}
where
\begin{equation} \label{endef}
 {\cE}({\wt \Phi}) (t, \cdot):= \|({\wt \Phi}_t - \frac b 2 \cdot \nabla {\wt \Phi}) (t, \cdot)\|_{L^2}^2   +  \sum^{4}_{i, j=1} \big (a_{i, j} {\wt \Phi}_ {x_i}(t, \cdot) \big|{\wt \Phi}_ {x_j} (t, \cdot)\big)_{L^2}+  \|(\frac b 2 \cdot \nabla {\wt \Phi}) (t, \cdot)\|_{L^2}^2  \, , \end{equation}
and where $I_0$ admits the estimate \begin{equation} \label{estI0}|I_0(t)| \leq a_0(t) \, \big(\|(\nabla {\wt \Phi}) (t, \cdot)\|_{L^2}^2+\|{\wt \Phi}_t (t, \cdot)\|_{L^2}^2 \big)\,  ,\end{equation}
with
\begin{equation} \label{gronwal} a_0(t, \cdot)= \cT\big(\|\cA (t, \cdot)\|_{L^\infty}, \| (\nabla_{t,x} \cA) (t, \cdot)\|_{L^\infty}, \|b(t, \cdot)\|_{L^\infty}, \| (\nabla_{t,x} b) (t, \cdot)\|_{L^\infty} \big) \,  ,\end{equation}
 $\cT$ denoting  a polynomial function of all its arguments.

\medskip 
By virtue of estimates\refeq{encauchyeq} and \refeq{encauchyeqbis}, we have \begin{equation} \label{grl2} \| a_0\|_{L^\infty([0,T] \times \R^4)} \leq C_{A, A_1}\, .\end{equation}
Observe also that thanks to\refeq{chypest}, we have   
\begin{equation} \label{estfunten}
\begin{aligned}  
\qquad &
\cE({\wt \Phi}) (t, \cdot) \leq 4 \big(\|(\nabla {\wt \Phi}) (t, \cdot)\|_{H^{L-1}}^2+\|{\wt \Phi}_t (t, \cdot)\|_{H^{L-1}}^2 \big) \andf \\
& \cE({\wt \Phi}) (t, \cdot) \geq \|({\wt \Phi}_t - \frac b 2 \cdot \nabla {\wt \Phi}) (t, \cdot)\|_{L^2}^2 + \varepsilon \, \| \nabla {\wt \Phi} (t, \cdot)\|_{L^2}^2\, .\end{aligned}
\end{equation}
\medskip 

Now in order to investigate  ${ \Phi}_t (t,\cdot)$ and $\nabla ({ \Phi}-Q) (t,\cdot)$ in the setting of  $H^{L-1} $, we   differentiate  the nonlinear wave equation~\eqref{tildeeq} with respect to the variable space up to the order $L-1$.
 By straightforward computations, we obtain formally for any multi-index $\alpha$ of length $ |\al | \leq L-1$
\begin{equation} \label{tildeeqldif}\left\{
\begin{array}{l}
\ds (\partial^\alpha {\wt \Phi})_ {tt} - \sum^{4}_{i, j=1} a_{i, j} (\partial^\alpha {\wt \Phi})_ {x_i x_j}-\sum^{4}_{i=1} b_{i}  (\partial^\alpha {\wt \Phi})_ {tx_i}=  f_\alpha  \\
\ds {(\partial^\alpha {\wt \Phi})}_{|t=0}= \partial^\alpha (\Phi_0-Q) \\
\ds (\partial _t (\partial^\alpha {\wt \Phi}))_{|t=0}= \partial^\alpha \Phi_1 \,, 
\end{array}
\right.\end{equation}
with under the above notations $$ \begin{aligned}
   f_\alpha= \partial^\alpha f + {\wt f}_\alpha\,. 
  \end{aligned}$$
where
\begin{equation} \label{tildsource} {\wt f}_\alpha:= \sum^{4}_{i, j=1}  \sum_{\beta <  \al} \left(\begin{array}{c} \al \\ \beta
 \end{array}\right) \,\big(\partial^{\alpha- \beta} a_{i, j} \big) (\partial^\beta {\wt \Phi})_ {x_i x_j}+ \sum^{4}_{i=1} \sum_{\beta <  \al} \left(\begin{array}{c} \al \\ \beta
 \end{array}\right) \,\big(\partial^{\alpha- \beta} b_{i} \big) (\partial^\beta {\wt \Phi})_ {tx_i}\,.\end{equation}
Then taking the $L^2$-scalar product of Equation \eqref{tildeeqldif} with $\ds (({\partial^\alpha  \wt \Phi})_t - \frac b 2 \cdot \nabla ({\partial^\alpha \wt \Phi}))$, we get applying 
 the same lines of reasoning
 as above  \begin{equation} \label{enestall}
  \frac12  \frac d {dt} {\cE}(\partial^\alpha {\wt \Phi}) (t, \cdot)   =  I_\al(t)+\big ({f}_\alpha  (t, \cdot)\big|[({\partial^\alpha  \wt \Phi})_t - \frac b 2 \cdot \nabla ({\partial^\alpha \wt \Phi})](t, \cdot)\big)_{L^2} \, ,\end{equation}
where 
\begin{equation} \label{sourestal} |I_\al(t)| \leq a_0(t) \,  \big(\|\nabla (\partial^\alpha {\wt \Phi}) (t, \cdot)\|_{L^2}^2+\|(\partial^\alpha {\wt \Phi})_t (t, \cdot)\|_{L^2}^2 \big) \, .\end{equation}
Now since the functions~$\big(\nabla a_{i, j} \big)_{1 \leq i,j \leq 4} $ and $(\nabla b_i)_{1 \leq i \leq 4}$  belong to the Sobolev space $H^{L-2} (\R^4)$, the function  ${\wt f}_\alpha$ belongs to $L^{2} (\R^4)$ and satisfies uniformly on $[0,T]$
\begin{equation} \label{estftildal} \|{\wt f}_\alpha (t, \cdot)\|_{L^2} \leq C_A \big[\|\nabla {\wt \Phi} (t, \cdot)\|_{H^{|\alpha|}} +\|{\wt \Phi}_t (t, \cdot)\|_{H^{|\alpha|}}  \big]  .\end{equation}

\smallskip \noindent
Therefore taking into account \eqref{sourcL},  we get  for any  $ |\al | \leq L-1$  
\beq
\label{sourceal} 
\begin{split}
& \qquad  \big|\big ({f}_\alpha  (t, \cdot)\big|[({\partial^\alpha  \wt \Phi})_t - \frac b 2 \cdot \nabla ({\partial^\alpha \wt \Phi})](t, \cdot)\big)_{L^2}\big| \leq C_ {\varepsilon, A} \,  \big[\|\nabla {\wt \Phi} (t, \cdot)\|^2_{H^{|\alpha|}} +\|{\wt \Phi}_t (t, \cdot)\|^2_{H^{|\alpha|}} \\
& \qquad \qquad  + \big(\|\nabla {\wt \Phi} (t, \cdot)\|_{H^{|\alpha|}} +\|{\wt \Phi}_t (t, \cdot)\|_{H^{|\alpha|}}  \big)  \big(\|u_t (t,\cdot)\|_{H^{L-1} } + \|\nabla (u-Q) (t,\cdot)\|_{H^{L-1}}\big) \big]\, .
\end{split}
\eeq

\smallskip \noindent
Combining \refeq{enl2}, \refeq{estI0}, \refeq{grl2}, \refeq{enestall}, \refeq{sourestal} and \refeq{sourceal}, we easily gather that 
\beq
\label{finenr} 
\begin{split}
& \qquad \Big|\sum_{|\al | \leq L-1}\frac d {dt} {\cE}(\partial^\alpha {\wt \Phi}) (t, \cdot) \Big| \leq C_ {\varepsilon, A, A_1} \,  \big[\|\nabla {\wt \Phi} (t, \cdot)\|^2_{H^{L-1}} +\|{\wt \Phi}_t (t, \cdot)\|^2_{H^{L-1}} \\
& \qquad \qquad  + \big(\|\nabla {\wt \Phi} (t, \cdot)\|_{H^{L-1}} +\|{\wt \Phi}_t (t, \cdot)\|_{H^{L-1}}  \big)  \big(\|u_t (t,\cdot)\|_{H^{L-1} } + \|\nabla (u-Q) (t,\cdot)\|_{H^{L-1}}\big) \big]\, .
\end{split}
\eeq

\smallskip \noindent
Applying  Gronwall lemma and taking into account \refeq{estfunten},   we deduce  for all  $t$ in $[0,T]$ \begin{equation} \label{endproofens} \begin{aligned} & \|\nabla {\wt \Phi} (t, \cdot)\|_{H^{L-1}} +\|{\wt \Phi}_t (t, \cdot)\|_{H^{L-1}}  \leq C_ {\varepsilon} \,    {\rm
e}^{t \, C_ {\varepsilon, A, A_1}}\big(\|\Phi_1\|_{H^{L-1} }  + \|\nabla (\Phi_0-Q)\|_{H^{L-1}} \big)  \\ & \qquad \qquad \qquad \qquad \qquad \qquad + C_ {\varepsilon, A, A_1} \, \int^t_0 \bigl(\|u_s (s,\cdot)\|_{H^{L-1} } + \|\nabla (u-Q) (s,\cdot)\|_{H^{L-1} }\bigr) ds\, .\end{aligned}\end{equation} 
To achieve the proof of the energy estimates, it remains to estimate $\|{\Phi}_{tt} (t,\cdot)\|_{H^{L-2} (\R^4)} $. To this end, we make use of Equation \eqref{tildeeq} which implies that $${\Phi}_ {tt} = \sum^{4}_{i, j=1} a_{i, j} \big(\nabla u, u_t \big) {\wt \Phi}_ {x_i x_j}+\sum^{4}_{i=1} b_{i} \big(\nabla u, u_t \big) {\wt \Phi}_ {tx_i}+ f \big(u,\nabla u, u_t\big)\,.$$
This ensures the result according to\refeq{encauchyeq}  and \eqref{sourcL}.
\end{proof}

\bigbreak

Let us now return  to the proof of Theorem \ref {Cauchypb}. The first step  can be deduced from Proposition~\ref {keylem}  by a standard argument that can be found for instance in the monographs \cite{BCD, hormander, Taylor}.  The key point consists to prove that the  sequence $(u^{(n)})_{n \in \N}$ of  solutions to  the initial value problem
~${\rm(W)}_{n}$ introduced page \pageref{genformulation} is uniformly bounded, in the sense that there  exist a small positive time \begin{equation}   \label{condT}T=T \big(\|u_1\|_{H^{L-1}}, \|\nabla (u_0-Q)\|_{H^{L-1}}, \varepsilon\big),\end{equation} and a  positive constant $C= \big(\|u_1\|_{H^{L-1}}, \|\nabla (u_0-Q)\|_{H^{L-1}}, \varepsilon\big)$  such that  for any integer $n$ and any  time~$t$ in   $[0,T]$, we have 
\begin{equation}  \begin{aligned} \label{unif} & \| u_t^{(n)} (t,\cdot)\|_{H^{L-1}} + \|\nabla (u^{(n)}-Q) (t,\cdot)\|_{H^{L-1} } + \|u^{(n)}_{tt} (t,\cdot)\|_{H^{L-2} }  \leq C   \, ,  \end{aligned}\end{equation} 
and \begin{equation} \label{eqepsn} u^{(n)}(t,\cdot) \geq   \varepsilon \andf (1-(u^{(n)}_t)^2+|\nabla u^{(n)}|^2)(t,\cdot) \geq \varepsilon \, . \end{equation}

\medskip  In order to establish the uniform estimate \eqref{unif}, set
\begin{equation} \label{consts} A= 2 \, C_\varepsilon \, \big(\|u_1\|_{H^{L-1}}+\|\nabla (u_0-Q)\|_{H^{L-1}}\big)\, ,\end{equation}  
and \begin{equation} \label{consts2} A_1= (C_A+ C_{\varepsilon, A} )\, A \, ,\end{equation}  
where $C_\varepsilon$, $C_A$ and $C_{\varepsilon, A}$ are the constants introduced in \refeq{encauchyeqsol}-\refeq{encauchyeqsolbis}.

\medskip  We claim that there exists a positive    time $T \leq 1$ under the form\refeq{condT} such that for any integer $n \geq 0$: 

\smallskip \noindent
if   for any time $t$ in $[0,T]$, we have 
\begin{equation} \label{eqenergyn}\|u^{(n)}_t (t,\cdot)\|_{H^{L-1}} + \|\nabla (u^{(n)}-Q) (t,\cdot)\|_{H^{L-1}} \leq A  \, ,\end{equation}  
\begin{equation} \label{eqenergyn2}\|u^{(n)}_{tt} (t,\cdot)\|_{H^{L-2}}\leq A_1  \, ,\end{equation}  
and 
 \beq \label {condeq} u^{(n)}(t,x) \geq    \varepsilon \, , \,  \,   \frac{1-(u^{(n)}_t(t,x))^2+|\nabla u^{(n)}(t,x)|^2}
 {1+|\nabla u^{(n)}(t,x)|^2}
\geq    \varepsilon \, , \,  \,  \forall (t,x) \in [0,T] \times \R^d \, \virgp \eeq  then the same bounds remain true for $u^{(n+1)}$.

\medskip   \noindent Indeed, by virtue of the energy estimate \refeq{encauchyeqsol} in Proposition \ref{keylem},   $u^{(n+1)}$ verifies for all $t \in [0,T]$
$$ \begin{aligned}  \qquad &  \|u^{(n+1)}_t (t,\cdot)\|_{H^{L-1}} + \|\nabla (u^{(n+1)}-Q) (t,\cdot)\|_{H^{L-1} } \leq C_ \varepsilon\,  {\rm
e}^{t \, C_ {\varepsilon, A, A_1}}\big(\|u_1\|_{H^{L-1} }  + \|\nabla (u_0-Q)\|_{H^{L-1}} \big)    \\ &   \qquad  \qquad  \qquad  \qquad  \qquad  \qquad  \qquad  \qquad  \qquad  \qquad  \qquad  \qquad  \qquad +  A \, C_ {\varepsilon, A, A_1} \, t \, ,\end{aligned} $$
 which implies that there exists $T(A, \varepsilon) >0$  such that if $t \leq T(A, \varepsilon)$, then  $$  \|u^{(n+1)}_t (t,\cdot)\|_{H^{L-1}} + \|\nabla (u^{(n+1)}-Q) (t,\cdot)\|_{H^{L-1} } \leq 2 \, C_ \varepsilon \,\big(\|u_1\|_{H^{L-1}}+\|\nabla (u_0-Q)\|_{H^{L-1}}\big)= A\, .$$
 
 \medskip \noindent
 Invoking then \refeq{encauchyeqsolbis}, we get  for all $t \leq T(A, \varepsilon)$
$$  \|u^{(n+1)}_{tt} (t,\cdot)\|_{H^{L-2} }   \leq  (C_A+ C_{\varepsilon, A} )\, A = A_1\, . $$
This ends the proof of the fact that $u^{(n+1)}$ satisfies the bounds\refeq{eqenergyn} and \refeq{eqenergyn2}.
 
\medskip \noindent Finally, Property \refeq{condeq}  results directly from 
 the following straightforward estimates:
$$ \begin{aligned}
 \| u^{(n+1)} (t,\cdot) - u_0\|_{L^\infty(\R^4)} &   \leq  \int_0^t  \|  \partial_s u^{(n+1)} (s, \cdot)  \|_{L^\infty(\R^4)} \, ds \lesssim A \,t \,, \\
  \|  (\partial_t u^{(n+1)}) (t, \cdot) - u_1\|_{L^\infty(\R^4)} &   \leq  \int_0^t  \|  \partial^2_s u^{(n+1)}(s, \cdot)  \|_{L^\infty(\R^4)} \, ds \lesssim  A_1 \,t \,,  \\
 \|  (\nabla u^{(n+1)}) (t, \cdot) - \nabla u_0\|_{L^\infty(\R^4)} &   \leq  \int_0^t  \|  (\partial_s \nabla  (u^{(n+1)}) (s, \cdot)  \|_{L^\infty(\R^4)} \, ds \lesssim A \,t \,, \end{aligned}$$
which implies  \refeq{condeq}  provided that $T= T \big(A, A_1, \varepsilon\big)$ is chosen sufficiently small. This achieves the proof of the claim.

\medskip To end the proof of the local well-posedness for the  Cauchy problem 
\eqref {Cauchy}, it suffices to establish that the sequences   $(\partial _t u^{(n)} )_{n \in \N }$  and $(\nabla (u^{(n)}-Q) )_{n \in \N }$ are  Cauchy sequences in the functional space~$L^\infty([0,T], H^{L-2} (\R^4))$. By a standard argument, this fact  follows easily from  \eqref{unif}. Indeed setting $w^{(n+1)}:=u^{(n+1)}-u^{(n)}$, we readily  gather that for all $n \geq 0$
$$  \left\{
\begin{array}{l}
\ds w^{(n+1)}_ {tt} - \sum^{4}_{i, j=1} a_{i, j} \big(\nabla u^{(n)}, u^{(n)}_t \big) w^{(n+1)}_ {x_i x_j}-\sum^{4}_{i=1} b_{i} \big(\nabla u^{(n)}, u^{(n)}_t \big) w^{(n+1)}_ {tx_i} = g^{(n)} \\
\ds {w^{(n+1)}}_{|t=0}= 0 \\
\ds (\partial _t w^{(n+1)})_{|t=0}= 0 \,, 
\end{array}
\right.$$
where$$ \begin{aligned} &   g^{(n)}= \sum^{4}_{i, j=1} \big(a_{i, j} \big(\nabla u^{(n)}, u^{(n)}_t \big)- a_{i, j} \big(\nabla u^{(n-1)}, u^{(n-1)}_t \big)\big) u^{(n)}_ {x_i x_j} \\ &   + \sum^{4}_{i=1} \big( b_{i} \big(\nabla u^{(n)}, u^{(n)}_t \big)- b_{i} \big(\nabla u^{(n-1)}, u^{(n-1)}_t \big) \big)u^{(n)}_ {tx_i}  +c\big(u^{(n )},\nabla u^{(n )}, u^{(n )}_t\big) - c\big(u^{(n-1)},\nabla u^{(n-1)}, u^{(n-1)}_t\big) \, .\end{aligned} $$
Since by construction, we have for any $(t,x)$ in $ [0,T] \times \R^d$ 
$$u^{(n)}(t,x) \geq   \varepsilon \andf \frac{1-(u^{(n)}_t(t,x))^2+|\nabla u^{(n)}(t,x)|^2}
 {1+|\nabla u^{(n)} (t,x)|^2} \geq \varepsilon \, ,$$
we obtain  arguing  as for the proof of Proposition \ref {keylem} 
 $$ \begin{aligned} & \|w^{(n+1)}_t (t,\cdot)\|_{L^\infty([0,T],  H^{L-2} )} + \|\nabla w^{(n+1)} (t,\cdot)\|_{L^\infty([0,T],  H^{L-2} )}\\ & \qquad \qquad \qquad \qquad \qquad \qquad \qquad\leq  C \, T  \Big(\|w^{(n)}_t (t,\cdot)\|_{L^\infty([0,T],  H^{L-2} )} + \|\nabla w^{(n)} (t,\cdot)\|_{L^\infty([0,T],  H^{L-2} )} \Big) \,. \end{aligned} $$
  This  ensures the result provided that $T$ is small enough and completes the proof of the first step. 

\bigskip

Let us now address the second step, and  establish the uniqueness of solutions to the Cauchy problem \eqref {Cauchy}. For that purpose, we shall prove the following continuation criterion which easily gives the result:  
\begin{lemma} \label{cr}{\sl
Consider  $u$ and $v$ two solutions  of the  Cauchy problem  \eqref {Cauchy} respectively associated to the initial data~$(u_0,u_1)$ and $(v_0,v_1)$ in $X_s$, such that $(u, u_t)$ and $(v, v_t)$ are in~$C([0,T], X_s)$ and such that $u_t$ and $v_t$  belong to $C^1([0,T], H^ {s-1})$, with~$s$ a positive real number strictly greater  than~$4$. Then there is a positive constant $C$ such that, for all $t$ in $[0,T]$, the following estimate holds:
$$\|(u-v)_t (t,\cdot)\|_{L^{2} (\R^4)} + \|\nabla (u-v) (t,\cdot)\|_{L^{2} (\R^4)}  \leq C \Big(
\|u_1-v_1\|_{L^{2} (\R^4)} + \|\nabla (u_0-v_0)\|_{L^{2}(\R^4)} \Big)\, .$$
}
\end{lemma}   \begin {proof} By straightforward computations,  we readily gather  that the function $w=:u-v$ solves the following Cauchy problem:\begin{equation} \label{uniqueness} \left\{
\begin{array}{l}
\ds w_ {tt} - \sum^{4}_{i, j=1} a_{i, j} \big(\nabla u, u_t \big) w_ {x_i x_j}-\sum^{4}_{i=1} b_{i} \big(\nabla u, u_t \big) w_ {tx_i} = g \\
\ds {w}_{|t=0}= u_0-v_0 \\
\ds (\partial _t w)_{|t=0}= u_1-v_1 \,, 
\end{array}
\right.\end{equation}
where$$ \begin{aligned}    g & = \sum^{4}_{i, j=1} \big(a_{i, j} \big(\nabla u, u_t \big)- a_{i, j} \big(\nabla v, v_t \big)\big) v_ {x_i x_j}  \\ & \qquad \qquad \qquad \qquad \qquad + \sum^{4}_{i=1} \big( b_{i} \big(\nabla u, u_t \big)- b_{i} \big(\nabla v, v_t \big) \big)v_ {tx_i}  +c\big(u,\nabla u, u_t\big) - c\big(v,\nabla v, v_t\big) \, .\end{aligned} $$
Therefore, taking  the~$L^2$-scalar product of \eqref{uniqueness} with $\ds ({w}_t - \frac b 2 \cdot \nabla {w})$, we get as for the proof of Proposition \ref {keylem} the following energy inequality:
$$\begin{aligned} & \|w_t (t,\cdot)\|_{L^{2} (\R^4)} + \|\nabla w (t,\cdot)\|_{L^{2} (\R^4)}  \\ & \qquad \qquad \qquad \leq C \Big(
\|u_1-v_1\|_{L^{2} (\R^4)} + \|\nabla (u_0-v_0)\|_{L^{2}(\R^4)} +\int^t_0 \|g(s, \cdot)\|_{L^{2} (\R^4)} ds\Big)\, . \end{aligned}$$
As before, we have by straightforward computations 
$$ \|g(s, \cdot)\|_{L^{2} (\R^4)} \leq C  \Big(\|w_t (t,\cdot)\|_{L^{2} (\R^4)} + \|\nabla w (t,\cdot)\|_{L^{2} (\R^4)}  \Big) \, , $$
which easily achieves the proof of the continuation criterion. \end{proof}  

\bigbreak

Finally   the blow up criterion \eqref{star} results by standard arguments from the fact that if $$\limsup_{t \nearrow T} \Big(\Big\|\frac 1 {u (t,\cdot)} \Big\|_{L^{ \infty}} + \Big\|\frac 1 {(1+|\nabla u|^2-(\partial_t u)^2)(t,\cdot)} \Big\|_{L^{ \infty}} +  \sup_{|\gamma| \leq 1 }  \big\| \partial^\gamma_{x} \nabla_{t,x}  u\big\|_{L^{ \infty} }\Big) < \infty \, \virgp$$
then the solution to the Cauchy problem \eqref {Cauchy} can be extended beyond $T$. This ends the proof of Theorem \ref {Cauchypb}. 

%%%%%%%%%%%%%%%%%%
%%%%%%%%%%%%%%%%%%%%%%%%%%%%%%%%%%%%%%%%%%%%

\section{Some simple ordinary differential equations  results }\label {ap:genLres2}  
\subsection{Proof of  Duhamel's formula \eqref{gensol}}\label {ap1}   The formula results from   the following lemma:
\begin{lemma}
\label {genLres}
{\sl Under the above notations, the homogeneous equation   \begin{equation} \label{homL}{\mathcal L} f=0 \end{equation} has a basis of solutions $\ds \big\{{\rm
e}_{1}, {\rm
e}_{ 2} \big\}$ given by:
\begin{equation} \label{basisL}\quad  \left\{
\begin{array}{l}
\ds {\rm
e}_{1}(y) = (\Lambda  Q)(y)  \andf \\
\ds {\rm
e}_{2} (y) =  (\Lambda  Q)(y) \int^y_1 \frac {(1+ (Q_r (r))^2)^ {\frac 3 2} } {Q^3(r) \, r ^3 \,(\Lambda  Q)^2(r)} \, dr \,\cdot
\end{array}
\right.\end{equation} 
Besides for any regular function $g$, the solution to the Cauchy problem  \beq
\label{gencp}  \quad  \left\{
\begin{array}{l}
{\mathcal
L} f= g\\
f(0)= 0\, \andf f'(0)= 0\,,  
\end{array}
\right.
\eeq writes under the following form
$$f (y)=- (\Lambda  Q)(y) \int^y_0 \frac {(1+ (Q_r (r))^2)^ {\frac 3 2} } {Q^3(r) \, r ^3 \,(\Lambda  Q)^2(r)}  \int^r_0 \frac {Q^3(s) \, s ^3 \,(\Lambda  Q)(s)} {(1+ (Q_s (s))^2)^ {\frac 3 2}} \,  g(s) \,  ds \, dr \, \cdot $$
}
\end{lemma}

\medbreak
\begin{proof}
Recall that it was proved page~\pageref{linQ} that   ${\rm
e}_{1}:=\Lambda  Q$ is a positive function on~$\R_+$ which solves  the homogeneous equation ${\mathcal L} f=0$. 

\medskip In order to obtain  a  solution ${\rm
e}_{2}$ to \eqref{homL} linearly independent with ${\rm
e}_{1}$, let us  firstly emphasize that if we denote $f= {\hat H} \, {\hat f}$, where
\beq
\label{condhatH}  2 \frac  {(\hat H)_y}  {\hat H} =  - \Big(\frac 3 y +{B}_1\Big) \, \virgp \eeq
with ${ B}_1$ defined by \eqref{linQcoef}, then 
$$ {\mathcal L} f=g \Leftrightarrow {\hat {\mathcal L}} {\hat f}={\hat g} \, $$
with $g= {\hat H} \, {\hat g}$ and 
$$ {\hat {\mathcal L}} =  \partial^2_{y} + {\hat {\cP}}\, ,$$
where 
$${\hat {\cP}}= B_0 + \Big(\frac 3 y +{B}_1\Big) \frac  {(\hat H)_{y}}  {\hat H}+ \frac  {(\hat H)_{yy}}  {\hat H}\, \cdot$$
Since for any  two solutions ${\hat f}_1$ and ${\hat f}_2$ to the homogeneous equation \begin{equation} \label{homLhat}{\hat {\mathcal L}} {\hat f}=0 \,, \end{equation}  the Wronskian~$W({\hat f}_1, {\hat f}_2)$  is constant, we infer that ${\hat  {\rm
e}}_2$ defined by 
\beq
\label{formule2} {\hat  {\rm
e}}_2(y):= {\hat  {\rm
e}}_1(y) \int^y_1  \frac  {ds}  {{\hat  {\rm
e}}^2_1(s)}  \with {\hat  {\rm
e}}_1= {\hat H}^{-1}  {\rm
e}_1\, , \eeq
constitutes a solution to \eqref{homLhat}  linearly independent with ${\hat  {\rm
e}}_1$. 

\medskip  \noindent  
Since
$$ {   B}_1(y)=     \frac{9 \, Q^2_{y}} {y} -  \frac {6 \, Q_y} {Q} \, \virgp$$
we get   in view of   \eqref {eq:ST}  
$$ {   B}_1(y)= 3  \Big(  \frac {Q_y} {Q} -  \frac {Q_{yy} \, Q_y } {(1+ Q^2_{y})}  \Big)\cdot$$
Therefore  
$$ \frac 3 y +{B}_1(y)= 3 \Big(\log \Big( \frac { y \,Q } {(1+ Q^2_{y})^{\frac  1 2}} \Big)\Big)_y \virgp$$
 and thus taking account \eqref{condhatH}, one can choose
\beq
\label{formulehhar}{\hat H}(y)=  \frac {(1+ (Q_y )^2(y))^ {\frac 3 4} } {(y  \, Q(y))^ {\frac 3 2}}\, \cdot\eeq  
This  achieves the proof of  \eqref{basisL}. 

\medskip To end the proof of the lemma, it remains to establish Duhamel's formula \eqref{gensol}. For that purpose, let us start by noticing that since by construction~$W({\hat  {\rm
e}}_1, {\hat  {\rm
e}}_2)=1$, then ${\hat f}$ the solution to the inhomogeneous equation $ {\hat {\mathcal L}} {\hat f}={\hat g} $ undertakes the following form:
$${\hat f}(y)= \int^y_0 \big({\hat  {\rm
e}}_1(y) {\hat  {\rm
e}}_2(s)- {\hat  {\rm
e}}_1(s) {\hat  {\rm
e}}_2(y)\big) \,  {\hat g}(s) \,  ds \, \virgp$$
which by definition gives rise to 
$$f(y)= \int^y_0 \frac {\big( {\rm
e}_1(y)  {\rm
e}_2(s)-  {\rm
e}_1(s)  {\rm
e}_2(y)\big) } {{\hat H}(s)^ {2}} \,  { g}(s) \,  ds \, \cdot$$
In view of  \eqref{basisL}, we deduce that 
\begin{eqnarray*}
\ds f(y) &=&  {\rm
e}_1(y) \int^y_0 \left( {\rm
e}_1(s) \int^s_1  \frac  {{\hat H}^2(s') \, ds'}  { {\rm
e}^2_1(s')}-  {\rm
e}_1(s) \int^y_1  \frac  {{\hat H}^2(s') \, ds'}  { {\rm
e}^2_1(s')}\right)  \frac { g(s)} {{\hat H}(s)^ {2}} \,   ds \\
 &= & - {\rm
e}_1(y) \int^y_0\frac { {\rm
e}_1(s)\,   g(s)} {{\hat H}(s)^ {2}} \,   \int^y_s  \frac  {{\hat H}^2(s') \, ds'}  { {\rm
e}^2_1(s')} \,   ds\, \cdot
\end{eqnarray*}
Finally performing an integration by parts, we readily gather that 
$$ f(y)= - {\rm
e}_1(y) \int^y_0  \frac  {{\hat H}^2(s)}  { {\rm
e}^2_1(s)} \, \int^s_0\frac { {\rm
e}_1(s')\,   g(s')} {{\hat H}(s')^ {2}} \,  ds' \,   ds \,  \virgp$$
which ends the proof of the lemma by virtue of \eqref{formulehhar}. 
\end{proof}

\bigbreak

 \subsection{Proof of Lemma \ref {genLres2}} \label {ap2}
To prove the first item,  let us for  $g$    in  $ \cC^ {\infty}(\R^*_+)$   look for the solution
~$f $    of   the inhomogeneous equation   
$$ \quad  \left\{
\begin{array}{l}
\ds {\wt {\mathcal
L}}_{k} f= (2z^2-1) \partial^2_{z} f-  (\frac {6} {z} +4 z  \nu k)\partial_{z}f - (\frac {6} {z^2} - 2 \,  \nu k (1 +\nu  k))f= g \, , \\
\ds  f \big(\frac 1 {\sqrt{2}}\big)= 0\, \virgp
\end{array}
\right.$$
 under the   form:
$$ f=f^{(0)} + f^{(1)} \with f^{(0)}(z):= \sum^{ N+ 1}_{m=1} \alpha_{m} \Big(z -\frac 1 {\sqrt{2}}\big)^{m}\, ,$$
where $\ds N:=  [k\nu] + 3$ and where the coefficients $\alpha_{m} $, for  $1\leq m \leq N+1$, are  uniquely determined by the requirement    that the function~$\wt g$ defined by:
$$
 {\wt 
g}:=g-{\wt {\mathcal
L}}_{k} f^{(0)}$$ 
verifies 
\begin{equation}\label{condreg}\ds {\wt 
g}^{(\ell)} \Big(\frac 1 {\sqrt{2}}\Big)= 0, \quad \forall \, \ell \in \{0, \cdots, N\} \,.\end{equation}
Then   $ f^{(1)}$  has to satisfy 
$$ \quad  \left\{
\begin{array}{l}
\ds {\wt {\mathcal
L}}_{k} f^{(1)}= {\wt 
g} \, , \\
\ds  f^{(1)} \big(\frac 1 {\sqrt{2}}\big)= 0 \, \virgp \, \,   f^{(1)} \in \cC^ {\infty}(\R^*_+)\, \virgp\end{array}
\right.$$
and can be recovered by  Duhamel formula:  $$ f^{(1)} (z)= \int^z_{\frac 1 {\sqrt{2}} }   \frac { {\wt 
g}(s)} { 2 s^2- 1}  \, \frac {1} { {\mathcal
W}(f^{ 0, +}_{k,0} 
, f^{ 0, -}_{k,0} )(s)}
\big( f^{ 0, -}_{k,0} (z)\, f^{ 0, +}_{k,0} (s)- f^{ 0, +}_{k,0} (z) f^{ 0, -}_{k,0} (s) \big) \,  ds \virgp$$
  where $$  {\mathcal
W}(f^{ 0, +}_{k,0} 
, f^{ 0, -}_{k,0} ):= f^{ 0, +}_{k,0}(f^{ 0, -}_{k,0})_z-f^{ 0, -}_{k,0} (f^{ 0, +}_{k,0})_z$$  denotes  the Wronskian  of   the basis $\ds \{f^{ 0, +}_{k,0}, f^{ 0, -}_{k,0}\}$ 
defined by \eqref{basisk}.  By straightforward computations,  we have 
$${\mathcal
W}(f^{ 0, +}_{k,0} 
, f^{ 0, -}_{k,0} )(z)= \frac {\sqrt{2} \,  \alpha(\nu, k) \sgn(z- \frac 1 {2}) |z^2- \frac 1 {2}|^{\alpha(\nu, k)-1}} {z^6} \, \virgp$$
which implies that
\begin{equation} \label{it1} f^{(1)}(z)= \frac 1  {2  \,   \sqrt{2} \,  \alpha(\nu, k)} \,   \int_{\frac 1 {\sqrt{2}} } ^z  s^3 
{\wt g}(s) \, \biggl( \frac {f^{ 0, -}_{k,0}(z)} {|s -\frac 1 {\sqrt{2}}|^{\alpha(\nu, k)}}- \frac {f^{ 0,+ }_{k,0}(z)} {(s +\frac 1 {\sqrt{2}})^{\alpha(\nu, k)}}\biggr) \, ds 
\,  \cdot \end{equation}
The uniqueness follows immediately from Remark \ref{remself}. 

\bigskip Now we turn our attention to    the second item. Our task here is to solve uniquely  \eqref {seceq} in the functional space $ \ds \cC^ {\infty}\big(]0,\frac 1 {\sqrt{2}}]\big)$ under Condition \eqref {condgamma}.  Let us  start with the case when $q=0$  and  look for  a solution $f$ to the equation:
$${\wt {\mathcal
L}}_{k} f (z)= \big(\frac 1 {\sqrt{2}} -z\big)^{\gamma}  h(z) \, , $$ under the form:  
$$ f=f^{(0)} + f^{(1)} \, ,$$ with $$f^{(0)}(z):=  \big(\frac 1 {\sqrt{2}} - z\big)^{\gamma+ 1} 
\sum^{N}_{m=0} c_{m}   \big(\frac 1 {\sqrt{2}} -z\big)^{m}\,  ,$$
where again  $\ds N=  [k\nu] + 3$. Due to  \refeq{condgamma}, the coefficients $c_{m}$, for  $ 0 \leq m \leq N$,  can be  fixed  so that   \begin{equation} \label{eqint}\ds{\wt {\mathcal
L}}_{k} f^{(1)}(z)=\big(\frac 1 {\sqrt{2}} -z \big)^{\gamma}    {\wt 
h}  (z) \,,\end{equation}
where  $ {\wt 
h} $ is a function in  $ \ds \cC^ {\infty}\big(]0,\frac 1 {\sqrt{2}}]\big)$ 
 which satisfies 
$$ {\wt h}^{(\ell)} \Big(\frac 1 {\sqrt{2}}\Big)= 0, \quad \forall \, \ell \in \{0, \cdots, N\} \,.$$
But any solution to \eqref{eqint} is under the form $$  \frac 1  {2  \sqrt{2}  \alpha(\nu, k)}    \int_{\frac 1 {\sqrt{2}} } ^z  s^3 
\big(\frac 1 {\sqrt{2}} -s\big)^{ \gamma}  \,   {\wt 
h}  (s)   \biggl( \frac {f^{ 0, -}_{k,0}(z)} {(\frac 1 {\sqrt{2}}-s)^{\alpha(\nu, k)}}- \frac {f^{ 0,+ }_{k,0}(z)} {(s +\frac 1 {\sqrt{2}})^{\alpha(\nu, k)}}\biggr)\, ds + a_k^+  f^{ 0, +}_{k,0}(z)+ a_k^- f^{ 0, -}_{k,0}(z) \virgp$$
for some constants $a_k^+$ and $a_k^-$. Invoking  the fact that we look for solutions in  $ \ds \cC^ {\infty}\big(]0,\frac 1 {\sqrt{2}}]\big)$ vanishing at  $\ds z=\frac 1 {\sqrt{2}} \virgp$ we end up with  the result in the case when  $q=0$ by taking $a_k^+= a_k^-=0$. 

\medskip To establish the result in the general case of any integer $q \geq 1$,   we shall proceed by induction     assuming  that under Condition \eqref {condgamma}, for any integer $1 \leq j \leq q-1$, the inhomogeneous equation $${\wt {\mathcal
L}}_{k} f (z)= \big(\frac 1 {\sqrt{2}} -z\big)^{\gamma}   \big(\log \big( \frac 1 {\sqrt{2}}-z \big) \big)^j  h(z)
$$
admits a unique solution $f $ of the form: 
$$ f(z)= \big(\frac 1 {\sqrt{2}} -z\big)^{\gamma+ 1}  \sum_{0  \leq  \ell  \leq j } \big(\log \big( \frac 1 {\sqrt{2}} -z \big) \big)^ \ell   \,  { h }_ \ell (z)\,  ,$$
where  for all $0  \leq  \ell  \leq j$, ${ h }_ \ell $ is a function in  $ \ds \cC^ {\infty}\big(]0,\frac 1 {\sqrt{2}}]\big)\cdot$ Then we look for  a solution $f$ to 
$${\wt {\mathcal
L}}_{k} f (z)= \big(\frac 1 {\sqrt{2}} -z\big)^{\gamma} \big(\log \big(\frac 1 {\sqrt{2}} -z \big) \big)^q  h(z) \, , $$ under the form:  
\beq \label{decit2}f(z)= \big(\log \big(\frac 1 {\sqrt{2}} -z\big) \big)^q \, {\wt
f} (z)+ f^{(1)}(z)  \, ,\eeq  where
$${\wt {\mathcal
L}}_{k}  {\wt
f}  (z)= \big(\frac 1 {\sqrt{2}} -z\big)^{\gamma}  h(z) \, . $$
Thanks to the above computations, this implies that 
$$  {\wt
f}  (z)=  \big(\frac 1 {\sqrt{2}} -z\big)^{\gamma+1}   
h_q (z) \,,$$
where  $  
h_q $ belongs to   $ \ds \cC^ {\infty}\big(]0,\frac 1 {\sqrt{2}}]\big)\cdot$ 

  \medskip
Since in view of \eqref{decit2}
$${\wt {\mathcal
L}}_{k} f^{(1)}(z) = \big(\frac 1 {\sqrt{2}} -z\big)^{\gamma} \sum_{0  \leq  \ell  \leq q-1} \big(\log \big(\frac 1 {\sqrt{2}} -z \big) \big)^ \ell   \,  { \wt  h }_ \ell (z)\,  ,   $$
with $ \ds { \wt  h }_ \ell  \in \cC^ {\infty}\big(]0,\frac 1 {\sqrt{2}}]\big)\virgp$ this achieves the proof of the second item by  virtue of the induction hypothesis.

\bigskip  Let us now establish the third item. To this end, let  us for    $\ds g  \in  \cC^ {\infty}\big(]0,\frac 1 {\sqrt{2}}[\big)$     admitting  an asymptotic expansion at $0$ of the form: $$g(z)=   (\log z)^{\alpha_0} \,  \sum_{\beta \geq   \beta _0}  \, g _{\beta}   \, z^{\beta -2}  \, ,$$  for some  integers  $\alpha _0$ and $\beta _0$,  investigate  the non homogeneous equation ${\wt {\mathcal
L}}_{k} f  = g$.
Fixing some $z_0$ in~$\ds  ]0,\frac 1 {\sqrt{2}}[$ and invoking Duhamel's formula, we readily gather that  for all $z$   in $\ds ]0,\frac 1 {\sqrt{2}}[\virgp $ we have
$$f(z)= \frac 1  {2    \sqrt{2}    \alpha(\nu, k)}  \int_{z_0} ^z  s^3 
{  g}(s) \biggl( \frac {f^{ 0, -}_{k,0}(z)} {(\frac 1 {\sqrt{2}}-s)^{\alpha(\nu, k)}}- \frac {f^{ 0,+ }_{k,0}(z)} {(s +\frac 1 {\sqrt{2}})^{\alpha(\nu, k)}}\biggr)   ds + a_k^+  f^{ 0, +}_{k,0}(z)+ a_k^- f^{ 0, -}_{k,0}(z)\virgp$$
for some constants $a_k^+$ and $a_k^-$.

 \medskip 

 Taking into account  \eqref{basisk}, we infer that any solution to ${\wt {\mathcal
L}}_{k} f  = g$ admits for $z$ close to $0$ an asymptotic expansion of the form $$ f(z)= \sum_{\beta \geq -3}   f _{0,\beta}    \,  z^\beta+  \sum_{1 \leq \alpha \leq \alpha_0} \sum_{\beta \geq \beta _0 }   f _{\alpha,\beta}   (\log z)^{\alpha}  z^\beta  
   \, ,$$    in the case when $\beta _0 \geq -1$,  and of the type 
   $$\quad \quad \quad f(z)= \sum_{\beta \geq {\rm
\min}(\beta _0, -3)}   f _{0,\beta}    \,  z^\beta+  \sum_{1 \leq \alpha \leq \alpha_0} \sum_{\beta \geq \beta _0 }   f _{\alpha,\beta}   (\log z)^{\alpha}  z^\beta  
  +   (\log z)^{\alpha_0+1} \sum_{\beta \geq \max(\beta _0, -3) }   f _{\alpha,\beta}  z^\beta   \, , $$ 
   in the case when $\beta _0 \leq -2 $. This completes the proof of the third item. 

\bigskip
To end the proof of the lemma, it remains to establish the fourth   item.  Applying  Duhamel's formula, we get  for all $z$  in $ \ds \cC^ {\infty}\big(]\frac 1 {\sqrt{2}} \virgp \infty[\big)$ $$f(z)= - \frac 1  {2    \sqrt{2}    \alpha(\nu, k)}  \int_{z} ^\infty  s^3 
{  g}(s) \biggl( \frac {f^{ 0, -}_{k,0}(z)} {(s -\frac 1 {\sqrt{2}})^{\alpha(\nu, k)}}- \frac {f^{ 0,+ }_{k,0}(z)} {(s +\frac 1 {\sqrt{2}})^{\alpha(\nu, k)}}\biggr)   ds + a_k^+  f^{ 0, +}_{k,0}(z)+ a_k^- f^{ 0, -}_{k,0}(z)\virgp$$
for some constants $a_k^+$ and $a_k^-$. Since $A< \nu k$,   the  unique solution to \eqref {5eq}  which 
admits an asymptotic expansion at infinity under the form:  
$$ f(z)=  \sum_{0 \leq \alpha \leq \alpha_0 } \sum_{p \in  \N}  {\hat f}^k _{\alpha,p}  (\log z)^{\alpha} \, \, z^{A-p} \,     $$
is given by the above formula, with $a_k^+=a_k^-=0$.
This ends the proof of the lemma.

%%%%%%%%%%%%%%%%%%%%%%%%%%%%%%%%%%%%%%%%%%%%

\end{document}